\documentclass[11pt]{amsart}
\usepackage{amsmath,amssymb,amsthm}
\usepackage{booktabs}
\usepackage{longtable}
\usepackage[vcentermath]{youngtab}
\usepackage[all]{xy}
\usepackage{ccaption}
\theoremstyle{plain}
\newtheorem{thm}{Theorem}[section]
\newtheorem{lem}[thm]{Lemma}
\newtheorem{cor}[thm]{Corollary}
\newtheorem{prop}[thm]{Proposition}
\newtheorem{defn}[thm]{Definition}
\newtheorem{fact}[thm]{Fact}
\newtheorem{example}[thm]{Example}
\newtheorem{setting}[thm]{Setting}
\newtheorem{prob}[thm]{Problem}
\newtheorem{rem}[thm]{Remark}

\newcommand{\simrightarrow}{\mathrel{\stackrel{\sim}{\longrightarrow}}}
\newcommand{\bijarrow}{\mathrel{\stackrel{1:1}{\longleftrightarrow}}}
\newcommand{\ad}{\mathop{\mathrm{ad}}\nolimits}
\newcommand{\Ad}{\mathop{\mathrm{Ad}}\nolimits}

\newcommand{\End}{\mathop{\mathrm{End}}\nolimits}

\newcommand{\Map}{\mathop{\mathrm{Map}}\nolimits}

\newcommand{\Ker}{\mathop{\mathrm{Ker}}\nolimits}

\newcommand{\rank}{\mathop{\mathrm{rank}}\nolimits}
\newcommand{\hyp}{\text{hyp}}
\newcommand{\nilp}{\text{nilp}}
\makeatletter
    
    \@addtoreset{equation}{section}
\usepackage{hyperref}
\makeatother
\begin{document}

\title[Semisimple symmetric spaces with proper $SL(2,\mathbb{R})$-actions]{Classification of semisimple symmetric spaces with proper $SL(2,\mathbb{R})$-actions}
\author{Takayuki Okuda}
\address{Graduate School of Mathematical Science, The University of Tokyo,\\ 3-8-1 Komaba, Meguro-ku, Tokyo 153-8914, Japan}
\email{okuda@ms.u-tokyo.ac.jp}
\thanks{This work is supported by Grant-in-Aid for JSPS Fellows}
\date{}

\maketitle

\begin{abstract}
We give a complete classification of irreducible symmetric spaces for which there exist proper $SL(2,\mathbb{R})$-actions as isometries, using the criterion for proper actions by T.~Kobayashi [Math.\ Ann.\ '89] and combinatorial techniques of nilpotent orbits.
In particular, we classify irreducible symmetric spaces that admit surface groups as discontinuous groups, combining this with Benoist's theorem [Ann.\ Math.\ '96].
\end{abstract}

\section{Introduction}

The aim of this paper is to classify semisimple symmetric spaces $G/H$ that admit isometric proper actions of non-compact simple Lie group $SL(2,\mathbb{R})$, 
and also those of surface groups $\pi_1(\Sigma_g)$.
Here, isometries are considered with respect to the natural pseudo-Riemannian structure on $G/H$.

We motivate our work in one of the fundamental problems on locally symmetric spaces, stated below:

\begin{prob}[See \cite{Kobayashi01}]\label{prob:intro_symm}
Fix a simply connected symmetric space $\widetilde{M}$ as a model space.
What discrete groups can arise as the fundamental groups of complete affine manifolds $M$ which are locally isomorphic to the space $\widetilde{M}$? 
\end{prob}

By a theorem of \'E.\,Cartan, such $M$ is represented as the double coset space $\Gamma \backslash G/H$. 
Here $\widetilde{M} = G/H$ is a simply connected symmetric space and $\Gamma \simeq \pi_1(M)$ a discrete subgroup of $G$ acting properly discontinuously and freely on $\widetilde{M}$.

Conversely, for a given symmetric pair $(G,H)$ and an abstract group $\Gamma$ with discrete topology, if there exists a group homomorphism $\rho: \Gamma \rightarrow G$ for which $\Gamma$ acts on $G/H$ properly discontinuously and freely via $\rho$, 
then the double coset space $\rho(\Gamma) \backslash G/H$ becomes a $C^\infty$-manifold such that the natural quotient map 
\[
G/H \rightarrow \rho(\Gamma) \backslash G/H
\]
is a $C^\infty$-covering.
The double coset manifold $\rho(\Gamma) \backslash G/H$ is called a Clifford--Klein form of $G/H$, which is endowed with a locally symmetric structure through the covering.
We say that \textit{$G/H$ admits $\Gamma$ as a discontinuous group} if there exists such $\rho$.

Then Problem \ref{prob:intro_symm} may be reformalized as:
\begin{prob}\label{prob:discontinuous_group}
Fix a symmetric pair $(G,H)$.
What discrete groups does $G/H$ admit as discontinuous groups?
\end{prob}

For a compact subgroup $H$ of $G$, 
the action of any discrete subgroup of $G$ on $G/H$ is automatically properly discontinuous. 
Thus our interest is in non-compact $H$, for which not all discrete subgroups $\Gamma$ of $G$ act properly discontinuously on $G/H$.
Problem \ref{prob:discontinuous_group} is non-trivial, 
even when $\widetilde{M} = \mathbb{R}^n$ regarded as an affine symmetric space, i.e.~$(G,H) = (GL(n,\mathbb{R}) \ltimes \mathbb{R}^n, GL(n,\mathbb{R}))$.
In this case, the long-standing conjecture (Auslander's conjecture) states that such discrete group $\Gamma$ will be virtually polycyclic if the Clifford--Klein form $M$ is compact (see \cite{Abels-Margulis-Soifer11,Auslander64,Goldman-Kamishima84,Tomanov90}). 
On the other hand, as was shown by E.~Calabi and L.~Markus \cite{Calabi-Murkus62} in 1962, 
no infinite discrete subgroup of $SO_0(n+1,1)$ acts properly discontinuously on the de Sitter space $SO_0(n+1,1)/SO_0(n,1)$.
More generally, 
if $G/H$ does not admit any infinite discontinuous group, 
we say that a Calabi--Markus phenomenon occurs for $G/H$. 

For the rest of this paper, we consider the case that $G$ is a linear semisimple Lie group.
In this setting, a systematic study of Problem \ref{prob:discontinuous_group} for the general homogeneous space $G/H$ was initiated in the late 1980s by T.~Kobayashi \cite{Kobayashi89,Kobayashi92b,Kobayashi92a}.
One of the fundamental results of Kobayashi in \cite{Kobayashi89} is a criterion 
for proper actions, including a criterion for the Calabi--Markus phenomenon on homogeneous spaces $G/H$. 
More precisely, he showed that the following four conditions on $G/H$ are equivalent: the space $G/H$ admits an infinite discontinuous group; the space $G/H$ admits a proper $\mathbb{R}$-action; the space $G/H$ admits the abelian group $\mathbb{Z}$ as a discontinuous group; and $\rank_\mathbb{R} \mathfrak{g} > \rank_\mathbb{R} \mathfrak{h}$.
Furthermore, Y.~ Benoist \cite{Benoist96} obtained a criterion for the existence of infinite non-virtually abelian discontinuous groups for $G/H$.

Obviously, such discontinuous groups exist if there exists a Lie group homomorphism  $\Phi : SL(2,\mathbb{R}) \rightarrow G$ such that $SL(2,\mathbb{R})$ acts on $G/H$ properly via $\Phi$. 
We prove that the converse statement also holds when $G/H$ is a semisimple symmetric space.
More strongly, our first main theorem gives a characterization of symmetric spaces $G/H$ that admit proper $SL(2,\mathbb{R})$-actions:

\begin{thm}[see Theorem \ref{thm:main}]\label{thm:intro}
Suppose that $G$ is a connected linear semisimple Lie group.
Then the following five conditions on a symmetric pair $(G,H)$ are equivalent$:$
\begin{enumerate}
\item\label{item:intro_SL2_proper} There exists a Lie group homomorphism $\Phi: SL(2,\mathbb{R}) \rightarrow G$ such that $SL(2,\mathbb{R})$ acts on $G/H$ properly via $\Phi$.
\item \label{item:intro_surface_some} For some $g \geq 2$, the symmetric space $G/H$ admits the surface group $\pi_1(\Sigma_g)$ as a discontinuous group, where $\Sigma_g$ is a closed Riemann surface of genus $g$.
\item \label{item:intro_non-virtually} $G/H$ admits an infinite discontinuous group $\Gamma$ which is not virtually abelian $($i.e. $\Gamma$ has no abelian subgroup of finite index$)$.
\item \label{item:intro_nilp} There exists a complex nilpotent orbit $\mathcal{O}^{G_\mathbb{C}}_\nilp$ in $\mathfrak{g}_\mathbb{C}$ such that $\mathcal{O}^{G_\mathbb{C}}_\nilp \cap \mathfrak{g} \neq \emptyset$ and $\mathcal{O}^{G_\mathbb{C}}_\nilp \cap \mathfrak{g}^c = \emptyset$, 
where $\mathfrak{g}^c$ is the $c$-dual of the symmetric pair 
$(\mathfrak{g},\mathfrak{h})$ 
$($see $\eqref{eq:c-dual}$ for definition$)$.
\item \label{item:intro_anti_hyp} There exists a complex antipodal hyperbolic orbit $\mathcal{O}^{G_\mathbb{C}}_\hyp$ in $\mathfrak{g}_\mathbb{C}$ $($see Definition $\ref{defn:hyp_anti}$$)$ such that $\mathcal{O}^{G_\mathbb{C}}_\hyp \cap \mathfrak{g} \neq \emptyset$ and $\mathcal{O}^{G_\mathbb{C}}_\hyp \cap \mathfrak{g}^c = \emptyset$.
\end{enumerate}
\end{thm}

The implication \eqref{item:intro_SL2_proper} $\Rightarrow$ \eqref{item:intro_surface_some} $\Rightarrow$ \eqref{item:intro_non-virtually} is straightforward and easy.
The non-trivial part of Theorem \ref{thm:intro} is the implication \eqref{item:intro_non-virtually} $\Rightarrow$ \eqref{item:intro_SL2_proper}.

By using Theorem \ref{thm:intro}, we give a complete classification of semisimple symmetric spaces $G/H$ that admit a proper $SL(2,\mathbb{R})$-action.
As is clear for \eqref{item:intro_nilp} or \eqref{item:intro_anti_hyp} in Theorem \ref{thm:intro}, it is sufficient to work at the Lie algebra level.
Recall that the classification of semisimple symmetric pairs $(\mathfrak{g},\mathfrak{h})$ was accomplished by M.~Berger \cite{Berger57}.
Our second main theorem is to single out which symmetric pairs among his list satisfy the equivalent conditions in Theorem \ref{thm:intro}:

\begin{thm}\label{thm:cls_intro}
Suppose $G$ is a simple Lie group.
Then, the two conditions below on a symmetric pair $(G,H)$ are equivalent$:$
\begin{enumerate}
\item $(G,H)$ satisfies one of $($therefore, all of$)$ the equivalent conditions in Theorem $\ref{thm:intro}$.
\item The pair $(\mathfrak{g}, \mathfrak{h})$ belongs to Table $\ref{table:sl2-proper}$ in Appendix \ref{section:appendix}.
\end{enumerate}
\end{thm}

The existence problem for compact Clifford--Klein forms has been actively studied in the last two decades since 
Kobayashi's paper \cite{Kobayashi89}.
The properness 
criteria of Kobayashi and Benoist yield necessary conditions on $(G,H)$ for 
the existence  \cite{Benoist96,Kobayashi89}.
See also \cite{Kobayashi-Yoshino05,Labourie-Zimmer95,Margulis97,Oh-Witte02,Zimmer94} for some other methods for 
the existence problem of compact Clifford--Klein forms.
The recent developments on this topic can be found in \cite{Kobayashi02,Kobayashi09,Labourie96,Margulis00}.

We go back to semisimple symmetric pair $(G,H)$.
By Kobayashi's criterion \cite[Corollary 4.4]{Kobayashi89},
the Calabi--Markus phenomenon occurs for $G/H$ 
if and only if 
$\rank_\mathbb{R} \mathfrak{g} = \rank_\mathbb{R} \mathfrak{h}$ holds.
(see Fact \ref{fact:C-M} for more details).
In particular, $G/H$ does not admit compact Clifford--Klein forms 
in this case unless $G/H$ itself is compact.
In Section \ref{section:main_result}, 
we give the list, as Table \ref{table:non_sl2},
 of symmetric pair $(\mathfrak{g},\mathfrak{h})$ with simple $\mathfrak{g}$ 
which does not appear in Table \ref{table:sl2-proper} and $\rank_\mathbb{R} \mathfrak{g} > \rank_\mathbb{R} \mathfrak{h}$,
i.e.~$(\mathfrak{g},\mathfrak{h})$ does not satisfy 
the equivalent conditions in Theorem \ref{thm:intro} 
with $\rank_\mathbb{R} \mathfrak{g} > \rank_\mathbb{R} \mathfrak{h}$.
Apply a theorem of  Benoist \cite[Corollary 1]{Benoist96},
we see $G/H$ does not admit compact Clifford--Klein forms 
if $(\mathfrak{g},\mathfrak{h})$ is in Table \ref{table:non_sl2}
(see Corollary \ref{cor:no_cpt_CK_form}).
In this table,
we find some ``new'' examples of semisimple symmetric spaces $G/H$
that do not admit compact Clifford--Klein forms,
for which we can not find in the existing literature as follows:

\newpage
\begin{center} 
\begin{longtable}{ll} \toprule
$\mathfrak{g}$ & $\mathfrak{h}$ \\ \midrule
$\mathfrak{su}^*(4m+2)$ & $\mathfrak{sp}(m+1,m)$ \\ \hline
$\mathfrak{su}^*(4m)$ & $\mathfrak{sp}(m,m)$ \\ \hline
$\mathfrak{e}_{6(6)}$ & $\mathfrak{f}_{4(4)}$ \\ \hline 
$\mathfrak{e}_{6(-26)}$ & $\mathfrak{sp}(3,1)$ \\ \hline
$\mathfrak{e}_{6(-26)}$ & $\mathfrak{f}_{4(-20)}$ \\ \hline
$\mathfrak{so}(4m+2,\mathbb{C})$ & $\mathfrak{so}(2m+2,2m)$ \\ \hline
$\mathfrak{e}_{6,\mathbb{C}}$ & $\mathfrak{e}_{6(2)}$ \\
\bottomrule
\caption{Examples of $G/H$ without compact Clifford--Klein forms}
\label{table:new_example_for_non_cocompact}
\end{longtable}
\end{center}

We remark that Table $\ref{table:new_example_for_non_cocompact}$ is the list of symmetric pairs  $(\mathfrak{g},\mathfrak{h})$ in Table $\ref{table:non_sl2}$ which are neither in Benoist's examples \cite[Example~1]{Benoist96} nor in Kobayashi's examples \cite[Example 1.7, Table 4.4]{Kobayashi92a}, \cite[Table 5.18]{Kobayashi97disc}. 

The proof of the non-trivial implication 
\eqref{item:intro_non-virtually} $\Rightarrow$ \eqref{item:intro_SL2_proper} in Theorem \ref{thm:intro} is given by reducing it to an equivalent assertion on complex adjoint orbits, namely, \eqref{item:intro_anti_hyp} $\Rightarrow$ \eqref{item:intro_nilp}.
The last implication is proved by using the Dynkin--Kostant classification of $\mathfrak{sl}_2$-triples (equivalently, complex nilpotent orbits) in $\mathfrak{g}_\mathbb{C}$. 
We note that the proof does not need Berger's classification of semisimple symmetric pairs.

The reduction from 
\eqref{item:intro_non-virtually} $\Rightarrow$ \eqref{item:intro_SL2_proper} to \eqref{item:intro_anti_hyp} $\Rightarrow$ \eqref{item:intro_nilp} in Theorem \ref{thm:intro} is given by proving \eqref{item:intro_SL2_proper} $\Leftrightarrow$ \eqref{item:intro_nilp} and \eqref{item:intro_non-virtually} $\Leftrightarrow$ \eqref{item:intro_anti_hyp} as follows.
We show the equivalence \eqref{item:intro_SL2_proper} $\Leftrightarrow$ \eqref{item:intro_nilp} by combining Kobayashi's properness criterion \cite{Kobayashi89} and a result of J.~Sekiguchi for real nilpotent orbits in \cite{Sekiguchi87} with some observations on complexifications of real hyperbolic orbits. 
The equivalence \eqref{item:intro_non-virtually} $\Leftrightarrow$ \eqref{item:intro_anti_hyp} is obtained from Benoist's criterion \cite{Benoist96}.

As a refinement of the equivalence \eqref{item:intro_SL2_proper} $\Leftrightarrow$ \eqref{item:intro_nilp} 
in Theorem \ref{thm:intro},
we give a bijection between real nilpotent orbits $\mathcal{O}^{G}_\nilp$ in $\mathfrak{g}$ such that the complexifications of $\mathcal{O}^G_\nilp$ do not intersect the another real form $\mathfrak{g}^c$
and Lie group homomorphisms $\Phi : SL(2,\mathbb{R}) \rightarrow G$ for which the $SL(2,\mathbb{R})$-actions on $G/H$ via $\Phi$ are proper, up to inner automorphisms of $G$ (Theorem \ref{thm:sl2_and_real_nilpotent}).

Concerning the proof of Theorem \ref{thm:cls_intro}, 
for a given semisimple symmetric pair $(\mathfrak{g},\mathfrak{h})$,
we give an algorithm to check whether or not the condition \eqref{item:intro_anti_hyp} in Theorem \ref{thm:intro} holds, by using Satake diagrams of $\mathfrak{g}$ and $\mathfrak{g}^c$. 

The paper is organized as follows.
In Section \ref{section:main_result}, we set up notation and state our main theorems.
The next section contains a brief summary of Kobayashi's properness criterion \cite{Kobayashi89} and Benoist's criterion \cite{Benoist96} as preliminary results.
We prove Theorem \ref{thm:intro} in Section \ref{section:outline_of_proof}. 
The proof is based on some theorems, propositions and lemmas which are proved in Section \ref{section:hyp-orbit_and_proper_actions} to Section \ref{section:Symmetric_pair} (see Section \ref{section:outline_of_proof} for more details).
Section \ref{section:classifications} is about the algorithm for our classification.
The last section establishes the relation between proper $SL(2,\mathbb{R})$-actions on $G/H$ and real nilpotent orbits in $\mathfrak{g}$.

The main results of this paper were announced in \cite{Okuda11_proc} with a sketch of the proofs.

\section*{Acknowledgements.}
The author would like to give heartfelt thanks to Prof. Toshiyuki Kobayashi, whose suggestions were of inestimable value for this paper.

\section{Main results}\label{section:main_result}

Throughout this paper, we shall work in the following:
\begin{setting}\label{setting:main}
$G$ is a connected linear semisimple Lie group, 
$\sigma$ is an involutive automorphism on $G$,
and $H$ is an open subgroup of $G^\sigma:=\{\, g \in G \mid \sigma g=g \,\}$.
\end{setting}

This setting implies that $G/H$ carries a pseudo-Riemannian structure $g$ for which $G$ acts as isometries and $G/H$ becomes a symmetric space with respect to the Levi-Civita connection.
We call $(G,H)$ a semisimple symmetric pair.
Note that $g$ is positive definite, namely $(G/H,g)$ is Riemannian, if and only if $H$ is compact. 

Since $G$ is a connected linear Lie group, we can take a connected complexification, denoted by $G_\mathbb{C}$, of $G$.
We write $\mathfrak{g}_\mathbb{C}$, $\mathfrak{g}$ and $\mathfrak{h}$ for Lie algebras of $G_\mathbb{C}$, $G$ and $H$, respectively.
The differential action of $\sigma$ on $\mathfrak{g}$ will be denoted by the same letter $\sigma$.
Then $\mathfrak{h} = \{\, X \in \mathfrak{g} \mid \sigma X = X \,\}$, 
and we also call $(\mathfrak{g},\mathfrak{h})$ a semisimple symmetric pair.
Let us denote by $\mathfrak{q} := \{\, X \in \mathfrak{g} \mid \sigma X =-X \,\}$, 
and write the $c$-dual of $(\mathfrak{g},\mathfrak{h})$ for 
\begin{align}
\mathfrak{g}^c := \mathfrak{h} + \sqrt{-1}\mathfrak{q}. \label{eq:c-dual}
\end{align}
Then both $\mathfrak{g}$ and $\mathfrak{g}^c$ are real forms of $\mathfrak{g}_\mathbb{C}$.
We note that the complex conjugation corresponding to $\mathfrak{g}^c$ on $\mathfrak{g}_\mathbb{C}$ is the anti $\mathbb{C}$-linear extension of $\sigma$ on $\mathfrak{g}_\mathbb{C}$, and the semisimple symmetric pair $(\mathfrak{g}^c,\mathfrak{h})$ is the same as $(\mathfrak{g},\mathfrak{h})^{ada}$ (which coincides with $(\mathfrak{g},\mathfrak{h})^{dad}$; see \cite[Section 1]{Oshima-Sekiguchi84} for the notation).

For an abstract group $\Gamma$ with discrete topology,
we say that \textit{$G/H$ admits $\Gamma$ as a discontinuous group} 
if there exists a group homomorphism $\rho:\Gamma \rightarrow G$ such that $\Gamma$ acts properly discontinuously and freely on $G/H$ via $\rho$ 
(then $\rho$ is injective and $\rho(\Gamma)$ is discrete in $G$, automatically).
For such $\Gamma$-action on $G/H$, 
the double coset space $\Gamma \backslash G/H$, 
which is called a \textit{Clifford--Klein form of $G/H$}, becomes a $C^\infty$-manifold such that 
the quotient map
\[
G/H \rightarrow \rho(\Gamma) \backslash G/H
\]
is a $C^\infty$-covering.
In our context, the freeness of the action is less important than the properness of it (see \cite[Section 5]{Kobayashi89} for more details).

Here is the first main result:

\begin{thm}\label{thm:main}
In Setting $\ref{setting:main}$, the following ten conditions on a semisimple symmetric pair $(G,H)$ are equivalent$:$
\begin{enumerate}
\item \label{item:main_sl2-proper} There exists a Lie group homeomorphism $\Phi : SL(2,\mathbb{R}) \rightarrow G$ such that $SL(2,\mathbb{R})$ acts properly on $G/H$ via $\Phi$.
\item \label{item:main_any_surface_group} For any $g \geq 2$, the symmetric space $G/H$ admits the surface group $\pi_1(\Sigma_g)$ as a discontinuous group, where $\Sigma_g$ is a closed Riemann surface of genus $g$.
\item \label{item:main_some_surface_group} For some $g \geq 2$, the symmetric space $G/H$ admits the surface group $\pi_1(\Sigma_g)$ as a discontinuous group.
\item \label{item:main_non-virtually_abelian} $G/H$ admits an infinite discontinuous group $\Gamma$ which is not virtually abelian $($i.e. $\Gamma$ has no abelian subgroup of finite index$)$.
\item \label{item:main_unipotent_discontinuous} $G/H$ admits a discontinuous group which is a free group generated by a unipotent element in $G$. 
\item\label{item:main_nilpotent} There exists a complex nilpotent adjoint orbit $\mathcal{O}^{G_\mathbb{C}}_\nilp$ of $G_\mathbb{C}$ in $\mathfrak{g}_\mathbb{C}$ such that $\mathcal{O}^{G_\mathbb{C}}_\nilp \cap \mathfrak{g} \neq \emptyset$ and $\mathcal{O}^{G_\mathbb{C}}_\nilp \cap \mathfrak{g}^c = \emptyset$.
\item \label{item:main_realhyperbolic} There exists a real antipodal hyperbolic adjoint orbit $\mathcal{O}^{G}_\hyp$ of $G$ in $\mathfrak{g}$ $($defined below$)$ such that $\mathcal{O}^{G}_\hyp \cap \mathfrak{h} = \emptyset$. 
\item \label{item:main_hyperbolic} There exists a complex antipodal hyperbolic adjoint orbit $\mathcal{O}^{G_\mathbb{C}}_\hyp$ of $G_\mathbb{C}$ in $\mathfrak{g}_\mathbb{C}$ such that $\mathcal{O}^{G_\mathbb{C}}_\hyp \cap \mathfrak{g} \neq \emptyset$ and $\mathcal{O}^{G_\mathbb{C}}_\hyp \cap \mathfrak{g}^c = \emptyset$. 
\item \label{item:main_sl2-triple_R} There exists an $\mathfrak{sl}_2$-triple $(A,X,Y)$ in $\mathfrak{g}$ $($i.e.\ $A,X,Y \in \mathfrak{g}$ with $[A,X] =2X$, $[A,Y] = -2Y$ and $[X,Y] = A$$)$ such that $\mathcal{O}^{G}_A \cap \mathfrak{h} = \emptyset$, where $\mathcal{O}^{G}_A$ is the real adjoint orbit through $A$ of $G$ in $\mathfrak{g}$.
\item \label{item:main_sl2-triple} There exists an $\mathfrak{sl}_2$-triple $(A,X,Y)$ in $\mathfrak{g}_\mathbb{C}$ such that $\mathcal{O}^{G_\mathbb{C}}_A \cap \mathfrak{g} \neq \emptyset$ and $\mathcal{O}^{G_\mathbb{C}}_A \cap \mathfrak{g}^c = \emptyset$, where $\mathcal{O}^{G_\mathbb{C}}_A$ is the complex adjoint orbit through $A$ of $G_\mathbb{C}$ in $\mathfrak{g}_\mathbb{C}$.

\end{enumerate}
\end{thm}

Theorem \ref{thm:intro} is a part of this theorem.

The definitions of \textit{hyperbolic} orbits and \textit{antipodal} orbits are given here:

\begin{defn}\label{defn:hyp_anti}
Let $\mathfrak{g}$ be a complex or real semisimple Lie algebra.
An element $X$ of $\mathfrak{g}$ is said to be hyperbolic if the endomorphism $\ad_\mathfrak{g}(X) \in \End(\mathfrak{g})$ is diagonalizable with only real eigenvalues.
We say that an adjoint orbit $\mathcal{O}$ in $\mathfrak{g}$ is hyperbolic if any $($or some$)$ element in $\mathcal{O}$ is hyperbolic.
Moreover, an adjoint orbit $\mathcal{O}$ in $\mathfrak{g}$ is said to be antipodal if for any $($or some$)$ element $X$ in $\mathcal{O}$, the element $-X$ is also in $\mathcal{O}$.
\end{defn}

A proof of Theorem \ref{thm:main} will be given in Section \ref{section:outline_of_proof}.
Here is a short remark on it. 
In \eqref{item:main_sl2-proper} $\Rightarrow$ \eqref{item:main_sl2-triple_R}, the homomorphism $\Phi$ associates an $\mathfrak{sl}_2$-triples $(A,X,Y)$ by the differential of $\Phi$ (see Section \ref{subsection:proof_of_sl2-proper_triple}).
The complex adjoint orbits in \eqref{item:main_hyperbolic} and \eqref{item:main_sl2-triple} are obtained by the complexification of the real adjoint orbits in \eqref{item:main_realhyperbolic} and \eqref{item:main_sl2-triple_R}, respectively (see Section \ref{subsection:proof_main_complexification}).
In \eqref{item:main_sl2-triple} $\Rightarrow$ \eqref{item:main_nilpotent}, 
the $\mathfrak{sl}_2$-triple $(A,X,Y)$ in \eqref{item:main_sl2-triple} associates a complex nilpotent orbit in \eqref{item:main_nilpotent} by $\mathcal{O}^{G_\mathbb{C}}_X := \Ad(G_\mathbb{C}) \cdot X$ (see Section \ref{subsection:proof_sl2_nilp}).
The implication \eqref{item:main_sl2-proper} $\Rightarrow$ \eqref{item:main_any_surface_group} is obvious if we take $\pi_1(\Sigma_g)$ inside $SL(2,\mathbb{R})$.
The equivalence \eqref{item:main_non-virtually_abelian} $\Leftrightarrow$ \eqref{item:main_realhyperbolic} is a kind of paraphrase of Benoist's criterion \cite[Theorem 1.1]{Benoist96} on symmetric spaces (see Section \ref{subsection:proof_main_benoist}).
The key ingredient of Theorem $\ref{thm:main}$ is the implication \eqref{item:main_some_surface_group} $\Rightarrow$ \eqref{item:main_sl2-proper}.
We will reduce it to the implication \eqref{item:main_hyperbolic} $\Rightarrow$ \eqref{item:main_sl2-triple}.
The condition \eqref{item:main_hyperbolic} will be used for a classification of $(G,H)$ satisfying the equivalence conditions in Theorem \ref{thm:main} $($see Section $\ref{section:classifications}$$)$.

\begin{rem}
\begin{description} 
\item[(1)] K.~Teduka \cite{Teduka08cpx_symm} gave a list of $(G,H)$ satisfying the condition $\eqref{item:main_sl2-proper}$ in Theorem $\ref{thm:main}$ in the special cases 
where $(\mathfrak{g},\mathfrak{h})$ is a complex symmetric pair.
He also studied proper $SL(2,\mathbb{R})$-actions on some non-symmetric spaces in \cite{Teduka08sln}.
\item[(2)] Y.~Benoist \cite[Theorem 1.1]{Benoist96} proved a criterion for the condition \eqref{item:main_non-virtually_abelian} in a more general setting, than we treat here.
\item[(3)] The following condition on a semisimple symmetric pair $(G,H)$ is weaker than the equivalent conditions in Theorem $\ref{thm:main}$$:$
\begin{itemize}
\item There exists a real nilpotent adjoint orbit $\mathcal{O}^{G}_\nilp$ of $G$ in $\mathfrak{g}$ such that $\mathcal{O}^{G}_\nilp \cap \mathfrak{h} = \emptyset$.
\end{itemize}
\end{description}
\end{rem}

For a discrete subgroup $\Gamma$ of $G$, 
we say that a Clifford--Klein form $\Gamma \backslash G/H$ is \textit{standard} if $\Gamma$ is contained in closed reductive subgroup $L$ of $G$ (see Definition \ref{defn:reductive_in}) acting properly on $G/H$ (see \cite{Kassel-Kobayashi11}), and is \textit{nonstandard} if not.
See \cite{Kassel09} for an example of a Zariski-dense discontinuous group $\Gamma$ for $G/H$, which gives a nonstandard Clifford--Klein form.
We obtain the following corollary to the equivalence \eqref{item:main_sl2-proper} $\Leftrightarrow$ \eqref{item:main_some_surface_group} in Theorem \ref{thm:main}.

\begin{cor}\label{cor:standard}
Let $g \geq 2$.
Then, in Setting $\ref{setting:main}$, 
the symmetric space $G/H$ admits the surface group $\pi_1(\Sigma_g)$ as a discontinuous group if and only if there exists a discrete subgroup $\Gamma$ of $G$ such that $\Gamma \simeq \pi_1(\Sigma_g)$ and $\Gamma \backslash G/H$ is standard.
\end{cor}

Theorem \ref{thm:main} may be compared with the fact below for proper actions by the abelian group $\mathbb{R}$ consisting of hyperbolic elements:

\begin{fact}[Criterion for the Calabi--Markus phenomenon]\label{fact:C-M}
In Setting $\ref{setting:main}$, the following seven conditions on a semisimple symmetric pair $(G,H)$ are equivalent$:$
\begin{enumerate}
\item \label{item:C-M_proper_of_R} There exists a Lie group homomorphism $\Phi:\mathbb{R} \rightarrow G$ such that $\mathbb{R}$ acts properly on $G/H$ via $\Phi$.
\item \label{item:C-M_Z} $G/H$ admits the abelian group $\mathbb{Z}$ as a discontinuous group.
\item \label{item:C-M_inf_disconti} $G/H$ admits an infinite discontinuous group.
\item \label{item:C-M_hyp_free_gp} $G/H$ admits a discontinuous group which is a free group generated by a hyperbolic element in $G$. 
\item \label{item:C-M_split_rank} $\rank_\mathbb{R} \mathfrak{g} > \rank_\mathbb{R} \mathfrak{h}$.
\item \label{item:C-M_real_hyp_orbits} There exists a real hyperbolic adjoint orbit $\mathcal{O}^{G}_{\hyp}$ of $G$ in $\mathfrak{g}$ such that $\mathcal{O}^{G}_\hyp \cap \mathfrak{h} = \emptyset$.
\item \label{item:C-M_cpx_hyp_orbits} There exists a complex hyperbolic adjoint orbit $\mathcal{O}^{G_\mathbb{C}}_{\hyp}$ of $G_\mathbb{C}$ in $\mathfrak{g}_\mathbb{C}$ such that $\mathcal{O}^{G_\mathbb{C}}_\hyp \cap \mathfrak{g} \neq \emptyset$ and $\mathcal{O}^{G_\mathbb{C}}_\hyp \cap \mathfrak{g}^c = \emptyset$.
\end{enumerate}
\end{fact}

The equivalence among \eqref{item:C-M_proper_of_R}, \eqref{item:C-M_Z}, \eqref{item:C-M_inf_disconti}, \eqref{item:C-M_hyp_free_gp} and \eqref{item:C-M_split_rank} in Fact \ref{fact:C-M} was proved in a more general setting in T.~Kobayashi \cite[Corollary 4.4]{Kobayashi89}.
The real rank condition \eqref{item:C-M_split_rank} serves as a criterion for the Calabi--Markus phenomenon \eqref{item:C-M_inf_disconti} in Fact \ref{fact:C-M} (cf. \cite{Calabi-Murkus62}, \cite{Kobayashi89}). 
We will give a proof of the equivalence among \eqref{item:C-M_split_rank}, \eqref{item:C-M_real_hyp_orbits} and \eqref{item:C-M_cpx_hyp_orbits} in Appendix \ref{subsection:proof_of_fact}.

The second main result is a classification of semisimple symmetric pairs $(G,H)$ satisfying one of (therefore, all of) the equivalent conditions in Theorem \ref{thm:main}.

If a semisimple symmetric pair $(G,H)$ is irreducible, but $G$ is not simple, 
then $G/H$ admits a proper $SL(2,\mathbb{R})$-action, since the symmetric space $G/H$ can be regarded as a complex simple Lie group.
Therefore, the crucial case is on symmetric pairs $(\mathfrak{g},\mathfrak{h})$ with simple Lie algebra $\mathfrak{g}$.

To describe our classification, 
we denote by
\begin{multline*}
S := \{\, (\mathfrak{g},\mathfrak{h}) \mid (\mathfrak{g},\mathfrak{h}) \text{ is a semisimple symmetric pair} \\ \text{with a simple Lie algebra } \mathfrak{g} \,\}
\end{multline*}

The set $S$ was classified by M.~Berger \cite{Berger57} up to isomorphisms.
We also put 
\begin{align*}
A &:= \{\, (\mathfrak{g},\mathfrak{h}) \in S \mid \text{$(\mathfrak{g},\mathfrak{h})$ satisfies one of the conditions in Theorem \ref{thm:main}} \,\}, \\
B &:= \{\, (\mathfrak{g},\mathfrak{h}) \in S \mid \rank_\mathbb{R} \mathfrak{g} > \rank_\mathbb{R} \mathfrak{h} \,\} \setminus A, \\
C &:= \{\, (\mathfrak{g},\mathfrak{h}) \in S \mid \rank_\mathbb{R} \mathfrak{g} = \rank_\mathbb{R} \mathfrak{h} \,\}.
\end{align*}
Then $A \cap C = \emptyset$ by Fact \ref{fact:C-M},
and we have 
\[
S = A \sqcup B \sqcup C.
\]
One can easily determine the set $C$ in $S$.
Thus, to describe the classification of $A$, we only need to give the classification of $B$.

Here is our classification of the set $B$, 
namely, a complete list of $(\mathfrak{g},\mathfrak{h})$ satisfying the following: 
\begin{multline}
\text{$\mathfrak{g}$ is simple, $(\mathfrak{g},\mathfrak{h})$ is a symmetric pair with $\rank_\mathbb{R} \mathfrak{g} > \rank_\mathbb{R} \mathfrak{h}$} \\
\text{but does not satisfies the equivalent conditions in Theorem \ref{thm:main}.} \label{condition:classification}
\end{multline}

\begin{center}
\begin{longtable}{ll} \toprule
$\mathfrak{g}$ & $\mathfrak{h}$ \\ \midrule
$\mathfrak{sl}(2k,\mathbb{R})$ & $\mathfrak{sp}(k,\mathbb{R})$ \\ \hline
$\mathfrak{sl}(2k,\mathbb{R})$ & $\mathfrak{so}(k,k)$ \\ \hline
$\mathfrak{sl}(2k-1,\mathbb{R})$ & $\mathfrak{so}(k,k-1)$ \\ \hline
$\mathfrak{su}^*(4m+2)$ & $\mathfrak{sp}(m+1,m)$ \\ \hline
$\mathfrak{su}^*(4m)$ & $\mathfrak{sp}(m,m)$ \\ \hline
$\mathfrak{su}^*(2k)$ & $\mathfrak{so}^*(2k)$ \\ \hline 
$\mathfrak{so}(2k-1,2k-1)$ & $\mathfrak{so}(i+1,i) \oplus \mathfrak{so}(j,j+1)$ \\ 
 & $( i+j = 2k-2 )$ \\ \hline
$\mathfrak{e}_{6(6)}$ & $\mathfrak{f}_{4(4)}$ \\ \hline 
$\mathfrak{e}_{6(6)}$ & $\mathfrak{sp}(4,\mathbb{R})$ \\ \hline 
$\mathfrak{e}_{6(-26)}$ & $\mathfrak{sp}(3,1)$ \\ \hline
$\mathfrak{e}_{6(-26)}$ & $\mathfrak{f}_{4(-20)}$ \\ \hline
$\mathfrak{sl}(n,\mathbb{C})$ & $\mathfrak{so}(n,\mathbb{C})$ \\ \hline
$\mathfrak{sl}(2k,\mathbb{C})$ & $\mathfrak{sp}(k,\mathbb{C})$ \\ \hline
$\mathfrak{sl}(2k,\mathbb{C})$ & $\mathfrak{su}(k,k)$ \\ \hline

$\mathfrak{so}(4m+2,\mathbb{C})$ & $\mathfrak{so}(i,\mathbb{C}) \oplus \mathfrak{so}(j,\mathbb{C})$ \\
& $(i+j = 4m+2, \text{ $i,j$ are odd})$ \\ \hline
$\mathfrak{so}(4m+2,\mathbb{C})$ & $\mathfrak{so}(2m+2,2m)$ \\ \hline

$\mathfrak{e}_{6,\mathbb{C}}$ & $\mathfrak{sp}(4,\mathbb{C})$ \\ \hline
$\mathfrak{e}_{6,\mathbb{C}}$ & $\mathfrak{f}_{4,\mathbb{C}}$ \\ \hline
$\mathfrak{e}_{6,\mathbb{C}}$ & $\mathfrak{e}_{6(2)}$ \\
\bottomrule
\caption{Classification of $(\mathfrak{g},\mathfrak{h})$ satisfying \eqref{condition:classification}}
\label{table:non_sl2}
\end{longtable}
\end{center}

Here, $k \geq 2$, $m \geq 1$ and $n \geq 2$.

Theorem \ref{thm:cls_intro}, which gives a classification of the set $A$, is obtained by Table \ref{table:non_sl2}.

Concerning our classification, we will give an algorithm to check whether or not a given symmetric pair $(\mathfrak{g},\mathfrak{h})$ satisfies the condition \eqref{item:main_hyperbolic} in Theorem \ref{thm:main}.
More precisely, we will determine the set of complex antipodal hyperbolic orbits in a complex simple Lie algebra $\mathfrak{g}_\mathbb{C}$ (see Section \ref{subsection:complex_Bhyp}) and introduce an algorithm to check whether or not a given such orbit meets a real form $\mathfrak{g}$ [resp. $\mathfrak{g}^c$] (see Section \ref{section:Real_forms}).
Table \ref{table:non_sl2} is obtained by using this algorithm (see Section \ref{section:classifications}).

\begin{rem}
\begin{description}
\item[(1)] Using \cite[Theorem 1.1]{Benoist96}, Benoist gave a number of examples of symmetric pairs $(G,H)$ which do not satisfy the condition $\eqref{item:main_non-virtually_abelian}$ in Theorem $\ref{thm:main}$ with $\rank_\mathbb{R} \mathfrak{g} > \rank_\mathbb{R} \mathfrak{h}$ $($see \cite[Example~1]{Benoist96}$)$. Table $\ref{table:non_sl2}$ gives its complete list.
\item[(2)] We take this opportunity to correct \cite[Table~2.6]{Okuda11_proc}, where the pair $(\mathfrak{sl}(2k-1,\mathbb{R}), \mathfrak{so}(k,k-1))$ was missing.
\end{description}
\end{rem}

We discuss an application of the main result (Theorem \ref{thm:main}) to the existence problem of compact Clifford--Klein forms.
As we explained in Introduction, a Clifford--Klein form of $G/H$ is the double coset space $\Gamma \backslash G/H$ when $\Gamma$ is a discrete subgroup of $G$ acting on $G/H$ properly discontinuously and freely.
Recall that 
we say that a homogeneous space $G/H$ \textit{admits compact Clifford--Klein forms},
if there exists such $\Gamma$ where $\Gamma \backslash G/H$ is compact. 
See also \cite{Benoist96,Kobayashi89,Kobayashi92b,Kobayashi92a,Kobayashi97disc,Kobayashi-Ono90Hirzebruch,Kobayashi-Yoshino05,Labourie-Zimmer95,Margulis97,Oh-Witte02,Shalom00rigidity,Zimmer94} for preceding results for the existence problem for compact Clifford--Klein forms. 
Among them, there are three methods that can be applied to semisimple symmetric spaces to show the non-existence of compact Clifford--Klein forms:
\begin{itemize}
\item 
Using the Hirzebruch--Kobayashi--Ono proportionality principle \cite[Proposition 4.10]{Kobayashi89}, \cite{Kobayashi-Ono90Hirzebruch}.
\item Using a comparison theorem of cohomological dimension \cite[Theorem 1.5]{Kobayashi92a}. (A generalization of the criterion in \cite{Kobayashi89} of the Calabi--Markus phenomenon.)
\item Using a criterion for the non-existence of properly discontinuous actions of non-virtually abelian groups \cite[Corollary 1]{Benoist96}.
\end{itemize}

As an immediate corollary of the third method
and the description of the set $B$ by Table \ref{table:non_sl2}, 
one concludes:

\begin{cor}\label{cor:no_cpt_CK_form}
The simple symmetric space $G/H$ does not admit compact Clifford--Klein forms if $(\mathfrak{g},\mathfrak{h})$ is in Table \ref{table:non_sl2}.
\end{cor}

\section{Preliminary results for proper actions}\label{section:pre:KB}

In this section, we recall results of T.~Kobayashi \cite{Kobayashi89} and Y.~Benoist \cite{Benoist96} in a form that we shall need.
Our proofs of the equivalences \eqref{item:main_sl2-proper} $\Leftrightarrow$ \eqref{item:main_sl2-triple} and \eqref{item:main_non-virtually_abelian} $\Leftrightarrow$ \eqref{item:main_hyperbolic} in Theorem \ref{thm:main} will be based on these results (see Section \ref{subsection:proof_of_sl2-proper_triple} and Section \ref{subsection:proof_main_benoist}).

\subsection{Kobayashi's properness criterion}\label{subsection:Kobayashi's_criterion}

Let $G$ be a connected linear semisimple Lie group and write $\mathfrak{g}$ for the Lie algebra of $G$.
First, we fix a terminology as follows:

\begin{defn}\label{defn:reductive_in}
We say that a subalgebra $\mathfrak{h}$ of $\mathfrak{g}$ is \textit{reductive in $\mathfrak{g}$} if there exists a Cartan involution $\theta$ of $\mathfrak{g}$ such that $\mathfrak{h}$ is $\theta$-stable.
Furthermore, we say that a closed subgroup $H$ of $G$ is \textit{reductive in $G$} if $H$ has only finitely many connected components and the Lie algebra $\mathfrak{h}$ of $H$ is reductive in $\mathfrak{g}$.
\end{defn}

For simplicity, we call $\mathfrak{h}$ [resp. $H$] a reductive subalgebra of $\mathfrak{g}$ [resp. a reductive subgroup of $G$] if $\mathfrak{h}$ is reductive in $\mathfrak{g}$ [resp. $H$ is reductive in $G$]. 
We call such $(G,H)$ a reductive pair.
Note that a reductive subalgebra $\mathfrak{h}$ of $\mathfrak{g}$ is a reductive Lie algebra.

We give two examples relating to Theorem \ref{thm:main}:

\begin{example}\label{ex:ss_pair_is_red}
In Setting $\ref{setting:main}$, the subgroup $H$ is reductive in $G$ since there exists a Cartan involution $\theta$ on $\mathfrak{g}$, which is commutative with $\sigma$ $($cf. \cite{Berger57}$)$.
\end{example}

\begin{example}\label{ex:ss_is_red}
Let $\mathfrak{l}$ be a semisimple subalgebra of $\mathfrak{g}$.
Then any Cartan involution on $\mathfrak{l}$ can be extended to a Cartan involution on $\mathfrak{g}$ $($cf. G.~D.~Mostow \cite{Mostow55}$)$ 
and the analytic subgroup $L$ corresponding to $\mathfrak{l}$ is closed in $G$ $($cf. K.~Yosida \cite{yosida38}$)$.
Therefore, $\mathfrak{l}$ $[$resp. $L$$]$ is reductive in $\mathfrak{g}$ $[$resp. $G$$]$.
\end{example}

In the rest of this subsection, we follow the setting below:
\begin{setting}\label{setting:reductive}
$G$ is a connected linear semisimple Lie group, $H$ and $L$ are reductive subgroups of $G$. 
\end{setting}

We denote by $\mathfrak{g}$, $\mathfrak{h}$ and $\mathfrak{l}$ the Lie algebras of $G$, $H$ and $L$, respectively.
Take a Cartan involution $\theta$ of $\mathfrak{g}$ which preserves $\mathfrak{h}$.
We write $\mathfrak{g}= \mathfrak{k} + \mathfrak{p}$, $\mathfrak{h} = \mathfrak{k}(\mathfrak{h}) + \mathfrak{p}(\mathfrak{h})$ for the Cartan decomposition of $\mathfrak{g}$, $\mathfrak{h}$ corresponding to $\theta$, $\theta|_\mathfrak{h}$, respectively.
We fix a maximal abelian subspace $\mathfrak{a}_\mathfrak{h}$ of $\mathfrak{p}(\mathfrak{h})$ (i.e. $\mathfrak{a}_\mathfrak{h}$ is a maximally split abelian subspace of $\mathfrak{h}$), and extend it to a maximal abelian subspace $\mathfrak{a}$ in $\mathfrak{p}$ (i.e. $\mathfrak{a}$ is a maximally split abelian subspace of $\mathfrak{g}$). 
We write $K$ for the maximal compact subgroup of $G$ with its Lie algebra $\mathfrak{k}$, 
and denote the Weyl group acting on $\mathfrak{a}$ by $W(\mathfrak{g},\mathfrak{a}) := N_K(\mathfrak{a})/Z_K(\mathfrak{a})$.
Since $\mathfrak{l}$ is also reductive in $\mathfrak{g}$, 
we can take a Cartan involution $\theta'$ of $\mathfrak{g}$ preserving $\mathfrak{l}$.
We write $\mathfrak{l} = \mathfrak{k}'(\mathfrak{l}) + \mathfrak{p}'(\mathfrak{l})$ for the Cartan decomposition of $\mathfrak{l}$ corresponding to $\theta'|_{\mathfrak{l}}$, and fix a maximal abelian subspace $\mathfrak{a}'_\mathfrak{l}$ of $\mathfrak{p}'(\mathfrak{l})$.
Then there exists $g \in G$ such that $\Ad(g) \cdot \mathfrak{a}'_\mathfrak{l}$ is contained in $\mathfrak{a}$, and we put $\mathfrak{a}_\mathfrak{l} := \Ad(g) \cdot \mathfrak{a}'_\mathfrak{l}$. 
The subset $W(\mathfrak{g},\mathfrak{a}) \cdot \mathfrak{a}_\mathfrak{l}$ of $\mathfrak{a}$ does not depend on a choice of such $g \in G$.

The following fact holds:
\begin{fact}[T.~Kobayashi {\cite[Theorem 4.1]{Kobayashi89}}]\label{fact:Kobayashi's_criterion}
In Setting $\ref{setting:reductive}$, 
$L$ acts on $G/H$ properly 
if and only if 
\[
\mathfrak{a}_\mathfrak{h} \cap W(\mathfrak{g},\mathfrak{a}) \cdot \mathfrak{a}_\mathfrak{l} = \{0\}.
\]
\end{fact}

The proof of Fact \ref{fact:C-M} is reduced to Fact \ref{fact:Kobayashi's_criterion} (see \cite{Kobayashi89}).
However, to prove the equivalences between \eqref{item:C-M_split_rank}, \eqref{item:C-M_real_hyp_orbits} and \eqref{item:C-M_cpx_hyp_orbits} in Fact \ref{fact:C-M} we need an additional argument which will be described in Appendix \ref{subsection:proof_of_fact}.

\subsection{Benoist's criterion}\label{subsection:Benoist_result}

Let $(G,H)$ be a reductive pair (see Definition \ref{defn:reductive_in}).
In this subsection, we use the notation
$\mathfrak{g}$, $\mathfrak{h}$, $\theta$, $\mathfrak{a}_\mathfrak{h}$, $\mathfrak{a}$ and $W(\mathfrak{g},\mathfrak{a})$
as in the previous subsection.

Let us denote the restricted root system of $(\mathfrak{g},\mathfrak{a})$ by
$\Sigma(\mathfrak{g}, \mathfrak{a})$.
We fix a positive system $\Sigma^+(\mathfrak{g},\mathfrak{a})$ of $\Sigma(\mathfrak{g},\mathfrak{a})$, and put 
\[
\mathfrak{a}_+ := \{\, A \in \mathfrak{a} \mid \xi(X) \geq 0 \ \text{for any } \xi \in \Sigma^+(\mathfrak{g},\mathfrak{a}) \,\}.
\]
Then $\mathfrak{a}_+$ is a fundamental domain for the action of the Weyl group $W(\mathfrak{g},\mathfrak{a})$.
We write $w_0$ for the longest element in $W(\mathfrak{g},\mathfrak{a})$ with respect to the positive system $\Sigma^+(\mathfrak{g},\mathfrak{a})$.
Then, by the action of $w_0$, every element in $\mathfrak{a}_+$ moves to $-\mathfrak{a}_+ := \{ -A \mid A \in \mathfrak{a}_+ \}$.
In particular,
\[
-w_0 : \mathfrak{a} \rightarrow \mathfrak{a},\quad A \mapsto -(w_0 \cdot A)
\] 
is an involutive automorphism on $\mathfrak{a}$ preserving $\mathfrak{a}_+$.
We put 
\begin{align*}
\mathfrak{b} := \{\, A \in \mathfrak{a} \mid -w_0 \cdot A = A \,\}, \quad
\mathfrak{b}_+ := \mathfrak{b} \cap \mathfrak{a}_+.
\end{align*}

Then the next fact holds:
\begin{fact}[Y.~Benoist {\cite[Theorem in Section 1.1]{Benoist96}}]\label{fact:Benoist's_result}
The following conditions on a reductive pair $(G,H)$ are equivalent$:$
\begin{enumerate}
\item \label{item:Benoist_virtually_abelian} $G/H$ admits an infinite discontinuous group which is not virtually abelian.
\item \label{item:Benoist_b_+_for_a_w} $\mathfrak{b}_+ \not \subset w \cdot \mathfrak{a}_\mathfrak{h}$ for any $w \in W(\mathfrak{g},\mathfrak{a})$.
\item \label{item:b_+} $\mathfrak{b}_+ \not \subset W(\mathfrak{g},\mathfrak{a}) \cdot \mathfrak{a}_\mathfrak{h}$.
\end{enumerate}
\end{fact}

\begin{rem}
Benoist 
showed $\eqref{item:Benoist_virtually_abelian}$ $\Leftrightarrow$ $\eqref{item:Benoist_b_+_for_a_w}$ in Fact $\ref{fact:Benoist's_result}$.
The equivalence $\eqref{item:Benoist_b_+_for_a_w}$ $\Leftrightarrow$ $\eqref{item:b_+}$ follows from the fact below $($since $\mathfrak{b}_+$ is a convex set of $\mathfrak{a}$ and $w \cdot \mathfrak{a}_\mathfrak{h}$ is a linear subspace of $\mathfrak{a}$ for any $w \in W(\mathfrak{g},\mathfrak{a})$$)$$.$
\begin{fact}
Let $U_1$, $U_2$, \dots, $U_n$ be subspaces of a finite dimensional real vector space $V$ and $\Omega$ a convex set of $V$.
Then $\Omega$ is contained in $\bigcup_{i=1}^{n} U_i$ if and only if $\Omega$ is contained in $U_k$ for some $k \in \{1,\dots,n\}$.
\end{fact}
\end{rem}

\section{Proof of Theorem \ref{thm:main}}\label{section:outline_of_proof}

We give a proof of Theorem \ref{thm:main} by proving the implications in the figure below:

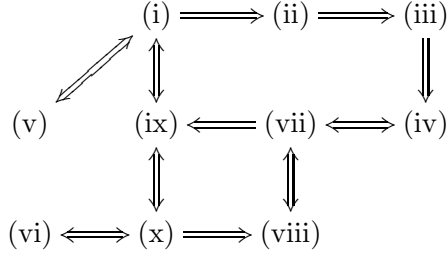
\[\xymatrix{
& \eqref{item:main_sl2-proper} \ar@{=>}[r] \ar@{<=>}[d] \ar@{<=>}[ld] & \eqref{item:main_any_surface_group} \ar@{=>}[r] & \eqref{item:main_some_surface_group} \ar@{=>}[d] \\
\eqref{item:main_unipotent_discontinuous} & \eqref{item:main_sl2-triple_R} \ar@{<=>}[d] \ar@{<=}[r] & \eqref{item:main_realhyperbolic} \ar@{<=>}[r] \ar@{<=>}[d] & \eqref{item:main_non-virtually_abelian} \\
\eqref{item:main_nilpotent} \ar@{<=>}[r] & \eqref{item:main_sl2-triple} \ar@{=>}[r] & \eqref{item:main_hyperbolic}& 
}\]

In this section, to show the implications, we use some theorems, propositions and lemmas, which will be proved later in this paper.

{\bf Notation:} Throughout this paper, for a complex semisimple Lie algebra $\mathfrak{g}_\mathbb{C}$ and its real form $\mathfrak{g}$,
we denote a complex [resp. real] nilpotent, hyperbolic, antipodal hyperbolic adjoint orbit in $\mathfrak{g}_\mathbb{C}$ [resp. $\mathfrak{g}$] simply by a complex [resp. real] nilpotent, hyperbolic, antipodal hyperbolic orbit in $\mathfrak{g}_\mathbb{C}$ [resp. $\mathfrak{g}$].

\subsection{Proof of \eqref{item:main_sl2-proper} $\Leftrightarrow$ \eqref{item:main_sl2-triple_R} in Theorem \ref{thm:main}}\label{subsection:proof_of_sl2-proper_triple}

Our proof of the equivalence \eqref{item:main_sl2-proper} $\Leftrightarrow$ \eqref{item:main_sl2-triple_R} in Theorem \ref{thm:main} starts with the next theorem, which will be proved in Section \ref{section:hyp-orbit_and_proper_actions}:

\begin{thm}[Corollary to Fact \ref{fact:Kobayashi's_criterion}]\label{thm:cor_to_Kobayashi}
In Setting $\ref{setting:reductive}$, the following conditions on $(G,H,L)$ are equivalent$:$
\begin{enumerate}
\item $L$ acts on $G/H$ properly,
\item \label{item:orbit_criterion_of_Kobayashi} There do not exist real hyperbolic orbits in $\mathfrak{g}$ $($see Definition $\ref{defn:hyp_anti}$$)$ meeting both $\mathfrak{l}$ and $\mathfrak{h}$ other than the zero-orbit, 
\end{enumerate}
where $\mathfrak{g}$, $\mathfrak{h}$ and $\mathfrak{l}$ are Lie algebras of $G$, $H$ and $L$, respectively.
\end{thm}

By using Theorem \ref{thm:cor_to_Kobayashi}, we will prove the next proposition in Section \ref{section:hyp-orbit_and_proper_actions}:

\begin{prop}\label{prop:Sl_2_proper_and_sl2_triple}
Let $(G,H)$ be a reductive pair $($see Definition $\ref{defn:reductive_in}$$)$.
Then there exists a bijection between the following two sets$:$
\begin{itemize}
\item The set of Lie group homomorphisms $\Phi : SL(2,\mathbb{R}) \rightarrow G$ such that $SL(2,\mathbb{R})$ acts on $G/H$ properly via $\Phi$,
\item The set of $\mathfrak{sl}_2$-triples $(A,X,Y)$ in $\mathfrak{g}$ such that the real adjoint orbit through $A$ does not meet $\mathfrak{h}$.
\end{itemize}
\end{prop}

In Setting \ref{setting:main}, 
the subgroup $H$ of $G$ is reductive in $G$ (see Example \ref{ex:ss_pair_is_red}).
Hence, we obtain the equivalence \eqref{item:main_sl2-proper} $\Leftrightarrow$ \eqref{item:main_sl2-triple_R} in Theorem \ref{thm:main}.

\subsection{Proof of \eqref{item:main_non-virtually_abelian} $\Leftrightarrow$ \eqref{item:main_realhyperbolic} in Theorem \ref{thm:main}}\label{subsection:proof_main_benoist}

We will prove the next theorem in Section \ref{section:hyp-orbit_and_proper_actions}:
\begin{thm}[Corollary to Fact \ref{fact:Benoist's_result}]\label{thm:cor_to_Benoist}
The following conditions on a reductive pair $(G,H)$ $($see Definition $\ref{defn:reductive_in}$$)$ are equivalent$:$
\begin{enumerate}
\item $G/H$ admits an infinite discontinuous group that is not virtually abelian.
\item \label{item:Benoist_orbit} There exists a real antipodal hyperbolic orbit in $\mathfrak{g}$ that does not meet $\mathfrak{h}$.
\end{enumerate}
\end{thm}

In Setting \ref{setting:main},
the equivalence \eqref{item:main_non-virtually_abelian} $\Leftrightarrow$ \eqref{item:main_realhyperbolic} in Theorem \ref{thm:main} holds as a special case of Theorem \ref{thm:cor_to_Benoist}.

\subsection{Proofs of \eqref{item:main_sl2-triple} $\Leftrightarrow$ \eqref{item:main_sl2-triple_R}, \eqref{item:main_hyperbolic} $\Leftrightarrow$ \eqref{item:main_realhyperbolic} and \eqref{item:main_sl2-triple} $\Rightarrow$ \eqref{item:main_hyperbolic} in Theorem \ref{thm:main}}\label{subsection:proof_main_complexification}

Let $\mathfrak{g}_\mathbb{C}$ be a complex semisimple Lie algebra.
We use the following convention for hyperbolic elements (see Definition \ref{defn:hyp_anti}):
\begin{align*}
\mathcal{H} &:= \{\, A \in \mathfrak{g}_\mathbb{C} \mid \text{$A$ is a hyperbolic element in $\mathfrak{g}_\mathbb{C}$} \,\}, \\
\mathcal{H}^a &:= \{\, A \in \mathcal{H} \mid \text{The complex adjoint orbit through $A$ is antipodal} \,\},\\
\mathcal{H}^n &:= \{\, A \in \mathfrak{g}_\mathbb{C} \mid \text{There exist $X,Y \in \mathfrak{g}_\mathbb{C}$} \\
&\hspace{5cm}  \text{such that $(A,X,Y)$ is an $\mathfrak{sl}_2$-triple} \,\}.
\end{align*}
We also write $\mathcal{H}/{G_\mathbb{C}}$, $\mathcal{H}^a/{G_\mathbb{C}}$ for the sets of complex hyperbolic orbits and complex antipodal hyperbolic orbits in $\mathfrak{g}_\mathbb{C}$, respectively. 
Let us denote by $\mathcal{H}^n/{G_\mathbb{C}}$ the set of complex adjoint orbits contained in $\mathcal{H}^n$.

The next lemma will be proved in Section \ref{subsection:WDD_of_comp_nilp}:

\begin{lem}\label{lem:sl2-triple_to_Bhyp}
For any $\mathfrak{sl}_2$-triple $(A,X,Y)$ in $\mathfrak{g}_\mathbb{C}$, 
the element $A$ of $\mathfrak{g}_\mathbb{C}$ is hyperbolic and the complex adjoint orbit through $A$ in $\mathfrak{g}_\mathbb{C}$ is antipodal.
\end{lem}

By Lemma \ref{lem:sl2-triple_to_Bhyp}, we have 
\[
\mathcal{H}^n \subset \mathcal{H}^a \subset \mathcal{H}.
\]

Hence, the implication \eqref{item:main_sl2-triple} $\Rightarrow$ \eqref{item:main_hyperbolic} in Theorem \ref{thm:main} follows.

Further, for any subalgebra $\mathfrak{l}$ of $\mathfrak{g}_\mathbb{C}$, 
we also use the following convention: 
\begin{align*}
\mathcal{H}_\mathfrak{l} &:= \{\, A \in \mathcal{H} \mid \text{The complex adjoint orbit through $A$ meets $\mathfrak{l}$} \,\}, \\
\mathcal{H}^a_\mathfrak{l} &:= \mathcal{H}^a \cap \mathcal{H}_\mathfrak{l}, \\
\mathcal{H}^n_\mathfrak{l} &:= \mathcal{H}^n \cap \mathcal{H}_\mathfrak{l}.
\end{align*}
Let us write $\mathcal{H}_\mathfrak{l}/{G_\mathbb{C}}$, $\mathcal{H}^a_\mathfrak{l}/{G_\mathbb{C}}$, $\mathcal{H}^n_\mathfrak{l}/{G_\mathbb{C}}$ for the sets of complex adjoint orbits contained in $\mathcal{H}$, $\mathcal{H}^a$, $\mathcal{H}^n$ meeting $\mathfrak{l}$, respectively.

Here, we fix a real form $\mathfrak{g}$, and set
\begin{align*}
\mathcal{H}(\mathfrak{g}) &:= \{\, A \in \mathfrak{g} \mid \text{$A$ is a hyperbolic element in $\mathfrak{g}$} \,\}, \\
\mathcal{H}^a(\mathfrak{g}) &:= \{\, A \in \mathcal{H}(\mathfrak{g}) \mid \text{The real adjoint orbit through $A$ is antipodal} \,\}, \\
\mathcal{H}^n(\mathfrak{g}) &:= \{\, A \in \mathfrak{g} \mid \text{There exist $X,Y \in \mathfrak{g}$} \\
& \hspace{4.5cm} \text{such that $(A,X,Y)$ is an $\mathfrak{sl}_2$-triple} \,\}.
\end{align*}
We also write $\mathcal{H}(\mathfrak{g})/{G}$, $\mathcal{H}^a(\mathfrak{g})/{G}$, $\mathcal{H}^n(\mathfrak{g})/{G}$ for the sets of real adjoint orbits contained in $\mathcal{H}(\mathfrak{g})$, $\mathcal{H}^a(\mathfrak{g})$, $\mathcal{H}^n(\mathfrak{g})$, respectively. 

Then the following proposition gives a one-to-one correspondence between real hyperbolic orbits and complex hyperbolic orbits with real points:

\begin{prop}\label{prop:Hyp_in_gC_in_g}
\begin{enumerate}
\item \label{item:Hyp_c_r:hyp}
The following map gives a one-to-one correspondence between $\mathcal{H}(\mathfrak{g})/{G}$ and $\mathcal{H}_\mathfrak{g}/{G_\mathbb{C}}$$:$
\begin{align*}
\mathcal{H}(\mathfrak{g})/{G} \rightarrow \mathcal{H}_\mathfrak{g}/{G_\mathbb{C}}, 
&\quad 
\mathcal{O}^{G}_\hyp \mapsto \Ad(G_\mathbb{C}) \cdot \mathcal{O}^{G}_\hyp, \\
\mathcal{H}_\mathfrak{g}/{G_\mathbb{C}} \rightarrow \mathcal{H}(\mathfrak{g})/{G}, 
&\quad 
\mathcal{O}^{G_\mathbb{C}}_\hyp \mapsto \mathcal{O}^{G_\mathbb{C}}_\hyp \cap \mathfrak{g}.
\end{align*}
\item \label{item:Hyp_c_r:anti} The bijection in \eqref{item:Hyp_c_r:hyp} gives the one-to-one correspondence below:
\[
\mathcal{H}^a(\mathfrak{g})/G \bijarrow \mathcal{H}^a_\mathfrak{g}/{G_\mathbb{C}}.
\]
\item \label{item:Hyp_c_r:neut} The bijection in \eqref{item:Hyp_c_r:hyp} gives the one-to-one correspondence below:
\[
\mathcal{H}^n(\mathfrak{g})/G \bijarrow \mathcal{H}^n_\mathfrak{g}/{G_\mathbb{C}}.
\]
\end{enumerate}
\end{prop}

The proof of Proposition \ref{prop:Hyp_in_gC_in_g} will be given in Section \ref{section:Real_forms}.

In Setting \ref{setting:main}, recall that both $\mathfrak{g}$ and $\mathfrak{g}^c$ are real forms of $\mathfrak{g}_\mathbb{C}$.
In Section \ref{section:Symmetric_pair}, we will prove the following proposition, which claims that a complex hyperbolic orbit meets $\mathfrak{h}$ if it meets both $\mathfrak{g}$ and $\mathfrak{g}^c$:

\begin{prop}\label{prop:keyprop}
In Setting $\ref{setting:main}$, 
$
\mathcal{H}_\mathfrak{g} \cap \mathcal{H}_{\mathfrak{g}^c} = \mathcal{H}_{\mathfrak{h}}$.
\end{prop}

The equivalences \eqref{item:main_sl2-triple} $\Leftrightarrow$ \eqref{item:main_sl2-triple_R} and \eqref{item:main_hyperbolic} $\Leftrightarrow$ \eqref{item:main_realhyperbolic} in Theorem \ref{thm:main} follows from Proposition \ref{prop:Hyp_in_gC_in_g} and Proposition \ref{prop:keyprop}.

\subsection{Proof of \eqref{item:main_nilpotent} $\Leftrightarrow$ \eqref{item:main_sl2-triple} in Theorem \ref{thm:main}}\label{subsection:proof_sl2_nilp}

The equivalence \eqref{item:main_nilpotent} $\Leftrightarrow$ \eqref{item:main_sl2-triple} in Theorem \ref{thm:main} can be obtained by the Jacobson--Morozov theorem and the lemma below (see Proposition \ref{prop:cor_to_Sekiguchi} for a proof):

\begin{lem}[Corollary to J.~Sekiguchi {\cite[Proposition 1.11]{Sekiguchi87}}]\label{lem:A_meets_X_meets}
Let $\mathfrak{g}_\mathbb{C}$ be a complex semisimple Lie algebra and $\mathfrak{g}$ a real form of $\mathfrak{g}_\mathbb{C}$. 
Then the following conditions on an $\mathfrak{sl}_2$-triple $(A,X,Y)$ in $\mathfrak{g}_\mathbb{C}$ are equivalent$:$
\begin{enumerate}
\item The complex adjoint orbit through $A$ in $\mathfrak{g}_\mathbb{C}$ meets $\mathfrak{g}$. 
\item The complex adjoint orbit through $X$ in $\mathfrak{g}_\mathbb{C}$ meets $\mathfrak{g}$. 
\end{enumerate}
\end{lem}

\subsection{Proof of \eqref{item:main_realhyperbolic} $\Rightarrow$ \eqref{item:main_sl2-triple_R} in Theorem \ref{thm:main}}\label{subsection:proof_main}

Let $\mathfrak{g}$ be a semisimple Lie algebra.
In this subsection, we use $\mathcal{H}(\mathfrak{g})$, $\mathcal{H}^a(\mathfrak{g})$ and $\mathcal{H}^n(\mathfrak{g})$ as in Section \ref{subsection:proof_main_complexification}.

To prove the implication \eqref{item:main_realhyperbolic} $\Rightarrow$ \eqref{item:main_sl2-triple_R}, 
we use the next proposition and lemma:

\begin{prop}\label{prop:neutral_control}
We take \[
\mathfrak{b} := \{\, A \in \mathfrak{a} \mid -w_0 \cdot A = A \,\},\quad \mathfrak{b}_+ := \mathfrak{b} \cap \mathfrak{a}_+
\]
as in Section $\ref{subsection:Benoist_result}$.
Then the following 
holds$:$
\begin{enumerate}
\item \label{item:neut_key} $\mathfrak{b} = \mathbb{R}\text{-}\mathrm{span} (\mathfrak{a}_+ \cap \mathcal{H}^n(\mathfrak{g}))$.
\item \label{item:neut_b_anti} $\mathcal{H}^a(\mathfrak{g}) = \Ad(G) \cdot \mathfrak{b}_+$.
\end{enumerate}
\end{prop}

\begin{lem}\label{lem:a_h}
Let $(\mathfrak{g},\mathfrak{h},\sigma)$ be a semisimple symmetric pair.
We fix a Cartan involution $\theta$ on $\mathfrak{g}$ 
such that $\theta \sigma = \sigma \theta$
and denote by $\mathfrak{g} = \mathfrak{k} + \mathfrak{p}$
the Cartan decomposition of $\mathfrak{g}$ with respect to $\theta$.
Let us take $\mathfrak{a}$ and $\mathfrak{a}_\mathfrak{h} = \mathfrak{a} \cap \mathfrak{h}$
as in Section $\ref{subsection:Kobayashi's_criterion}$.
We fix an ordering on $\mathfrak{a}_\mathfrak{h}$ and 
extend it to $\mathfrak{a}$,
and put $\mathfrak{a}_+$ to the closed Weyl chamber of $\mathfrak{a}$ 
with respect to the ordering.
Then
\[
\mathfrak{a}_+ \cap \mathcal{H}_\mathfrak{h}(\mathfrak{g}) \subset \mathfrak{a}_\mathfrak{h},
\]
where $\mathcal{H}_\mathfrak{h}(\mathfrak{g})$ is the set of hyperbolic elements in $\mathfrak{g}$ whose adjoint orbits meet $\mathfrak{h}$.
\end{lem}

Postponing the proof of Proposition \ref{prop:neutral_control} and Lemma \ref{lem:a_h} in later sections, 
we complete the proof of 
the implication \eqref{item:main_hyperbolic} $\Rightarrow$ \eqref{item:main_sl2-triple} in Theorem \ref{thm:main}.

\begin{proof}[Proof of $\eqref{item:main_hyperbolic}$ $\Rightarrow$ $\eqref{item:main_sl2-triple}$ in Theorem $\ref{thm:main}$]
We shall prove that $\mathcal{H}^a(\mathfrak{g}) \subset \mathcal{H}_\mathfrak{h}(\mathfrak{g})$ under the assumption $\mathcal{H}^n(\mathfrak{g}) \subset \mathcal{H}_\mathfrak{h}(\mathfrak{g})$.
By combining Proposition \ref{prop:neutral_control} \eqref{item:neut_key}, Lemma \ref{lem:a_h} with the assumption, we have
\[
\mathfrak{b} \subset \mathfrak{a}_\mathfrak{h} (\subset \mathfrak{h}).
\]
Therefore, by Proposition \ref{prop:neutral_control} \eqref{item:neut_b_anti}, we obtain that $\mathcal{H}^a(\mathfrak{g}) \subset \mathcal{H}_\mathfrak{h}(\mathfrak{g})$.
\end{proof}

We shall give a proof of 
Proposition \ref{prop:neutral_control} \eqref{item:neut_key} 
in Section \ref{subsection:proof_of_keyprop} 
by comparing Dynkin's classification of $\mathfrak{sl}_2$-triples in $\mathfrak{g}_\mathbb{C}$ \cite{Dynkin52eng} with the Satake diagram of the real form $\mathfrak{g}$ of $\mathfrak{g}_\mathbb{C}$. 
The proof of Proposition \ref{prop:neutral_control} \eqref{item:neut_b_anti} will be given in Section \ref{subsection:proof_of_cor_to_KB},
and that of Lemma \ref{lem:a_h} in Section \ref{section:Symmetric_pair}.

\subsection{Proofs of \eqref{item:main_sl2-proper} $\Rightarrow$ \eqref{item:main_any_surface_group}, \eqref{item:main_some_surface_group} $\Rightarrow$ \eqref{item:main_hyperbolic} and \eqref{item:main_sl2-proper} $\Leftrightarrow$ \eqref{item:main_unipotent_discontinuous} in Theorem \ref{thm:main}}

The implication \eqref{item:main_sl2-proper} $\Rightarrow$ \eqref{item:main_any_surface_group} in Theorem \ref{thm:main} is deduced from the lifting theorem of surface groups (cf. \cite{Kra85}).
The implication \eqref{item:main_some_surface_group} $\Rightarrow$ \eqref{item:main_hyperbolic} follows by the fact that the surface group of genus $g$ is not virtually abelian for any $g \geq 2$.

The equivalence \eqref{item:main_sl2-proper} $\Leftrightarrow$ \eqref{item:main_unipotent_discontinuous} can be proved by the observation below: 
Let $\Gamma_0$ be the free group generated by $\begin{pmatrix}1&1\\0&1\end{pmatrix}$ in $SL(2,\mathbb{R})$;
Then, for any free group $\Gamma$ generated by a unipotent element in a linear semisimple Lie group $G$, there exists a Lie group homomorphism $\Phi : SL(2,\mathbb{R}) \rightarrow G$ such that $\Phi(\Gamma_0) = \Gamma$ (by the Jacobson--Morozov theorem);
Furthermore, by \cite[Lemma 3.2]{Kobayashi93}, for any closed subgroup $H$ of $G$, the $SL(2,\mathbb{R})$-action on $G/H$ via $\Phi$ is proper if and only if the $\Gamma$-action on $G/H$ is properly discontinuous.

\section{Real hyperbolic orbits and proper actions of reductive subgroups}
\label{section:hyp-orbit_and_proper_actions}

In this section,
we prove Theorem \ref{thm:cor_to_Kobayashi}, Proposition \ref{prop:Sl_2_proper_and_sl2_triple}, Theorem \ref{thm:cor_to_Benoist} and Proposition \ref{prop:neutral_control} \eqref{item:neut_b_anti}.

\subsection{Kobayashi's properness criterion and Benoist's criterion rephrased by real hyperbolic orbits}\label{subsection:proof_of_cor_to_KB}

In this subsection, Theorem \ref{thm:cor_to_Kobayashi} and Theorem \ref{thm:cor_to_Benoist} are proved as corollaries to Fact \ref{fact:Kobayashi's_criterion} and Fact \ref{fact:Benoist's_result}, respectively.
We also prove Proposition \ref{prop:neutral_control} \eqref{item:neut_b_anti} in this subsection.

Let $\mathfrak{g}$ be a semisimple Lie algebra. 
The next fact for real hyperbolic orbits in $\mathfrak{g}$ (see Definition \ref{defn:hyp_anti}) is 
well known:

\begin{fact}\label{fact:hyp_inter_a}
Fix a Cartan decomposition $\mathfrak{g} = \mathfrak{k} + \mathfrak{p}$ of $\mathfrak{g}$ and a maximally split abelian subspace $\mathfrak{a}$ of $\mathfrak{g}$ $($i.e. $\mathfrak{a}$ is a maximal abelian subspace of $\mathfrak{p}$$)$.
Then any real hyperbolic orbit $\mathcal{O}^{G}_\hyp$ in $\mathfrak{g}$ meets $\mathfrak{a}$, 
and the intersection $\mathcal{O}^{G}_\hyp \cap \mathfrak{a}$ is a single $W(\mathfrak{g},\mathfrak{a})$-orbit, where $W(\mathfrak{g},\mathfrak{a}):= N_K(\mathfrak{a})/Z_K(\mathfrak{a})$.
In particular, we have a bijection 
\[
\mathcal{H}(\mathfrak{g})/G \rightarrow \mathfrak{a}/{W(\mathfrak{g},\mathfrak{a})}, \quad \mathcal{O}^{G}_\hyp \mapsto \mathcal{O}^{G}_\hyp \cap \mathfrak{a},
\]
where $\mathcal{H}(\mathfrak{g})/{G}$ is the set of real hyperbolic orbits in $\mathfrak{g}$ and $\mathfrak{a}/{W(\mathfrak{g},\mathfrak{a})}$ the set of $W(\mathfrak{g},\mathfrak{a})$-orbits in $\mathfrak{a}$.
\end{fact}

Let $\mathfrak{h}$ be a reductive subalgebra of $\mathfrak{g}$ (see Definition \ref{defn:reductive_in}).
Take a maximally split abelian subspace $\mathfrak{a}_\mathfrak{h}$ of $\mathfrak{h}$ and extend it to a maximally split abelian subspace $\mathfrak{a}$ of $\mathfrak{g}$ in a similar way as in Section \ref{subsection:Kobayashi's_criterion}.
Then the following lemma holds:

\begin{lem}\label{lem:hyperbolic_in_reductive}
A real hyperbolic orbit $\mathcal{O}^G_\hyp$ in $\mathfrak{g}$ meets $\mathfrak{h}$ if and only if it meets $\mathfrak{a}_\mathfrak{h}$.
In particular, we have a bijection 
\[
\mathcal{H}_\mathfrak{h}(\mathfrak{g})/G \rightarrow \{\, \mathcal{O}^{W(\mathfrak{g},\mathfrak{a})} \in \mathfrak{a}/{W(\mathfrak{g},\mathfrak{a})} \mid  \mathcal{O}^{W(\mathfrak{g},\mathfrak{a})} \cap  \mathfrak{a}_\mathfrak{h}  \neq \emptyset  \,\}, \quad \mathcal{O}^{G}_\hyp \mapsto \mathcal{O}^{G}_\hyp \cap \mathfrak{a},
\]
where $\mathcal{H}_\mathfrak{h}(\mathfrak{g})/G$ is the set of real hyperbolic orbits in $\mathfrak{g}$ meeting $\mathfrak{h}$.
\end{lem}

\begin{proof}[Sketch of the proof]
Suppose that $\mathcal{O}^G_\hyp$ meets $\mathfrak{h}$; we shall prove that $\mathcal{O}^G_\hyp$ meets $\mathfrak{a}_\mathfrak{h}$.
If $\mathfrak{h}$ is semisimple, 
then $\mathcal{O}^G_\hyp \cap \mathfrak{h}$ contains some hyperbolic orbits in $\mathfrak{h}$.
Hence, our claim follows by Fact \ref{fact:hyp_inter_a}. 
For the cases where $\mathfrak{h}$ is reductive in $\mathfrak{g}$ with non-trivial center $Z(\mathfrak{h})$,
we put 
\begin{align*}
Z_{\mathfrak{k}(\mathfrak{h})}(\mathfrak{h}) := Z(\mathfrak{h}) \cap \mathfrak{k}, \quad
Z_{\mathfrak{p}(\mathfrak{h})}(\mathfrak{h}) := Z(\mathfrak{h}) \cap \mathfrak{p},
\end{align*}
where $\mathfrak{g} = \mathfrak{k} + \mathfrak{p}$, $\mathfrak{h} = \mathfrak{k}(\mathfrak{h}) + \mathfrak{p}(\mathfrak{h})$ are Cartan decompositions of $\mathfrak{g}$, $\mathfrak{h}$ in Section \ref{subsection:Kobayashi's_criterion}.
Then we have
\begin{align*}
\mathfrak{h} = Z_{\mathfrak{k}(\mathfrak{h})}(\mathfrak{h}) \oplus Z_{\mathfrak{p}(\mathfrak{h})}(\mathfrak{h}) \oplus [\mathfrak{h},\mathfrak{h}].
\end{align*}
Here, we fix any $X \in \mathcal{O}^{G}_\hyp \cap \mathfrak{h}$ and put 
\[
X = X_\mathfrak{k} + X_\mathfrak{p} + X'
\]
for $X_\mathfrak{k} \in Z_{\mathfrak{k}(\mathfrak{h})}(\mathfrak{h})$, $X_\mathfrak{p} \in Z_{\mathfrak{p}(\mathfrak{h})}(\mathfrak{h})$ and $X' \in [\mathfrak{h},\mathfrak{h}]$.
Then one can prove that $X_\mathfrak{k} = 0$ and $X'$ is hyperbolic in the semisimple subalgebra $[\mathfrak{h},\mathfrak{h}]$ of $\mathfrak{g}$.
Hence our claim follows from Fact \ref{fact:hyp_inter_a}.
\end{proof}

We now prove Theorem \ref{thm:cor_to_Kobayashi} as a corollary to Fact \ref{fact:Kobayashi's_criterion}.

\begin{proof}[Proof of Theorem $\ref{thm:cor_to_Kobayashi}$]
In Setting \ref{setting:reductive}, 
by Fact \ref{fact:hyp_inter_a} and Lemma \ref{lem:hyperbolic_in_reductive}, 
we have a bijection between the following two sets:
\begin{itemize}
\item The set of $W(\mathfrak{g},\mathfrak{a})$-orbits in $\mathfrak{a}$ meeting both $\mathfrak{a}_\mathfrak{h}$ and $\mathfrak{a}_\mathfrak{l}$,
\item The set of real hyperbolic orbits in $\mathfrak{g}$ meeting both $\mathfrak{h}$ and $\mathfrak{l}$.
\end{itemize}
Hence, 
our claim follows from Fact \ref{fact:Kobayashi's_criterion}.
\end{proof}

To prove Theorem \ref{thm:cor_to_Benoist}, 
we shall show the next lemma:

\begin{lem}\label{lem:b_+_and_B-orbit}
Let $\mathfrak{g}$ be a semisimple Lie algebra.
Then a real hyperbolic orbit in $\mathfrak{g}$ is antipodal if and only if it meets $\mathfrak{b}_+$ $($see Section $\ref{subsection:Benoist_result}$ for the notation$)$.
In particular, we have a bijection
\[
\mathcal{H}^a(\mathfrak{g})/G \rightarrow \{\, \mathcal{O}^{W(\mathfrak{g},\mathfrak{a})} \in \mathfrak{a}/{W(\mathfrak{g},\mathfrak{a})} \mid  \mathcal{O}^{W(\mathfrak{g},\mathfrak{a})} \cap  \mathfrak{b}_+  \neq \emptyset  \,\}, \ \mathcal{O}^{G}_\hyp \mapsto \mathcal{O}^{G}_\hyp \cap \mathfrak{a},
\]
where $\mathcal{H}^a(\mathfrak{g})/G$ is the set of real antipodal hyperbolic orbits in $\mathfrak{g}$.
\end{lem}

\begin{proof}[Proof of Lemma $\ref{lem:b_+_and_B-orbit}$]
By Fact \ref{fact:hyp_inter_a}, any real hyperbolic orbit $\mathcal{O}^G_\hyp$ in $\mathfrak{g}$ meets $\mathfrak{a}_+$ with a unique element $A_0$ in $\mathcal{O}^G_\hyp \cap \mathfrak{a}_+$. 
It remains to prove that $-A_0$ is in $\mathcal{O}^{G}_\hyp$ if and only if $-w_0 \cdot A_0 =  A_0$.
First, we suppose that $-A_0 \in \mathcal{O}^G_\hyp$.
Then the element $-A_0$ of $-\mathfrak{a}_+$ is conjugate to $A_0$ under the action of $W(\mathfrak{g},\mathfrak{a})$ by Fact \ref{fact:hyp_inter_a}. 
Recall that both $\mathfrak{a}_+$ and $-\mathfrak{a}_+$ are fundamental domains of $\mathfrak{a}$ for the action of $W(\mathfrak{g},\mathfrak{a})$,
and $w_0 \cdot \mathfrak{a}_+ = -\mathfrak{a}_+$.
Hence, we obtain that $-w_0 \cdot A_0 = A_0$.
Conversely, we assume that $-A_0 = w_0 \cdot A_0$.
In particular, $-A_0$ is in $W(\mathfrak{g},\mathfrak{a}) \cdot A_0$.
This implies that $-A_0$ is also in $\mathcal{O}^{G}_\hyp$. 
\end{proof}

We are ready to prove Theorem \ref{thm:cor_to_Benoist}.

\begin{proof}[Proof of Theorem $\ref{thm:cor_to_Benoist}$]
In the setting of Fact \ref{fact:Benoist's_result}, by Fact \ref{fact:hyp_inter_a}, Lemma \ref{lem:hyperbolic_in_reductive} and Lemma \ref{lem:b_+_and_B-orbit},
we have a bijection between the following two sets:
\begin{itemize}
\item The set of $W(\mathfrak{g},\mathfrak{a})$-orbits in $\mathfrak{a}$ which meet $\mathfrak{b}_+$ but not $\mathfrak{a}_{\mathfrak{h}}$.
\item The set of real antipodal hyperbolic orbits in $\mathfrak{g}$ that do not meet $\mathfrak{h}$.
\end{itemize}
Hence, our claim follows from Fact \ref{fact:Benoist's_result}.
\end{proof}

Proposition \ref{prop:neutral_control} \eqref{item:neut_b_anti} is also obtained by Lemma \ref{lem:b_+_and_B-orbit} as follows:

\begin{proof}[Proof of Proposition $\ref{prop:neutral_control}$ $\eqref{item:neut_b_anti}$]
The first claim of Lemma \ref{lem:b_+_and_B-orbit}
means that 
an adjoint orbit $\mathcal{O}$ in $\mathfrak{g}$
is real antipodal hyperbolic 
if and only if
$\mathcal{O}$ is in $\Ad(G) \cdot \mathfrak{b}_+$.
Thus we have $\mathcal{H}^a(\mathfrak{g}) = \Ad(G) \cdot \mathfrak{b}_+$.
\end{proof}

\subsection{Lie group homomorphisms from $SL(2,\mathbb{R})$}\label{subsection:proof_of_SL_set}

In this subsection, we prove Proposition \ref{prop:Sl_2_proper_and_sl2_triple} by using Theorem \ref{thm:cor_to_Kobayashi}.

Let $G$ be a connected linear semisimple Lie group and write $\mathfrak{g}$ for its Lie algebra.
Then the next lemma holds:

\begin{lem}\label{lem:alg_hom_to_group_hom}
Any Lie algebra homomorphism $\phi: \mathfrak{sl}(2,\mathbb{R}) \rightarrow \mathfrak{g}$ can be uniquely lifted to $\Phi: SL(2,\mathbb{R}) \rightarrow G$ $($i.e. $\Phi$ is the Lie group homomorphism with its differential $\phi$$)$.
In particular, we have a bijection between the following two sets$:$ 
\begin{itemize}
\item The set of Lie group homomorphism from $SL(2,\mathbb{R})$ to $G$,
\item The set of $\mathfrak{sl}_2$-triples in $\mathfrak{g}$.
\end{itemize}
\end{lem}

\begin{proof}[Proof of Lemma $\ref{lem:alg_hom_to_group_hom}$]
The uniqueness follows from the connectedness of $SL(2,\mathbb{R})$.
We shall lift $\phi$.
Let us denote by 
\[
\phi_\mathbb{C} : \mathfrak{sl}(2,\mathbb{C}) \rightarrow \mathfrak{g}_\mathbb{C}
\]
the complexification of $\phi$.
Recall that $G$ is linear.
Then we can take a complexification $G_\mathbb{C}$ of $G$.
Since $SL(2,\mathbb{C})$ is simply-connected, 
the Lie algebra homomorphism $\phi_\mathbb{C}$ can be lifted to 
\[
\Phi_\mathbb{C} : SL(2,\mathbb{C}) \rightarrow G_\mathbb{C}.
\]
Then $\Phi_\mathbb{C}(SL(2,\mathbb{R}))$ is an analytic subgroup of $G_\mathbb{C}$ corresponding to the semisimple subalgebra $\phi(\mathfrak{sl}(2,\mathbb{R}))$ of $\mathfrak{g}$.
In particular, $\Phi_\mathbb{C}(SL(2,\mathbb{R}))$ is a closed subgroup of $G$.
Therefore, we can lift $\phi$ to $\Phi_\mathbb{C}|_{SL(2,\mathbb{R})}$.
\end{proof}

Let $H$ be a reductive subgroup of $G$ (see Definition \ref{defn:reductive_in}) and denote by $\mathfrak{h}$ the Lie algebra of $H$.
To prove Proposition \ref{prop:Sl_2_proper_and_sl2_triple}, 
it remains to show the following corollary to Theorem \ref{thm:cor_to_Kobayashi}:

\begin{cor}\label{cor:proper_condition_for_alg}
Let $\Phi : SL(2,\mathbb{R}) \rightarrow G$ be a Lie group homomorphism, 
and denote its differential by $\phi : \mathfrak{sl}(2,\mathbb{R}) \rightarrow \mathfrak{g}$. We put
\[
A_\phi := \phi \begin{pmatrix}1&0\\0&-1\end{pmatrix} \in \mathfrak{g}.
\]
Then $SL(2,\mathbb{R})$ acts on $G/H$ properly via $\Phi$ if and only if 
the real adjoint orbit through $A_\phi$ in $\mathfrak{g}$ does not meet $\mathfrak{h}$.
\end{cor}

\begin{proof}[Proof of Corollary $\ref{cor:proper_condition_for_alg}$]
Since $\mathfrak{sl}(2,\mathbb{R})$ is simple, 
we can assume that $\phi: \mathfrak{sl}(2,\mathbb{R}) \rightarrow \mathfrak{g}$ is injective.
We put \[
L := \Phi(SL(2,\mathbb{R})),\quad \mathfrak{l} := \phi(\mathfrak{sl}(2,\mathbb{R})).
\]
Then $L$ is a reductive subgroup of $G$ (see Example \ref{ex:ss_is_red}).
Since $\phi$ is injective and the center of $SL(2,\mathbb{R})$ is finite, 
the kernel $\Ker \Phi$ is also finite.
Therefore, 
the action of $SL(2,\mathbb{R})$ on $G/H$ via $\Phi$ is proper if and only if 
the action of $L$ on $G/H$ is proper.
By Theorem \ref{thm:cor_to_Kobayashi}, the action of $L$ on $G/H$ is proper if and only if there does not exist a real hyperbolic orbit in $\mathfrak{g}$ meeting both $\mathfrak{h}$ and $\mathfrak{l}$ apart from the zero-orbit.
Here, we take $\mathfrak{a}_\mathfrak{l} := \mathbb{R} A_\phi$ as a maximally split abelian subspace of $\mathfrak{l}$.
Then, by Lemma \ref{lem:hyperbolic_in_reductive}, for any real hyperbolic orbits in $\mathfrak{g}$, if it meets $\mathfrak{l}$ then also meets $\mathfrak{a}_\mathfrak{l}$.
Therefore, the action of $SL(2,\mathbb{R})$ on $G/H$ via $\Phi$ is proper if and only if the real adjoint orbit through $A_\phi$ in $\mathfrak{g}$ does not meet $\mathfrak{h}$.
\end{proof}

\section{Weighted Dynkin diagrams of complex adjoint orbits}\label{section:pre_result_complex_hyperbolic}\label{section:pre:WDD}

Let $\mathfrak{g}_\mathbb{C}$ be a complex semisimple Lie algebra.
In this section, we recall some well-known facts for weighted Dynkin diagrams of complex hyperbolic orbits and complex nilpotent orbits in $\mathfrak{g}_\mathbb{C}$. 
We also prove Lemma \ref{lem:sl2-triple_to_Bhyp}, and determine weighted Dynkin diagrams of complex antipodal hyperbolic orbits in $\mathfrak{g}_\mathbb{C}$.

\subsection{Weighted Dynkin diagrams of complex hyperbolic orbits}\label{subsection:WDD_of_comp_hyp}\label{subsection:WDD_hyp}

In this subsection, we recall a parameterization of complex hyperbolic orbits in $\mathfrak{g}_\mathbb{C}$ by weighted Dynkin diagrams.

Fix a Cartan subalgebra $\mathfrak{j}_\mathbb{C}$ of $\mathfrak{g}_\mathbb{C}$.
Let us denote by $\Delta(\mathfrak{g}_\mathbb{C},\mathfrak{j}_\mathbb{C})$ the root system of $(\mathfrak{g}_\mathbb{C},\mathfrak{j}_\mathbb{C})$, 
and define the real form $\mathfrak{j}$ of $\mathfrak{j}_\mathbb{C}$ by 
\[
\mathfrak{j} := \{\, A \in \mathfrak{j}_\mathbb{C} \mid \alpha (A) \in \mathbb{R} \ \text{for any } \alpha \in \Delta(\mathfrak{g}_\mathbb{C},\mathfrak{j}_\mathbb{C}) \,\}.
\]
Then $\Delta(\mathfrak{g}_\mathbb{C},\mathfrak{j}_\mathbb{C})$ can be regarded as a subset of $\mathfrak{j}^*$.
We fix a positive system $\Delta^+(\mathfrak{g}_\mathbb{C},\mathfrak{j}_\mathbb{C})$ of the root system $\Delta(\mathfrak{g}_\mathbb{C},\mathfrak{j}_\mathbb{C})$.
Then a closed Weyl chamber 
\[
\mathfrak{j}_+ := \{\, A \in \mathfrak{j} \mid \alpha (A) \geq 0 \ \text{for any } \alpha \in \Delta^+(\mathfrak{g}_\mathbb{C},\mathfrak{j}_\mathbb{C}) \,\}
\]
is a fundamental domain of $\mathfrak{j}$ for the action of the Weyl group $W(\mathfrak{g}_\mathbb{C},\mathfrak{j}_\mathbb{C})$ of $\Delta(\mathfrak{g}_\mathbb{C},\mathfrak{j}_\mathbb{C})$.

In this setting, the next fact for complex hyperbolic orbits in $\mathfrak{g}_\mathbb{C}$ is 
well known.
\begin{fact}\label{fact:cpx_hyp}
Any complex hyperbolic orbit $\mathcal{O}^{G_\mathbb{C}}_\hyp$ in $\mathfrak{g}_\mathbb{C}$ meets $\mathfrak{j}$, 
and the intersection $\mathcal{O}^{G_\mathbb{C}}_\hyp \cap \mathfrak{j}$ is a single $W(\mathfrak{g}_\mathbb{C},\mathfrak{j}_\mathbb{C})$-orbit in $\mathfrak{j}$.
In particular, we have one-to-one correspondences below$:$
\[
\mathcal{H}/{G_\mathbb{C}} \bijarrow \mathfrak{j}/{W(\mathfrak{g}_\mathbb{C},\mathfrak{j}_\mathbb{C})}  \bijarrow \mathfrak{j}_+,
\]
where $\mathcal{H}/{G_\mathbb{C}}$ is the set of complex hyperbolic orbits in $\mathfrak{g}_\mathbb{C}$ and $\mathfrak{j}/{W(\mathfrak{g}_\mathbb{C},\mathfrak{j}_\mathbb{C})}$ the set of $W(\mathfrak{g}_\mathbb{C},\mathfrak{j}_\mathbb{C})$-orbits in $\mathfrak{j}$.
\end{fact}

Let $\Pi$ denote the fundamental system of $\Delta^+(\mathfrak{g}_\mathbb{C},\mathfrak{j}_\mathbb{C})$.
Then, for any $A \in \mathfrak{j}$, we can define a map 
\[
\Psi_A : \Pi \rightarrow \mathbb{R},\ \alpha \mapsto \alpha(A).
\]
We call $\Psi_A$ the weighted Dynkin diagram corresponding to $A \in \mathfrak{j}$,
and $\alpha(A)$ the weight on a node $\alpha \in \Pi$ of the weighted Dynkin diagram.
Since $\Pi$ is a basis of $\mathfrak{j}^*$, the correspondence 
\begin{align}
\Psi: \mathfrak{j} \rightarrow \Map(\Pi,\mathbb{R}),\ A \mapsto \Psi_A \label{eq:psi}
\end{align}
is a linear isomorphism between real vector spaces. 
In particular, $\Psi$ is bijective.
Furthermore, 
\[
\Psi|_{\mathfrak{j}_+} : \mathfrak{j}_+ \rightarrow \Map(\Pi,\mathbb{R}_{\geq 0}),\ A \mapsto \Psi_A
\]
is also bijective.
We say that a weighted Dynkin diagram is trivial if all weights are zero.
Namely, the trivial diagram corresponds to the zero of $\mathfrak{j}$ by $\Psi$.

The weighted Dynkin diagram of a complex hyperbolic orbit $\mathcal{O}^{G_\mathbb{C}}_\hyp$ in $\mathfrak{g}_\mathbb{C}$ is defined as the weighted Dynkin diagram corresponding to the unique element $A_{\mathcal{O}}$ in $\mathcal{O}^{G_\mathbb{C}}_{\hyp} \cap \mathfrak{j}_+$ (see Fact \ref{fact:cpx_hyp}).
Combining Fact \ref{fact:cpx_hyp} with the bijection $\Psi|_{\mathfrak{j}_+}$, 
the map \[
\mathcal{H}/{G_\mathbb{C}} \rightarrow \Map(\Pi,\mathbb{R}_{\geq 0}),\quad \mathcal{O}^{G_\mathbb{C}}_\hyp \mapsto \Psi_{A_\mathcal{O}}
\]
is also bijective.

\subsection{Weighted Dynkin diagrams of complex antipodal hyperbolic orbits}\label{subsection:complex_Bhyp}

In this subsection, we determine complex antipodal hyperbolic orbits in $\mathfrak{g}_\mathbb{C}$ (see Definition \ref{defn:hyp_anti}) by describing the weighted Dynkin diagrams. 

We consider the same setting as in Section \ref{subsection:WDD_of_comp_hyp}.
Let us denote by $w_0^\mathbb{C}$ the longest element of $W(\mathfrak{g}_\mathbb{C},\mathfrak{j}_\mathbb{C})$ corresponding to the positive system $\Delta^+(\mathfrak{g}_\mathbb{C},\mathfrak{j}_\mathbb{C})$.
Then, by the action of $w_0^\mathbb{C}$, every element in $\mathfrak{j}_+$ moves to $-\mathfrak{j}_+ := \{ -A \mid A \in \mathfrak{j}_+ \}$.
In particular,
\[
-w_0^\mathbb{C} : \mathfrak{j} \rightarrow \mathfrak{j},\quad A \mapsto -(w_0^\mathbb{C} \cdot A)
\] 
is an involutive automorphism on $\mathfrak{j}$ preserving $\mathfrak{j}_+$.
We put 
\[
\mathfrak{j}^{-w_0^\mathbb{C}} := \{\, A \in \mathfrak{j} \mid -w_0^\mathbb{C} \cdot A = A \,\},\quad  \mathfrak{j}_+^{-w_0^\mathbb{C}} := \mathfrak{j}_+ \cap \mathfrak{j}^{-w_0^\mathbb{C}}.
\]

We recall that any complex hyperbolic orbit $\mathcal{O}^{G_\mathbb{C}}_\hyp$ in $\mathfrak{g}_\mathbb{C}$ meets $\mathfrak{j}_+$ with a unique element $A_{\mathcal{O}}$ in $\mathcal{O}^{G_\mathbb{C}}_\hyp \cap \mathfrak{j}_+$ (see Fact \ref{fact:cpx_hyp}).
Then the lemma below holds:

\begin{lem}\label{lem:minus_and_w_0C}
A complex hyperbolic orbit $\mathcal{O}^{G_\mathbb{C}}_\hyp$ in $\mathfrak{g}_\mathbb{C}$ is antipodal if and only if the corresponding element $A_{\mathcal{O}}$ is in $\mathfrak{j}_+^{-w_0^\mathbb{C}}$.
In particular, we have a one-to-one correspondence
\[
\mathcal{H}^a/{G_\mathbb{C}} \bijarrow \mathfrak{j}^{-w_0^\mathbb{C}}_+,
\]
where $\mathcal{H}^a/{G_\mathbb{C}}$ is the set of complex antipodal hyperbolic orbits in $\mathfrak{g}_\mathbb{C}$.
\end{lem}

\begin{proof}[Proof of Lemma $\ref{lem:minus_and_w_0C}$]
The proof parallels to that of Lemma \ref{lem:b_+_and_B-orbit}.
\end{proof}

Recall that the map 
\[
\Psi : \mathfrak{j} \rightarrow \Map(\Pi,\mathbb{R}),\quad A \mapsto \Phi_A
\]
is a linear isomorphism (see Section \ref{subsection:WDD_of_comp_hyp}).
Thus $-w_0^\mathbb{C}$ induces an involutive endomorphism on $\Map(\Pi,\mathbb{R})$.
By using this endomorphism, the following theorem gives a classification of complex antipodal hyperbolic orbits in $\mathfrak{g}_\mathbb{C}$.

\begin{thm}\label{thm:Bhyp_to_iota-inv}
Let $\iota$ denote the involutive endomorphism on $\Map(\Pi,\mathbb{R})$ induced by $-w_0^\mathbb{C}$.
Then the following 
holds$:$
\begin{enumerate}
\item A complex hyperbolic orbit $\mathcal{O}^{G_\mathbb{C}}_\hyp$ in $\mathfrak{g}_\mathbb{C}$ is antipodal if and only if the weighted 
Dynkin diagram of $\mathcal{O}^{G_\mathbb{C}}_\hyp$ $($see Section $\ref{subsection:WDD_of_comp_hyp}$ for the notation$)$ is held invariant by $\iota$.
In particular, we have a one-to-one correspondence 
\[
\mathcal{H}^a/{G_\mathbb{C}} \bijarrow \{\, \Psi_A \in \Map(\Pi,\mathbb{R}_{\geq 0}) \mid \text{$\Psi_A$ is held invariant by $\iota$} \,\}.
\]
\item Suppose $\mathfrak{g}_\mathbb{C}$ is simple.
Then the endomorphism $\iota$ is non-trivial if and only if $\mathfrak{g}_\mathbb{C}$ is of type $A_n$, $D_{2k+1}$ or $E_6$ $($$n \geq 2,\ k \geq 2$$)$.
In such cases, the forms of $\iota$ are$:$
\begin{description}
\item[For type $A_n$ $($$n \geq 2$, $\mathfrak{g}_\mathbb{C} \simeq \mathfrak{sl}(n+1,\mathbb{C})$$)$]
\[
\begin{xy}
	*++!D{a_1} *\cir<2pt>{}        ="A",
	(10,0) *++!D{a_2} *\cir<2pt>{} ="B",
	(30,0) *++!D{a_{n-1}} *\cir<2pt>{} ="D",
	(40,0) *++!D{a_n} *\cir<2pt>{} ="E",
	\ar@{-} "A";"B"
	\ar@{-} "B";(15,0)
	\ar@{.} (15,0);(25,0)
	\ar@{-} (25,0);"D"
	\ar@{-} "D";"E"
\end{xy}
\mapsto 
\begin{xy}
	*++!D{a_n} *\cir<2pt>{}        ="A",
	(10,0) *++!D{a_{n-1}} *\cir<2pt>{} ="B",
	(30,0) *++!D{a_2} *\cir<2pt>{} ="D",
	(40,0) *++!D{a_1} *\cir<2pt>{} ="E",
	\ar@{-} "A";"B"
	\ar@{-} "B";(15,0)
	\ar@{.} (15,0);(25,0)
	\ar@{-} (25,0);"D"
	\ar@{-} "D";"E"
\end{xy}
\]
\item [For type $D_{2k+1}$ $($$k \geq 2$, $\mathfrak{g}_\mathbb{C} \simeq \mathfrak{so}(4k+2,\mathbb{C})$$)$] 
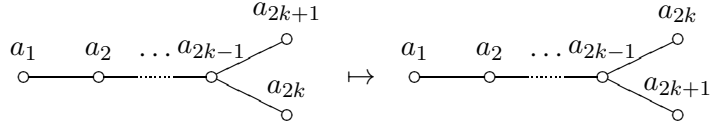
\[
\begin{xy}
	*++!D{a_1} *\cir<2pt>{}        ="A",
	(10,0) *++!D{a_2} *\cir<2pt>{} ="B",
	(25,0) *++!D{a_{2k-1}} *\cir<2pt>{} ="C",
	(35,-5) *++!D{a_{2k}} *++!U{} *\cir<2pt>{} ="D",
	(35,5) *++!D{a_{2k+1}} *\cir<2pt>{} ="E",	
	\ar@{-} "A";"B"
	\ar@{-} "B"; (15,0)
	\ar@{.} (15,0) ; (20,0)^*!U{\cdots}
	\ar@{-} (20,0) ; "C"
	\ar@{-} "C";"D"
	\ar@{-} "C";"E"
\end{xy}
\mapsto
\begin{xy}
	*++!D{a_1} *\cir<2pt>{}        ="A",
	(10,0) *++!D{a_2} *\cir<2pt>{} ="B",
	(25,0) *++!D{a_{2k-1}} *\cir<2pt>{} ="C",
	(35,-5) *++!D{a_{2k+1}} *++!U{} *\cir<2pt>{} ="D",
	(35,5) *++!D{a_{2k}} *\cir<2pt>{} ="E",	
	\ar@{-} "A";"B"
	\ar@{-} "B"; (15,0)
	\ar@{.} (15,0) ; (20,0)^*!U{\cdots}
	\ar@{-} (20,0) ; "C"
	\ar@{-} "C";"D"
	\ar@{-} "C";"E"
\end{xy}
\]
\item[For type $E_6$ $($$\mathfrak{g}_\mathbb{C} \simeq \mathfrak{e}_{6,\mathbb{C}}$$)$]
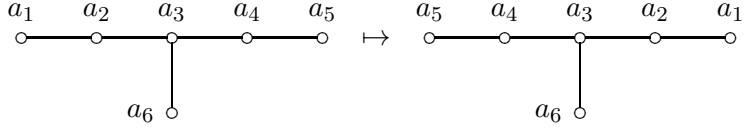
\[
\begin{xy}
	*++!D{a_1} *\cir<2pt>{}        ="A",
	(10,0) *++!D{a_2} *\cir<2pt>{} ="B",
	(20,0) *++!D{a_3} *\cir<2pt>{} ="C",
	(30,0) *++!D{a_4} *\cir<2pt>{} ="D",
	(40,0) *++!D{a_5} *\cir<2pt>{} ="E",
	(20,-10) *++!R{a_6} *\cir<2pt>{} ="F",
	\ar@{-} "A";"B"
	\ar@{-} "B";"C"
	\ar@{-} "C";"D"
	\ar@{-} "D";"E"
	\ar@{-} "C";"F"
\end{xy}
\mapsto
\begin{xy}
	*++!D{a_5} *\cir<2pt>{}        ="A",
	(10,0) *++!D{a_4} *\cir<2pt>{} ="B",
	(20,0) *++!D{a_3} *\cir<2pt>{} ="C",
	(30,0) *++!D{a_2} *\cir<2pt>{} ="D",
	(40,0) *++!D{a_1} *\cir<2pt>{} ="E",
	(20,-10) *++!R{a_6} *\cir<2pt>{} ="F",
	\ar@{-} "A";"B"
	\ar@{-} "B";"C"
	\ar@{-} "C";"D"
	\ar@{-} "D";"E"
	\ar@{-} "C";"F"
\end{xy}
\]
\end{description}
\end{enumerate}
\end{thm}

It should be noted that for the cases where $\mathfrak{g}_\mathbb{C}$ is of type $D_{2k}$ ($k \geq 2$), the involution $\iota$ on weighted Dynkin diagrams is trivial although the Dynkin diagram of type $D_{2k}$ admits some involutive automorphisms.

\begin{proof}[Proof of Theorem $\ref{thm:Bhyp_to_iota-inv}$]
The first claim of the theorem follows from Lemma \ref{lem:minus_and_w_0C}.
One can easily show that the involutive endomorphism $\iota$ on $\Map(\Pi,\mathbb{R})$ is induced by the opposition involution on the Dynkin diagram with nodes $\Pi$, which is defined by 
\[
\Pi \rightarrow \Pi,\quad \alpha \mapsto - (w_0^\mathbb{C})^* \cdot \alpha.
\]
Suppose that $\mathfrak{g}_\mathbb{C}$ is simple.
Then the root system $\Delta(\mathfrak{g}_\mathbb{C},\mathfrak{j}_\mathbb{C})$ is irreducible.
It is known that the opposition involution is non-trivial if and only if $\mathfrak{g}_\mathbb{C}$ is of type $A_n$, $D_{2k+1}$ or $E_6$ $($$n \geq 2,\ k \geq 2$$)$ (see J.~Tits \cite[Section 1.5.1]{Tits66}), and the proof is complete.
\end{proof}

As a corollary to Theorem \ref{thm:Bhyp_to_iota-inv}, we have the following:

\begin{cor}\label{cor:Bhyp_in_without_ADE}
If the complex semisimple Lie algebra $\mathfrak{g}_\mathbb{C}$ has no simple factor of type $A_n$, $D_{2k+1}$ or $E_6$ $($$n \geq 2,\ k \geq 2$$)$,
then any complex hyperbolic orbit in $\mathfrak{g}_\mathbb{C}$ is antipodal.
Namely, 
$\mathcal{H}/{G_\mathbb{C}} = \mathcal{H}^a/{G_\mathbb{C}}$.
\end{cor}

By Corollary \ref{cor:Bhyp_in_without_ADE}, in Setting \ref{setting:main}, if $\mathfrak{g}_\mathbb{C}$ has no simple factor of type $A_n$, $D_{2k+1}$ or $E_6$ $($$n \geq 2,\ k \geq 2$$)$,
then the condition \eqref{item:main_hyperbolic} in Theorem \ref{thm:main} and the condition \eqref{item:C-M_cpx_hyp_orbits} in Fact \ref{fact:C-M} are equivalent.

\subsection{Weighted Dynkin diagrams of complex nilpotent orbits}\label{subsection:WDD_of_comp_nilp}

We consider the setting in Section \ref{subsection:WDD_of_comp_hyp}, and use the notation $\mathcal{H}^n$ and $\mathcal{H}^n/{G_\mathbb{C}}$ as in Section \ref{subsection:proof_main_complexification}.
In this subsection, we prove Lemma \ref{lem:sl2-triple_to_Bhyp}, 
and recall weighted Dynkin diagrams of complex nilpotent orbits in $\mathfrak{g}_\mathbb{C}$.

First, we prove Lemma \ref{lem:sl2-triple_to_Bhyp}, which claims that $\mathcal{H}^n \subset \mathcal{H}^a$, as follows:

\begin{proof}[Proof of Lemma $\ref{lem:sl2-triple_to_Bhyp}$]
For any $\mathfrak{sl}_2$-triple $(A,X,Y)$ in $\mathfrak{g}_\mathbb{C}$, 
it is 
well known that $\ad_{\mathfrak{g}_\mathbb{C}}(A) \in \End(\mathfrak{g}_\mathbb{C})$ is diagonalizable with only real integral numbers. 
Hence, $A$ is hyperbolic in $\mathfrak{g}_\mathbb{C}$.
We shall prove that the orbit $\mathcal{O}^{G_\mathbb{C}}_A := \Ad(G_\mathbb{C}) \cdot A$ is antipodal.
We can easily check that the elements 
\[
\begin{pmatrix}1&0\\0&-1\end{pmatrix} \text{ and } \begin{pmatrix}-1&0\\0&1\end{pmatrix} \text{ in } \mathfrak{sl}(2,\mathbb{C})
\]
are conjugate under the adjoint action of $SL(2,\mathbb{C})$.
Then, for a Lie algebra homomorphism $\phi_\mathbb{C} : \mathfrak{sl}(2,\mathbb{C}) \rightarrow \mathfrak{g}_\mathbb{C}$ with 
\[
\phi_\mathbb{C} \begin{pmatrix}1&0\\0&-1\end{pmatrix} = A,
\]
the elements $A$ and $-A$ are conjugate under the adjoint action of the analytic subgroup of $G_\mathbb{C}$ corresponding to $\phi_\mathbb{C}(\mathfrak{sl}(2,\mathbb{C}))$.
Hence, the orbit $\mathcal{O}^{G_\mathbb{C}}_A$ in $\mathfrak{g}_\mathbb{C}$ is antipodal.
\end{proof}Let $\mathcal{N}$ be the set of nilpotent elements in $\mathfrak{g}_\mathbb{C}$ and $\mathcal{N}/{G_\mathbb{C}}$ the set of nilpotent orbits in $\mathfrak{g}_\mathbb{C}$.
For any $\mathfrak{sl}_2$-triple $(A,X,Y)$ in $\mathfrak{g}_\mathbb{C}$, 
the element $A$ is in $\mathcal{H}^n (\subset \mathcal{H}^a)$ and the elements $X,Y$ are both in $\mathcal{N}$.
Let us consider the map from the conjugacy classes of $\mathfrak{sl}_2$-triples in $\mathfrak{g}_\mathbb{C}$ by inner automorphisms of $\mathfrak{g}_\mathbb{C}$ to $\mathcal{N}/G_\mathbb{C}$ defined by 
\[
[(A,X,Y)] \mapsto \mathcal{O}^{G_\mathbb{C}}_X
\]
where $[(A,X,Y)]$ is the conjugacy class of an $\mathfrak{sl}_2$-triple $(A,X,Y)$ in $\mathfrak{g}_\mathbb{C}$ and $\mathcal{O}^{G_\mathbb{C}}_X$ the complex adjoint orbit through $X$ in $\mathfrak{g}_\mathbb{C}$.
Then, by the Jacobson--Morozov theorem, with a result in B.~Kostant \cite{Kostant59}, the map is bijective.
On the other hand, by A.~I.~Malcev \cite{Malcev50}, the map from the conjugacy classes of $\mathfrak{sl}_2$-triples in $\mathfrak{g}_\mathbb{C}$ by inner automorphisms of $\mathfrak{g}_\mathbb{C}$ to $\mathcal{H}^n/{G_\mathbb{C}}$ defined by 
\[
[(A,X,Y)] \mapsto \mathcal{O}^{G_\mathbb{C}}_A
\]
is also bijective, where $\mathcal{O}^{G_\mathbb{C}}_A$ is the complex adjoint orbit through $A$ in $\mathfrak{g}_\mathbb{C}$.
Therefore, we have a one-to-one correspondence 
\[
\mathcal{N}/{G_\mathbb{C}} \bijarrow \mathcal{H}^n/{G_\mathbb{C}}.
\]
In particular, by combining the argument above with Fact \ref{fact:cpx_hyp}, 
we also obtain a bijection: 
\[
\mathcal{N}/{G_\mathbb{C}} \rightarrow \mathfrak{j}_+ \cap \mathcal{H}^n, \quad \mathcal{O}^{G_\mathbb{C}}_\nilp \mapsto A_\mathcal{O},
\]
where $A_\mathcal{O}$ is the unique element of $\mathfrak{j}_+$ with the property:
there exist $X,Y \in \mathcal{O}^{G_\mathbb{C}}_\nilp$ such that $(A_\mathcal{O},X,Y)$ is an $\mathfrak{sl}_2$-triple in $\mathfrak{g}_\mathbb{C}$. 

\begin{rem}
It is known that the Jacobson--Morozov theorem and the result of Kostant in \cite{Kostant59} also hold for any real semisimple Lie algebra $\mathfrak{g}$.
Therefore, we have a surjective map from the set of real nilpotent orbits in $\mathfrak{g}$ to $\mathcal{H}^n(\mathfrak{g})/{G}$, where $\mathcal{H}^n(\mathfrak{g})/{G}$ is the notation in Section $\ref{subsection:proof_main_complexification}$.
However, in general, the map is not injective.
\end{rem}

The weighted Dynkin diagram of a complex nilpotent orbit $\mathcal{O}^{G_\mathbb{C}}_\nilp$ in $\mathfrak{g}_\mathbb{C}$ is defined as the weighted Dynkin diagram corresponding to $A_{\mathcal{O}} \in \mathfrak{j}_+ \cap \mathcal{H}^n$.
Obviously, the weighted Dynkin diagram of $\mathcal{O}^{G_\mathbb{C}}_\nilp$ is the same as the weighted Dynkin diagram of the corresponding orbit in $\mathcal{H}^n/G_\mathbb{C}$.

E.~B.~Dynkin \cite{Dynkin52eng} proved that any weight of a weighted Dynkin diagram of any complex adjoint orbit in $\mathcal{H}^n/{G_\mathbb{C}}$ is $0$, $1$ or $2$.
Hence, $\mathcal{H}^n/{G_\mathbb{C}}$ is (and therefore $\mathcal{N}/{G_\mathbb{C}}$ is) finite.
Dynkin \cite{Dynkin52eng} gave a list of the weighted Dynkin diagrams of $\mathcal{H}^n/{G_\mathbb{C}}$ as the classification of $\mathfrak{sl}_2$-triples in $\mathfrak{g}_\mathbb{C}$.
This also gives a classification of complex nilpotent orbits in $\mathfrak{g}_\mathbb{C}$ (see Bala--Cater \cite{Bala-Carter76} or Collingwood--McGovern \cite[Section 3]{Collingwood-McGovern93} for more details).

We remark that by combining Theorem \ref{thm:Bhyp_to_iota-inv} with Lemma \ref{lem:sl2-triple_to_Bhyp},
if $\mathfrak{g}_\mathbb{C}$ is isomorphic to $\mathfrak{sl}(n+1,\mathbb{C})$, $\mathfrak{so}(4k+2,\mathbb{C})$ or $\mathfrak{e}_{6,\mathbb{C}}$ $($$n \geq 2,\ k \geq 2$$)$, 
then the weighted Dynkin diagram of any complex adjoint orbit in $\mathcal{H}^n/{G_\mathbb{C}}$ (and therefore the weighted Dynkin diagram of any complex nilpotent orbit) is invariant under the non-trivial involution $\iota$.

\begin{example}\label{ex:list_of_nilp_in_sl6}
It is known that there exists a bijection between complex nilpotent orbits in $\mathfrak{sl}(n,\mathbb{C})$ and partitions of $n$ $($see \cite[Section 3.1 and 3.6]{Collingwood-McGovern93}$)$.
Here is the list of weighted Dynkin diagrams of complex nilpotent orbits in $\mathfrak{sl}(6,\mathbb{C})$ $($i.e. the list of weighted Dynkin diagrams corresponding to $\mathfrak{j}_+ \cap \mathcal{H}^n$ for the case where $\mathfrak{g}_\mathbb{C} = \mathfrak{sl}(6,\mathbb{C})$$)$$:$
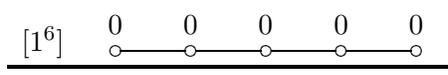
\begin{table}[htbp]
\begin{center}
	\begin{tabular}{ll} \toprule
	Partition & Weighted Dynkin diagram \\ \midrule
	$[6]$ & \begin{xy}
	*++!D{2} *\cir<2pt>{}        ="A",
	(10,0) *++!D{2} *\cir<2pt>{} ="B",
	(20,0) *++!D{2} *\cir<2pt>{} ="C",
	(30,0) *++!D{2} *\cir<2pt>{} ="D",
	(40,0) *++!D{2} *\cir<2pt>{} ="E",
	\ar@{-} "A";"B"
	\ar@{-} "B";"C"
	\ar@{-} "C";"D"
	\ar@{-} "D";"E"
\end{xy} \\ \hline
	$[5,1]$ & \begin{xy}
	*++!D{2} *\cir<2pt>{}        ="A",
	(10,0) *++!D{2} *\cir<2pt>{} ="B",
	(20,0) *++!D{0} *\cir<2pt>{} ="C",
	(30,0) *++!D{2} *\cir<2pt>{} ="D",
	(40,0) *++!D{2} *\cir<2pt>{} ="E",
	\ar@{-} "A";"B"
	\ar@{-} "B";"C"
	\ar@{-} "C";"D"
	\ar@{-} "D";"E"
\end{xy} \\ \hline
	$[4,2]$ & \begin{xy}
	*++!D{2} *\cir<2pt>{}        ="A",
	(10,0) *++!D{0} *\cir<2pt>{} ="B",
	(20,0) *++!D{2} *\cir<2pt>{} ="C",
	(30,0) *++!D{0} *\cir<2pt>{} ="D",
	(40,0) *++!D{2} *\cir<2pt>{} ="E",
	\ar@{-} "A";"B"
	\ar@{-} "B";"C"
	\ar@{-} "C";"D"
	\ar@{-} "D";"E"
\end{xy} \\ \hline
	$[4,1^2]$ & \begin{xy}
	*++!D{2} *\cir<2pt>{}        ="A",
	(10,0) *++!D{1} *\cir<2pt>{} ="B",
	(20,0) *++!D{0} *\cir<2pt>{} ="C",
	(30,0) *++!D{1} *\cir<2pt>{} ="D",
	(40,0) *++!D{2} *\cir<2pt>{} ="E",
	\ar@{-} "A";"B"
	\ar@{-} "B";"C"
	\ar@{-} "C";"D"
	\ar@{-} "D";"E"
\end{xy} \\ \hline
	$[3^2]$ & \begin{xy}
	*++!D{0} *\cir<2pt>{}        ="A",
	(10,0) *++!D{2} *\cir<2pt>{} ="B",
	(20,0) *++!D{0} *\cir<2pt>{} ="C",
	(30,0) *++!D{2} *\cir<2pt>{} ="D",
	(40,0) *++!D{0} *\cir<2pt>{} ="E",
	\ar@{-} "A";"B"
	\ar@{-} "B";"C"
	\ar@{-} "C";"D"
	\ar@{-} "D";"E"
\end{xy} \\ \hline
	$[3,2,1]$ & \begin{xy}
	*++!D{1} *\cir<2pt>{}        ="A",
	(10,0) *++!D{1} *\cir<2pt>{} ="B",
	(20,0) *++!D{0} *\cir<2pt>{} ="C",
	(30,0) *++!D{1} *\cir<2pt>{} ="D",
	(40,0) *++!D{1} *\cir<2pt>{} ="E",
	\ar@{-} "A";"B"
	\ar@{-} "B";"C"
	\ar@{-} "C";"D"
	\ar@{-} "D";"E"
\end{xy} \\ \hline
	$[3,1^3]$ & \begin{xy}
	*++!D{2} *\cir<2pt>{}        ="A",
	(10,0) *++!D{0} *\cir<2pt>{} ="B",
	(20,0) *++!D{0} *\cir<2pt>{} ="C",
	(30,0) *++!D{0} *\cir<2pt>{} ="D",
	(40,0) *++!D{2} *\cir<2pt>{} ="E",
	\ar@{-} "A";"B"
	\ar@{-} "B";"C"
	\ar@{-} "C";"D"
	\ar@{-} "D";"E"
\end{xy} \\ \hline
	$[2^3]$ & \begin{xy}
	*++!D{0} *\cir<2pt>{}        ="A",
	(10,0) *++!D{0} *\cir<2pt>{} ="B",
	(20,0) *++!D{2} *\cir<2pt>{} ="C",
	(30,0) *++!D{0} *\cir<2pt>{} ="D",
	(40,0) *++!D{0} *\cir<2pt>{} ="E",
	\ar@{-} "A";"B"
	\ar@{-} "B";"C"
	\ar@{-} "C";"D"
	\ar@{-} "D";"E"
\end{xy} \\ \hline
	$[2^2,1^2]$ & \begin{xy}
	*++!D{0} *\cir<2pt>{}        ="A",
	(10,0) *++!D{1} *\cir<2pt>{} ="B",
	(20,0) *++!D{0} *\cir<2pt>{} ="C",
	(30,0) *++!D{1} *\cir<2pt>{} ="D",
	(40,0) *++!D{0} *\cir<2pt>{} ="E",
	\ar@{-} "A";"B"
	\ar@{-} "B";"C"
	\ar@{-} "C";"D"
	\ar@{-} "D";"E"
\end{xy} \\ \hline
	$[2,1^4]$ & \begin{xy}
	*++!D{1} *\cir<2pt>{}        ="A",
	(10,0) *++!D{0} *\cir<2pt>{} ="B",
	(20,0) *++!D{0} *\cir<2pt>{} ="C",
	(30,0) *++!D{0} *\cir<2pt>{} ="D",
	(40,0) *++!D{1} *\cir<2pt>{} ="E",
	\ar@{-} "A";"B"
	\ar@{-} "B";"C"
	\ar@{-} "C";"D"
	\ar@{-} "D";"E"
\end{xy} \\ \hline
	\end{tabular}
	\end{center}
	\end{table}
	\begin{table}[htbp]
	\begin{center}
	\begin{tabular}{ll} \hline
	$[1^6]$ & \begin{xy}
	*++!D{0} *\cir<2pt>{}        ="A",
	(10,0) *++!D{0} *\cir<2pt>{} ="B",
	(20,0) *++!D{0} *\cir<2pt>{} ="C",
	(30,0) *++!D{0} *\cir<2pt>{} ="D",
	(40,0) *++!D{0} *\cir<2pt>{}  ="E",
	\ar@{-} "A";"B"
	\ar@{-} "B";"C"
	\ar@{-} "C";"D"
	\ar@{-} "D";"E"
\end{xy} \\
	\bottomrule 
	\end{tabular}
\legend{Classification of complex nilpotent orbits in $\mathfrak{sl}(6,\mathbb{C})$}
\end{center}
\end{table}
\end{example}

\newpage

\section{Complex adjoint orbits and real forms}\label{section:Real_forms}

Let $\mathfrak{g}_\mathbb{C}$ be a complex simple Lie algebra, and $\mathfrak{g}$ a real form of $\mathfrak{g}_\mathbb{C}$. 
Recall that, in Section \ref{section:pre:WDD}, we have a parameterization of complex hyperbolic [resp. antipodal hyperbolic, nilpotent] orbits in $\mathfrak{g}_\mathbb{C}$ by weighted Dynkin diagrams.
In this section, we also determine complex hyperbolic [resp. antipodal hyperbolic, nilpotent] orbits in $\mathfrak{g}_\mathbb{C}$ meeting $\mathfrak{g}$.
For this, we give an algorithm to check whether or not a given complex hyperbolic [resp. nilpotent] orbit in $\mathfrak{g}_\mathbb{C}$ meets $\mathfrak{g}$.
We also prove Proposition \ref{prop:Hyp_in_gC_in_g} and Proposition \ref{prop:neutral_control} \eqref{item:neut_key} in this section.

\subsection{Complex hyperbolic orbits and real forms}\label{subsection:hyp_real}
We give a proof of Proposition \ref{prop:Hyp_in_gC_in_g} \eqref{item:Hyp_c_r:hyp} in this subsection.

We fix a Cartan decomposition $\mathfrak{g} = \mathfrak{k} + \mathfrak{p}$,
and use the following convention:

\begin{defn}\label{defn:split_Cartan}
We say that a Cartan subalgebra $\mathfrak{j}_\mathfrak{g}$ of $\mathfrak{g}$ is split if $\mathfrak{a} := \mathfrak{j}_\mathfrak{g} \cap \mathfrak{p}$ is a maximal abelian subspace of $\mathfrak{p}$ $($i.e. $\mathfrak{a}$ is a maximally split abelian subspace of $\mathfrak{g}$$)$.
\end{defn}

Note that such $\mathfrak{j}_\mathfrak{g}$ is unique up to the adjoint action of $K$, where $K$ is the analytic subgroup of $G$ corresponding to $\mathfrak{k}$.

Take a split Cartan subalgebra $\mathfrak{j}_\mathfrak{g}$ of $\mathfrak{g}$ in Definition \ref{defn:split_Cartan}.
Then $\mathfrak{j}_\mathfrak{g}$ can be written as $\mathfrak{j}_\mathfrak{g} = \mathfrak{t} + \mathfrak{a}$ for a maximal abelian subspace $\mathfrak{t}$ of the centralizer of $\mathfrak{a}$ in $\mathfrak{k}$.
Let us denote by $\mathfrak{j}_\mathbb{C} := \mathfrak{j}_\mathfrak{g} + \sqrt{-1}\mathfrak{j}_\mathfrak{g}$ and $\mathfrak{j} := \sqrt{-1} \mathfrak{t} + \mathfrak{a}$.
Then $\mathfrak{j}_\mathbb{C}$ is a Cartan subalgebra of $\mathfrak{g}_\mathbb{C}$ and $\mathfrak{j}$ is a real form of it, with 
\[
\mathfrak{j} = \{ A \in \mathfrak{j}_\mathbb{C} \mid \alpha(A) \in \mathbb{R} \ \text{for any } \alpha \in \Delta(\mathfrak{g}_\mathbb{C},\mathfrak{j}_\mathbb{C}) \},
\]
where $\Delta(\mathfrak{g}_\mathbb{C},\mathfrak{j}_\mathbb{C})$ is the root system of $(\mathfrak{g}_\mathbb{C},\mathfrak{j}_\mathbb{C})$. 
We put \[
\Sigma(\mathfrak{g},\mathfrak{a}) := \{ \alpha|_\mathfrak{a} \mid \alpha \in \Delta(\mathfrak{g}_\mathbb{C},\mathfrak{j}_\mathbb{C}) \} \setminus \{0\} \subset \mathfrak{a}^*
\] to the restricted root system of $(\mathfrak{g},\mathfrak{a})$.
Then we can take a positive system $\Delta^+(\mathfrak{g}_\mathbb{C},\mathfrak{j}_\mathbb{C})$ of $\Delta(\mathfrak{g}_\mathbb{C},\mathfrak{j}_\mathbb{C})$ such that the subset 
\[
\Sigma^+(\mathfrak{g},\mathfrak{a}) := 
\{ \alpha|_\mathfrak{a} \mid \alpha \in \Delta^+(\mathfrak{g}_\mathbb{C},\mathfrak{j}_\mathbb{C}) \} \setminus \{0\}.
\]
of $\Sigma(\mathfrak{g},\mathfrak{a})$ becomes a positive system.
In fact, if we take an ordering on $\mathfrak{a}$ and extend it to $\mathfrak{j}$, then the corresponding positive system $\Delta^+(\mathfrak{g}_\mathbb{C},\mathfrak{j}_\mathbb{C})$ satisfies the condition above.
Let us denote by $W(\mathfrak{g}_\mathbb{C},\mathfrak{j}_\mathbb{C})$, $W(\mathfrak{g},\mathfrak{a})$ the Weyl groups of $\Delta(\mathfrak{g}_\mathbb{C}, \mathfrak{j}_\mathbb{C})$, $\Sigma(\mathfrak{g},\mathfrak{a})$, respectively.
We put the closed Weyl chambers 
\begin{align*}
\mathfrak{j}_+ &:= \{\, A \in \mathfrak{j} \mid \alpha(A) \geq 0 \ \text{for any } \alpha \in \Delta^+(\mathfrak{g}_\mathbb{C},\mathfrak{j}_\mathbb{C}) \,\}, \\
\mathfrak{a}_+ &:= \{\, A \in \mathfrak{a} \mid \xi(A) \geq 0 \ \text{for any } \xi \in \Sigma^+(\mathfrak{g},\mathfrak{a}) \,\}.
\end{align*}
Then $\mathfrak{j}_+$ and $\mathfrak{a}_+$ are fundamental domains of $\mathfrak{j}$, $\mathfrak{a}$ for the actions of $W(\mathfrak{g}_\mathbb{C},\mathfrak{j}_\mathbb{C})$ and $W(\mathfrak{g},\mathfrak{a})$, respectively.
By the definition of $\Delta^+(\mathfrak{g}_\mathbb{C},\mathfrak{j}_\mathbb{C})$ and $\Sigma^+(\mathfrak{g},\mathfrak{a})$, we have $\mathfrak{a}_+ = \mathfrak{j}_+ \cap \mathfrak{a}$.

We recall that any complex hyperbolic orbit $\mathcal{O}^{G_\mathbb{C}}_\hyp$ in $\mathfrak{g}_\mathbb{C}$ meets $\mathfrak{j}_+$ with a unique element $A_{\mathcal{O}}$ in $\mathcal{O}^{G_\mathbb{C}}_\hyp \cap \mathfrak{j}_+$ (see Fact \ref{fact:cpx_hyp}).
Then the lemma below holds:

\begin{lem}\label{lem:comp_hyp_meets_g_and_A_in_a}
A complex hyperbolic orbit $\mathcal{O}^{G_\mathbb{C}}_\hyp$ in $\mathfrak{g}_\mathbb{C}$ meets $\mathfrak{g}$ if and only if the corresponding element $A_{\mathcal{O}}$ is in $\mathfrak{a}_+$.
In particular, we have a one-to-one correspondence 
\[
\mathcal{H}_\mathfrak{g}/{G_\mathbb{C}} \bijarrow \mathfrak{a}_+,
\]
where $\mathcal{H}_\mathfrak{g}/{G_\mathbb{C}}$ is the set of complex hyperbolic orbits in $\mathfrak{g}_\mathbb{C}$ meeting $\mathfrak{g}$.
\end{lem}

Lemma \ref{lem:comp_hyp_meets_g_and_A_in_a} will be used in Section \ref{subsection:Satake_diagrams} to prove Theorem \ref{thm:hyp_match}.
We now prove Proposition \ref{prop:Hyp_in_gC_in_g} \eqref{item:Hyp_c_r:hyp} and Lemma \ref{lem:comp_hyp_meets_g_and_A_in_a} simultaneously.

\begin{proof}[Proof of Proposition $\ref{prop:Hyp_in_gC_in_g}$ $\eqref{item:Hyp_c_r:hyp}$ and Lemma $\ref{lem:comp_hyp_meets_g_and_A_in_a}$]
We show that for a complex hyperbolic orbit $\mathcal{O}^{G_\mathbb{C}}_\hyp$ in $\mathfrak{g}_\mathbb{C}$, the element $A_{\mathcal{O}}$ is in $\mathfrak{a}_+$ if $\mathcal{O}^{G_\mathbb{C}}_\hyp$ meets $\mathfrak{g}$.
Note that an element of $\mathfrak{g}$ is hyperbolic in $\mathfrak{g}$ (see Definition \ref{defn:hyp_anti}) if and only if hyperbolic in $\mathfrak{g}_\mathbb{C}$.
Thus any real adjoint orbit $\mathcal{O}'$ contained in $\mathcal{O}^{G_\mathbb{C}}_\hyp \cap \mathfrak{g}$ is hyperbolic, and hence $\mathcal{O}'$ meets $\mathfrak{a}_+$ with a unique element $A_0 \in \mathcal{O}' \cap \mathfrak{a}_+$ by Fact \ref{fact:hyp_inter_a}. 
Since $\mathfrak{a}_+$ is contained in $\mathfrak{j}_+$,
the element $A_0$ is in $\mathcal{O}^{G_\mathbb{C}}_\hyp \cap \mathfrak{j}_+$.
Thus, $A_0 = A_{\mathcal{O}}$.
Therefore, we obtain that $A_{\mathcal{O}}$ is in $\mathfrak{a}_+$ for any $\mathcal{O}^{G_\mathbb{C}}_\hyp \in \mathcal{H}_\mathfrak{g}/{G_\mathbb{C}}$, 
which completes the proof of Lemma \ref{lem:comp_hyp_meets_g_and_A_in_a}.

To prove Proposition \ref{prop:Hyp_in_gC_in_g} \eqref{item:Hyp_c_r:hyp}, it suffices to show that the intersection $\mathcal{O}^{G_\mathbb{C}}_\hyp \cap \mathfrak{g}$ becomes a single adjoint orbit.
By the argument above, we have 
\[
\Ad(G) \cdot A_{\mathcal{O}} = \mathcal{O}^{G_\mathbb{C}}_\hyp \cap \mathfrak{g},
\]
and hence Proposition \ref{prop:Hyp_in_gC_in_g} \eqref{item:Hyp_c_r:hyp} follows.
\end{proof}

\subsection{Weighted Dynkin diagrams and Satake diagrams}\label{subsection:Satake_diagrams}\label{subsection:WDD_and_Satake}

Let us consider the setting in Section \ref{subsection:hyp_real}.
In this subsection, 
we determine complex hyperbolic orbits in $\mathfrak{g}_\mathbb{C}$ meeting $\mathfrak{g}$ by using the Satake diagram of $\mathfrak{g}$.

First, we recall briefly the definition of the Satake diagram of the real form $\mathfrak{g}$ of $\mathfrak{g}_\mathbb{C}$ (see \cite{Araki62,Satake60} for more details).
Let us denote by $\Pi$ the fundamental system of $\Delta^+(\mathfrak{g}_\mathbb{C},\mathfrak{j}_\mathbb{C})$.
Then 
\[
\overline{\Pi} := \{\, \alpha|_\mathfrak{a} \mid \alpha \in \Pi \,\} \setminus \{0\}
\] is the fundamental system of $\Sigma^+(\mathfrak{g},\mathfrak{a})$.
We write $\Pi_0$ for the set of all simple roots in $\Pi$ whose restriction to $\mathfrak{a}$ is zero.
The Satake diagram $S_\mathfrak{g}$ of $\mathfrak{g}$ consists of the following data: the Dynkin diagram of $\mathfrak{g}_\mathbb{C}$ with nodes $\Pi$; black nodes $\Pi_0$ in $S$; and arrows joining $\alpha \in \Pi \setminus \Pi_0$ and $\beta \in \Pi \setminus \Pi_0$ in $S$ whose restrictions to $\mathfrak{a}$ are the same.

Second, we give the definition of weighted Dynkin diagrams \textit{matching} the Satake diagram $S_\mathfrak{g}$ of $\mathfrak{g}$ as follows:

\begin{defn}\label{defn:match}
Let $\Psi_A \in \Map(\Pi,\mathbb{R})$ be a weighted Dynkin diagram of $\mathfrak{g}_\mathbb{C}$ $($see Section $\ref{subsection:WDD_of_comp_hyp}$ for the notation$)$ and $S_\mathfrak{g}$ the Satake diagram of $\mathfrak{g}$ with nodes $\Pi$.
We say that $\Psi_A$ \textit{matches} $S_\mathfrak{g}$ if all the weights on black nodes in $\Pi_0$ are zero and any pair of nodes joined by an arrow have the same weights.
\end{defn}

Then the following theorem holds:

\begin{thm}\label{thm:hyp_match}
The weighted Dynkin diagram of a complex hyperbolic orbit $\mathcal{O}^{G_\mathbb{C}}_\hyp$ in $\mathfrak{g}_\mathbb{C}$ matches the Satake diagram of $\mathfrak{g}$ if and only if $\mathcal{O}^{G_\mathbb{C}}_\hyp$ meets $\mathfrak{g}$.
In particular, we have a one-to-one correspondence 
\[
\mathcal{H}_\mathfrak{g}/{G_\mathbb{C}} \bijarrow \{\, \Psi_A \in \Map(\Pi,\mathbb{R}_{\geq 0}) \mid \text{$\Psi_A$ matches $S_\mathfrak{g}$} \,\}.
\]
\end{thm}

Recall that $\Psi$ is a linear isomorphism from $\mathfrak{j}$ to $\Map(\Pi,\mathbb{R})$ (see \eqref{eq:psi} in Section \ref{subsection:WDD_of_comp_hyp} for the notation), 
and there exists a one-to-one correspondence between $\mathcal{H}_\mathfrak{g}/{G_\mathbb{C}}$ and $\mathfrak{a}_+$ (see Lemma \ref{lem:comp_hyp_meets_g_and_A_in_a}).
Therefore, to prove Theorem \ref{thm:hyp_match}, it suffices to show the next lemma:

\begin{lem}\label{lem:a_to_Satake}
The linear isomorphism $\Psi:\mathfrak{j} \rightarrow \Map(\Pi,\mathbb{R})$ induces a linear isomorphism  
\[
\mathfrak{a} \rightarrow \{\, \Psi_A \in \Map(\Pi,\mathbb{R}) \mid \text{$\Psi_A$ matches $S_\mathfrak{g}$} \,\},\quad A \mapsto \Psi_A.
\]
\end{lem}

\begin{proof}[Proof of Lemma $\ref{lem:a_to_Satake}$]
Let $A \in \mathfrak{j}$.
By Definition \ref{defn:match}, the weighted Dynkin diagram $\Psi_A$ matches the Satake diagram of $\mathfrak{g}$ if and only if $A$ satisfies the following condition $(\star)$:
\[
(\star)  \begin{cases}
\alpha(A) = 0 \quad (\text{for any } \alpha \in \Pi_0), \\
\alpha(A) = \beta(A) \quad (\text{for any } \alpha, \beta \in \Pi \setminus \Pi_0 \text{ with } \alpha|_\mathfrak{a} = \beta|_\mathfrak{a}).
\end{cases}
\]
Thus, it suffices to show that the subspace
\[
\mathfrak{a}' := \{\, A \in \mathfrak{j} \mid \text{ $A$ satisfies the condition $(\star)$}\,\}
\]
of $\mathfrak{j}$ coincides with $\mathfrak{a}$.
It is easy to check that $\mathfrak{a} \subset \mathfrak{a}'$. 
We now prove that $\dim_\mathbb{R} \mathfrak{a} = \dim_\mathbb{R} \mathfrak{a}'$.
Recall that $\overline{\Pi}$ is a fundamental system of $\Sigma^+(\mathfrak{g},\mathfrak{a})$. 
In particular, $\overline{\Pi}$ is a basis of $\mathfrak{a}^*$.
Thus, $\dim_\mathbb{R} \mathfrak{a} = \sharp \overline{\Pi}$.
We define the element $A'_\xi$ of $\mathfrak{a}'$ for each $\xi \in \overline{\Pi}$ by
\begin{align*}
\alpha(A'_\xi) = 
\begin{cases} 
1 \quad (\text{if $\alpha|_\mathfrak{a} = \xi$}), \\
0 \quad (\text{if $\alpha|_\mathfrak{a} \neq \xi$}),
\end{cases}
\end{align*}
for any $\alpha \in \Pi$.
Then $\{\, A'_\xi \mid \xi \in \overline{\Pi} \,\}$ is a basis of $\mathfrak{a}'$ since 
\[
\overline{\Pi} = \{\, \alpha|_\mathfrak{a} \mid \alpha \in \Pi \,\} \setminus \{0\}.
\]
Thus, $\dim_\mathbb{R} \mathfrak{a}' = \sharp \overline{\Pi}$, 
and hence $\mathfrak{a} = \mathfrak{a}'$.
\end{proof}

\subsection{Complex antipodal hyperbolic orbits and real forms}\label{subsection:antipodal_real}

We consider the setting in Section \ref{subsection:hyp_real} and \ref{subsection:Satake_diagrams}.
In this subsection, the proof of Proposition \ref{prop:Hyp_in_gC_in_g} \eqref{item:Hyp_c_r:anti} is given. 
Concerning to the proof of Proposition \ref{prop:keyprop} \eqref{item:neut_key}, which will be given in Section \ref{subsection:proof_of_keyprop}, 
we also 
determine the subset $\mathfrak{b}$ of $\mathfrak{a}$ (see Section \ref{subsection:Benoist_result} for the notation) by describing the weighted Dynkin diagrams in this subsection.

First, we prove Proposition \ref{prop:Hyp_in_gC_in_g} \eqref{item:Hyp_c_r:anti}, which gives a bijection between complex antipodal hyperbolic orbits in $\mathfrak{g}_\mathbb{C}$ meeting $\mathfrak{g}$ and real antipodal hyperbolic orbits in $\mathfrak{g}$, as follows:

\begin{proof}[Proof of Proposition $\ref{prop:Hyp_in_gC_in_g}$ $\eqref{item:Hyp_c_r:anti}$]
Note that Proposition $\ref{prop:Hyp_in_gC_in_g}$ $\eqref{item:Hyp_c_r:hyp}$ has been already proved in Section \ref{subsection:hyp_real}.
Therefore, to prove Proposition $\ref{prop:Hyp_in_gC_in_g}$ $\eqref{item:Hyp_c_r:anti}$, it remains to show that for any $\mathcal{O}^{G_\mathbb{C}} \in \mathcal{H}^a_\mathfrak{g}/{G_\mathbb{C}}$ and any element $A$ of $\mathcal{O}^{G_\mathbb{C}} \cap \mathfrak{g}$, the element $-A$ is also in $\mathcal{O}^{G_\mathbb{C}} \cap \mathfrak{g}$.
Since $\mathcal{O}^{G_\mathbb{C}}$ is antipodal, the element $-A$ is also in $\mathcal{O}^{G_\mathbb{C}}$.
Hence, we have $-A \in \mathcal{O}^{G_\mathbb{C}} \cap \mathfrak{g}$.
\end{proof}

Recall that we have bijections between $\mathcal{H}^a/{G_\mathbb{C}}$ and $\mathfrak{j}^{-w_0^\mathbb{C}}_+$ (see Lemma \ref{lem:minus_and_w_0C}) and between $\mathcal{H}^a(\mathfrak{g})/{G}$ and $\mathfrak{b}_+$ (see Lemma \ref{lem:b_+_and_B-orbit}).
By Proposition \ref{prop:Hyp_in_gC_in_g} \eqref{item:Hyp_c_r:anti}, which has been proved above, we have one-to-one correspondences \[
\mathfrak{b}_+ \bijarrow \mathcal{H}^a(\mathfrak{g})/{G} \bijarrow \mathcal{H}^a_\mathfrak{g}/{G_\mathbb{C}},
\]
where $\mathcal{H}^a_\mathfrak{g}/{G_\mathbb{C}}$ is the set of complex antipodal hyperbolic orbits in $\mathfrak{g}_\mathbb{C}$ meeting $\mathfrak{g}$.

To explain the relation between $\mathfrak{j}^{-w_0^\mathbb{C}}_+$ and $\mathfrak{b}_+$, we show the following lemma:

\begin{lem}\label{lem:b=jwcapa}
Let $w_0^\mathbb{C}$, $w_0$ be the longest elements of $W(\mathfrak{g}_\mathbb{C},\mathfrak{j}_\mathbb{C})$, $W(\mathfrak{g},\mathfrak{a})$ with respect to the positive systems $\Delta^+(\mathfrak{g}_\mathbb{C},\mathfrak{j}_\mathbb{C})$, $\Sigma^+(\mathfrak{g},\mathfrak{a})$, respectively.
Then: 
\begin{align*}
\mathfrak{b} = \mathfrak{j}^{-w_0^\mathbb{C}} \cap \mathfrak{a}, \quad 
\mathfrak{b}_+ = \mathfrak{j}^{-w_0^\mathbb{C}}_+ \cap \mathfrak{a}, 
\end{align*}
where $\mathfrak{b} = \{\, A \in \mathfrak{a} \mid -w_0 \cdot A = A \,\}$ and $\mathfrak{j}^{-w_0^\mathbb{C}} = \{\, A \in \mathfrak{j} \mid -w_0^\mathbb{C} \cdot A = A \,\}$.
\end{lem}

\begin{proof}[Proof of Lemma $\ref{lem:b=jwcapa}$]
We only need to show that $w_0^\mathbb{C}$ preserves $\mathfrak{a}$ and the action on $\mathfrak{a}$ is same as $w_0$.
Let us put $\tau$ to the complex conjugation on $\mathfrak{g}_\mathbb{C}$ with respect to the real form $\mathfrak{g}$.
Then we can easily check that both $\Pi$ and $-\Pi$ are $\tau$-fundamental systems of $\Delta(\mathfrak{g}_\mathbb{C},\mathfrak{j}_\mathbb{C})$ in the sense of \cite[Section 1.1]{Satake60}.
Since $(w_0^\mathbb{C})^* \cdot \Pi = -\Pi$, the endomorphism $w_0^\mathbb{C}$ is commutative with $\tau$ on $\mathfrak{j}$, and $w_0^\mathbb{C}$ induces on $\mathfrak{a}$ an element $w'_0$ of $W(\mathfrak{g},\mathfrak{a})$ by \cite[Proposition A]{Satake60}.
Recall that $\overline{\Pi} = \{\, \alpha|_\mathfrak{a} \mid \alpha \in \Pi \,\}$.
Then we have $(w'_0)^* \cdot \overline{\Pi} = -\overline{\Pi}$,
and hence $w'_0 = w_0$.
\end{proof}

Recall that we have a bijection between $\mathfrak{a}$ and the set of weighted Dynkin diagrams matching the Satake diagram of $\mathfrak{g}$ (see Lemma \ref{lem:a_to_Satake}).
Combining with Lemma \ref{lem:b=jwcapa}, 
we have a linear isomorphism
\begin{align*}
\mathfrak{b} &\rightarrow \{ \Psi_A \in \Map(\Pi,\mathbb{R}) \mid \text{$\Psi_A$ is held invariant by $\iota$ and matches $S_\mathfrak{g}$}\},\\ A &\mapsto \Psi_A,
\end{align*}
where $\iota$ is the involutive endomorphism on $\Map(\Pi,\mathbb{R})$ defined in Section \ref{subsection:complex_Bhyp}.
Therefore, we can determine the subsets $\mathfrak{b}$ and $\mathfrak{b}_+$ of $\mathfrak{a}$.
Here is an example of the isomorphism for the case where $\mathfrak{g} = \mathfrak{su}(4,2)$. 
\begin{example}\label{ex:b_for_su42}
Let $\mathfrak{g} = \mathfrak{su}(4,2)$.
Then the complexification of $\mathfrak{su}(4,2)$ is $\mathfrak{g}_\mathbb{C} = \mathfrak{sl}(6,\mathbb{C})$, and the involutive endomorphism $\iota$ on weighted Dynkin diagrams is described by 
\[
\begin{xy}
	*++!D{a} *\cir<2pt>{}        ="A",
	(10,0) *++!D{b} *\cir<2pt>{} ="B",
	(20,0) *++!D{c} *\cir<2pt>{} ="C",
	(30,0) *++!D{d} *\cir<2pt>{} ="D",
	(40,0) *++!D{e} *\cir<2pt>{} ="E",
	\ar@{-} "A";"B"
	\ar@{-} "B";"C"
	\ar@{-} "C";"D"
	\ar@{-} "D";"E"
\end{xy}
\mapsto 
\begin{xy}
	*++!D{e} *\cir<2pt>{}        ="A",
	(10,0) *++!D{d} *\cir<2pt>{} ="B",
	(20,0) *++!D{c} *\cir<2pt>{} ="C",
	(30,0) *++!D{b} *\cir<2pt>{} ="D",
	(40,0) *++!D{a} *\cir<2pt>{} ="E",
	\ar@{-} "A";"B"
	\ar@{-} "B";"C"
	\ar@{-} "C";"D"
	\ar@{-} "D";"E"
\end{xy}.
\]
The Satake diagram of $\mathfrak{g} = \mathfrak{su}(4,2)$ is here$:$
\[
S_{\mathfrak{su}(4,2)} : \begin{xy}
	*\cir<2pt>{}        ="A",
	(10,0) *\cir<2pt>{} ="B",
	(20,0) *{\bullet} ="C",
	(30,0) *\cir<2pt>{} ="D",
	(40,0) *\cir<2pt>{} ="E",
	\ar@{-} "A";"B"
	\ar@{-} "B";"C"
	\ar@{-} "C";"D"
	\ar@{-} "D";"E"
	
	\ar@(ru,lu) @<1mm> @{<->} "A" ; "E"
	\ar@/^4mm/ @<1mm> @{<->} "B" ; "D"	
\end{xy}.
\]
Therefore, we have a linear isomorphism
\[
\mathfrak{b} \simrightarrow \left\{
\begin{xy}
	*++!D{a} *\cir<2pt>{}        ="A",
	(10,0) *++!D{b} *\cir<2pt>{} ="B",
	(20,0) *++!D{0} *\cir<2pt>{} ="C",
	(30,0) *++!D{b} *\cir<2pt>{} ="D",
	(40,0) *++!D{a} *\cir<2pt>{} ="E",
	\ar@{-} "A";"B"
	\ar@{-} "B";"C"
	\ar@{-} "C";"D"
	\ar@{-} "D";"E"
\end{xy}
\mid 
a,b \in \mathbb{R}
\right\}.
\]
In particular, we have one-to-one correspondences below$:$
\[
\mathcal{H}^a_\mathfrak{g}/{G_\mathbb{C}} \bijarrow \mathfrak{b}_+ \bijarrow \left\{
\begin{xy}
	*++!D{a} *\cir<2pt>{}        ="A",
	(10,0) *++!D{b} *\cir<2pt>{} ="B",
	(20,0) *++!D{0} *\cir<2pt>{} ="C",
	(30,0) *++!D{b} *\cir<2pt>{} ="D",
	(40,0) *++!D{a} *\cir<2pt>{} ="E",
	\ar@{-} "A";"B"
	\ar@{-} "B";"C"
	\ar@{-} "C";"D"
	\ar@{-} "D";"E"
\end{xy}
\mid 
a,b \in \mathbb{R}_{\geq 0}
\right\}.
\]
\end{example}

\subsection{Complex nilpotent orbits and real forms}\label{subsection:nilpotent_matching}

Let us consider the setting in Section \ref{subsection:hyp_real} and \ref{subsection:WDD_and_Satake}.
In this subsection, we introduce an algorithm to check whether or not a given complex nilpotent orbit in $\mathfrak{g}_\mathbb{C}$ meets the real form $\mathfrak{g}$.
In this subsection, we also prove Proposition \ref{prop:Hyp_in_gC_in_g} \eqref{item:Hyp_c_r:neut}.

First, we show the next proposition:

\begin{prop}[Corollary to J.~Sekiguchi {\cite[Proposition 1.11]{Sekiguchi87}}]\label{prop:cor_to_Sekiguchi}
Let $(A,X,Y)$ be an $\mathfrak{sl}_2$-triple in $\mathfrak{g}_\mathbb{C}$.
Then the following conditions on $(A,X,Y)$ are equivalent$:$
\begin{enumerate}
\item \label{item:Sekiguchi_nilp} The complex adjoint orbit through $X$ meets $\mathfrak{g}$.
\item \label{item:Sekiguchi_hyp} The complex adjoint orbit through $A$ meets $\mathfrak{g}$.
\item \label{item:Sekiguchi_nilp_pc} The complex adjoint orbit through $X$ meets $\mathfrak{p}_\mathbb{C}$, where $\mathfrak{p}_\mathbb{C}$ is the complexification of $\mathfrak{p}$.
\item \label{item:Sekiguchi_hyp_pc} The complex adjoint orbit through $A$ meets $\mathfrak{p}_\mathbb{C}$.
\item \label{item:Sekiguchi_sl2} There exists an $\mathfrak{sl}_2$-triple $(A',X',Y')$ in $\mathfrak{g}$ such that $A'$ is in the complex adjoint orbit through $A$.
\item \label{item:Sekiguchi_match} The weighted Dynkin diagram of the complex adjoint orbit through $X$ matches the Satake diagram of $\mathfrak{g}$.
\end{enumerate}
\end{prop}

\begin{proof}[Proof of Proposition $\ref{prop:cor_to_Sekiguchi}$]
The equivalences between \eqref{item:Sekiguchi_nilp}, \eqref{item:Sekiguchi_nilp_pc} and \eqref{item:Sekiguchi_hyp_pc} were proved by \cite[Proposition~ 1.11]{Sekiguchi87}.
The equivalence \eqref{item:Sekiguchi_hyp_pc} $\Leftrightarrow$ \eqref{item:Sekiguchi_hyp} is obtained by the fact that $\mathcal{H}_\mathfrak{g} = \mathcal{H}_\mathfrak{a} = \mathcal{H}_{\mathfrak{p}_\mathbb{C}}$ (cf. Lemma \ref{lem:comp_hyp_meets_g_and_A_in_a} and the proof of \cite[Proposition~ 1.11]{Sekiguchi87}).
The equivalence \eqref{item:Sekiguchi_hyp} $\Leftrightarrow$ \eqref{item:Sekiguchi_match} is obtained by combining Theorem \ref{thm:hyp_match} with the observation that the weighted Dynkin diagrams of the complex adjoint orbit through $X$ is same as the weighted Dynkin diagram of the complex adjoint orbit through $A$ (see Section \ref{subsection:WDD_of_comp_nilp}).
The implication \eqref{item:Sekiguchi_hyp} $\Rightarrow$ \eqref{item:Sekiguchi_sl2} can be obtained by the lemma below.
\end{proof}

\begin{lem}\label{lem:sekiguchi_sub}
Let $(A,X,Y)$ be an $\mathfrak{sl}_2$-triple in $\mathfrak{g}_\mathbb{C}$.
Then the following 
holds$:$
\begin{enumerate}
\item \label{item:Sekiguchi_sub1} If $A$ is in $\mathfrak{g}$, then there exists $g \in G_\mathbb{C}$ such that $\Ad(g) \cdot A = A$ and $\Ad(g) \cdot X$ is in $\mathfrak{g}$.
\item \label{item:Sekiguchi_sub2} If both $A$ and $X$ are in $\mathfrak{g}$, then $Y$ is automatically in $\mathfrak{g}$.
\end{enumerate}
\end{lem}

\begin{proof}[Proof of Lemma $\ref{lem:sekiguchi_sub}$]
\eqref{item:Sekiguchi_sub1}: See the proof of \cite[Proposition 1.11]{Sekiguchi87}.
\eqref{item:Sekiguchi_sub2}: Easy.
\end{proof}

Here is a proof of Proposition \ref{prop:Hyp_in_gC_in_g} \eqref{item:Hyp_c_r:neut}, which gives a bijection between $\mathcal{H}^n_\mathfrak{g}/{G_\mathbb{C}}$ and $\mathcal{H}^n(\mathfrak{g})/{G}$ (see Section \ref{subsection:proof_main_complexification} for the notation):

\begin{proof}[Proof of Proposition $\ref{prop:Hyp_in_gC_in_g}$ $\eqref{item:Hyp_c_r:neut}$]
We recall that Proposition $\ref{prop:Hyp_in_gC_in_g}$ $\eqref{item:Hyp_c_r:hyp}$ has been proved already in Section \ref{subsection:hyp_real}.
Then Proposition $\ref{prop:Hyp_in_gC_in_g}$ $\eqref{item:Hyp_c_r:neut}$ follows from the implication \eqref{item:Sekiguchi_hyp} $\Rightarrow$ \eqref{item:Sekiguchi_sl2} in Proposition \ref{prop:cor_to_Sekiguchi}.
\end{proof}

Recall that we have the one-to-one correspondence
\[
\mathfrak{j}_+ \cap \mathcal{H}^n \bijarrow \mathcal{N}/{G_\mathbb{C}},
\]
where $\mathcal{N}/{G_\mathbb{C}}$ is the set of complex nilpotent orbits in $\mathfrak{g}_\mathbb{C}$ (see Section \ref{subsection:WDD_of_comp_nilp}).
Combining Lemma \ref{lem:comp_hyp_meets_g_and_A_in_a} with Proposition \ref{prop:cor_to_Sekiguchi}, we also obtain
\[
\mathfrak{a}_+ \cap \mathcal{H}^n(\mathfrak{g}) = (\mathfrak{j}_+ \cap \mathcal{H}^n) \cap \mathfrak{a} \bijarrow \mathcal{N}_\mathfrak{g}/{G_\mathbb{C}},
\]
where $\mathcal{N}_\mathfrak{g}/{G_\mathbb{C}}$ is the set of complex nilpotent orbits in $\mathfrak{g}_\mathbb{C}$ meeting $\mathfrak{g}$.
Therefore, by Lemma \ref{lem:a_to_Satake}, we obtain the theorem below:

\begin{thm}\label{thm:nilp_match}
Let $\mathfrak{g}_\mathbb{C}$ be a complex semisimple Lie algebra, and $\mathfrak{g}$ a real form of $\mathfrak{g}_\mathbb{C}$.
Then for a complex nilpotent orbit $\mathcal{O}^{G_\mathbb{C}}_\nilp$ in $\mathfrak{g}_\mathbb{C}$, the following two conditions are equivalent$:$
\begin{enumerate}
\item \label{item:nilp_match_O} 
$\mathcal{O}^{G_\mathbb{C}}_\nilp \cap \mathfrak{g} \neq \emptyset$ 
$($i.e. $\mathcal{O}^{G_\mathbb{C}}_\nilp \in \mathcal{N}_\mathfrak{g}/{G_\mathbb{C}}$$)$. 
\item \label{item:nilp_match_match} The weighted Dynkin diagram of $\mathcal{O}^{G_\mathbb{C}}_\nilp$ matches the Satake diagram $S_\mathfrak{g}$ of $\mathfrak{g}$ $($see Section $\ref{subsection:Satake_diagrams}$ for the notation$)$.
\end{enumerate}
\end{thm}

\begin{rem}
\begin{description}
\item[(1)] The same concept as Definition $\ref{defn:match}$ appeared earlier as ``weighted Satake diagrams'' in D.~Z.~Djokovic \cite{Djokovic82} and as the condition described in J.~Sekiguchi \cite[\text{Proposition 1.16}]{Sekiguchi84}. We call it ``match''.
\item[(2)] J.~Sekiguchi \cite[\text{Proposition 1.13}]{Sekiguchi87} showed the implication \eqref{item:nilp_match_match} $\Rightarrow$ \eqref{item:nilp_match_O} in Theorem $\ref{thm:nilp_match}$.
Our theorem claims that \eqref{item:nilp_match_O} $\Rightarrow$ \eqref{item:nilp_match_match} is also true.
\end{description}
\end{rem}

We give three examples of Theorem \ref{thm:nilp_match}:

\begin{example}
Let $\mathfrak{g}$ be a split real form of $\mathfrak{g}_\mathbb{C}$.
Then all nodes of the Satake diagram $S_\mathfrak{g}$ are white with no arrow.
Thus, all weighted Dynkin diagrams match the Satake diagram of $\mathfrak{g}$. 
Therefore, all complex nilpotent orbits in $\mathfrak{g}_\mathbb{C}$ meet $\mathfrak{g}$. 
\end{example}

\begin{example}
Let $\mathfrak{u}$ be a compact real form of $\mathfrak{g}_\mathbb{C}$.
Then all nodes of the Satake diagram $S_\mathfrak{u}$ are black.
Thus, no weighted Dynkin diagram matches the Satake diagram of $\mathfrak{u}$ except for the trivial one.
Therefore, no complex nilpotent orbit in $\mathfrak{g}_\mathbb{C}$ meets $\mathfrak{u}$ except for the zero-orbit.
\end{example}

\begin{example}\label{ex:nilp_in_sl6_meets_su42}
Let $(\mathfrak{g}_\mathbb{C},\mathfrak{g}) = (\mathfrak{sl}(6,\mathbb{C}),\mathfrak{su}(4,2))$.
The Satake diagram of $\mathfrak{su}(4,2)$ was given in Example $\ref{ex:b_for_su42}$.
Then, by combining with Example $\ref{ex:list_of_nilp_in_sl6}$, we obtain the list of complex nilpotent orbits in $\mathfrak{g}_\mathbb{C}$ meeting $\mathfrak{g}$ $($i.e. the list of $(\mathfrak{j}_+ \cap \mathcal{H}^n) \cap \mathfrak{a} $$)$ as follows$:$
\begin{multline*}
\mathcal{N}_\mathfrak{g}/{G_\mathbb{C}} 
\bijarrow \{\, [5,1],\ [4,1^2],\ [3^2],\ [3,2,1],\ [3,1^3],\ [2^2,1^2],\ [2,1^4],\ [1^6] \,\}.
\end{multline*}
\end{example}

\subsection{Proof of Proposition $\ref{prop:neutral_control}$ $\eqref{item:neut_key}$} \label{subsection:proof_of_keyprop}

In this subsection,
we first explain the strategy of the proof of Proposition $\ref{prop:neutral_control}$ $\eqref{item:neut_key}$,
and then 
illustrate actual computations by an example.

By Lemma \ref{lem:b_+_and_B-orbit}, we have 
\[
\mathfrak{b}_+ \supset \mathfrak{a}_+ \cap \mathcal{H}^n(\mathfrak{g}).
\]
Furthermore, in Section \ref{subsection:nilpotent_matching}, we also obtained
\[
\mathfrak{a}_+ \cap \mathcal{H}^n(\mathfrak{g}) = (\mathfrak{j}_+ \cap \mathcal{H}^n) \cap \mathfrak{a}.
\]
Therefore, the proof of Proposition $\ref{prop:neutral_control}$ $\eqref{item:neut_key}$ is reduced to the showing 
\begin{equation}
\mathfrak{b} \subset \mathbb{R}\text{-}\mathrm{span} ((\mathfrak{j}_+ \cap \mathcal{H}^n) \cap \mathfrak{a}) \label{eq:check_to_keyprop}
\end{equation}
for 
all simple Lie algebras $\mathfrak{g}$.

In order to show \eqref{eq:check_to_keyprop},
we recall that 
the Dynkin--Kostant classification of weighted Dynkin diagrams corresponding to elements of $\mathfrak{j}_+ \cap \mathcal{H}^n$ (which gives a classification of complex nilpotent orbits in $\mathfrak{g}_\mathbb{C}$; see Section \ref{subsection:WDD_of_comp_nilp})
As its subset,
we can classify the weighted Dynkin diagrams corresponding to elements in $(\mathfrak{j}_+ \cap \mathcal{H}^n) \cap \mathfrak{a}$ by using the Satake diagram of $\mathfrak{g}$ (cf. Example \ref{ex:nilp_in_sl6_meets_su42}).
What we need to prove for \eqref{eq:check_to_keyprop} is that 
this subset contains sufficiently many 
in the sense that 
the $\mathbb{R}$-span of the weighted Dynkin diagrams corresponding to this subset 
is coincide with 
the space of weighted Dynkin diagrams 
corresponding to elements in $\mathfrak{b}$.
Recall that  we can also determine such space corresponding to $\mathfrak{b}$ 
by the involution $\iota$ on weighted Dynkin diagrams (see Section \ref{subsection:complex_Bhyp} for the notation) with the Satake diagram of $\mathfrak{g}$ 
(cf. Example \ref{ex:b_for_su42}).

We illustrate this strategy by the following example:

\begin{example}
We give a proof of Proposition $\ref{prop:neutral_control}$ $\eqref{item:neut_key}$ for the case where $\mathfrak{g} = \mathfrak{su}(4,2)$, 
with its complexification $\mathfrak{g}_\mathbb{C} = \mathfrak{sl}(6,\mathbb{C})$. 

By Example $\ref{ex:nilp_in_sl6_meets_su42}$, we have the list of weighted Dynkin diagrams corresponding to elements of $(\mathfrak{j}_+ \cap \mathcal{H}^n) \cap \mathfrak{a}$ for $\mathfrak{g} = \mathfrak{su}(4,2)$. 
Here is a part of it$:$

\begin{table}[htbp]
\begin{center}
	\begin{tabular}{ll} \toprule
	Partition & Weighted Dynkin diagram \\ \midrule
	$[2^2,1^2]$ & \begin{xy}
	*++!D{0} *\cir<2pt>{}        ="A",
	(10,0) *++!D{1} *\cir<2pt>{} ="B",
	(20,0) *++!D{0} *\cir<2pt>{} ="C",
	(30,0) *++!D{1} *\cir<2pt>{} ="D",
	(40,0) *++!D{0} *\cir<2pt>{} ="E",
	\ar@{-} "A";"B"
	\ar@{-} "B";"C"
	\ar@{-} "C";"D"
	\ar@{-} "D";"E"
\end{xy} \\ \hline
	$[2,1^4]$ & \begin{xy}
	*++!D{1} *\cir<2pt>{}        ="A",
	(10,0) *++!D{0} *\cir<2pt>{} ="B",
	(20,0) *++!D{0} *\cir<2pt>{} ="C",
	(30,0) *++!D{0} *\cir<2pt>{} ="D",
	(40,0) *++!D{1} *\cir<2pt>{} ="E",
	\ar@{-} "A";"B"
	\ar@{-} "B";"C"
	\ar@{-} "C";"D"
	\ar@{-} "D";"E"
\end{xy} \\ 
	\bottomrule 
	\end{tabular}
\legend{A part of $(\mathfrak{j}_+ \cap \mathcal{H}^n) \cap \mathfrak{a}$ for $\mathfrak{g} = \mathfrak{su}(4,2)$}
\end{center}
\end{table}

By Example $\ref{ex:b_for_su42}$, we also have a linear isomorphism 
\[
\mathfrak{b} \simrightarrow \left\{
\begin{xy}
	*++!D{a} *\cir<2pt>{}        ="A",
	(10,0) *++!D{b} *\cir<2pt>{} ="B",
	(20,0) *++!D{0} *\cir<2pt>{} ="C",
	(30,0) *++!D{b} *\cir<2pt>{} ="D",
	(40,0) *++!D{a} *\cir<2pt>{} ="E",
	\ar@{-} "A";"B"
	\ar@{-} "B";"C"
	\ar@{-} "C";"D"
	\ar@{-} "D";"E"
\end{xy}
\mid 
a,b \in \mathbb{R}
\right\}.
\]
Hence, we can observe that 
\[
\mathfrak{b} \subset \mathbb{R}\text{-}\mathrm{span} ((\mathfrak{j}_+ \cap \mathcal{H}^n) \cap \mathfrak{a}).
\]
This completes the proof of Proposition $\ref{prop:neutral_control}$ $\eqref{item:neut_key}$ for the case where $\mathfrak{g} = \mathfrak{su}(4,2)$.
\end{example}

For the other simple Lie algebras $\mathfrak{g}$, 
we can find the Satake diagram of $\mathfrak{g}$ in \cite{Araki62} or \cite[Chapter X, Section 6]{Helgason01book} and the classification of weighted Dynkin diagrams of complex nilpotent orbits in $\mathfrak{g}_\mathbb{C}$ in \cite{Bala-Carter76}.
Then we can verify \eqref{eq:check_to_keyprop} 
in the spirit of case-by-case computations for other real simple Lie algebras.
Detailed computations will be reported elsewhere.

\section{Symmetric pairs}\label{section:Symmetric_pair}

In this section, we prove Proposition \ref{prop:keyprop} and Lemma \ref{lem:a_h}. 

Let $(\mathfrak{g},\mathfrak{h})$ be a semisimple symmetric pair and write $\sigma$ for the involution on $\mathfrak{g}$ corresponding to $\mathfrak{h}$. 
First, we give Cartan decompositions on $\mathfrak{g}$, $\mathfrak{h}$ and $\mathfrak{g}^c$ (see \eqref{eq:c-dual} in Section \ref{section:main_result} for the notation), simultaneously.

Recall that 
we can find a Cartan involution $\theta$ on $\mathfrak{g}$ 
with $\sigma \theta = \theta \sigma$ 
(cf.~\cite{Berger57}).
Let us denote by $\mathfrak{g} = \mathfrak{k} + \mathfrak{p}$ and 
$\mathfrak{h} = \mathfrak{k}(\mathfrak{h}) + \mathfrak{p}(\mathfrak{h})$
the Cartan decompositions of $\mathfrak{g}$ and $\mathfrak{h}$, respectively. 
We set $\mathfrak{u} := \mathfrak{k} + \sqrt{-1}\mathfrak{p}$.
Then $\mathfrak{u}$ becomes a compact real form of $\mathfrak{g}_\mathbb{C}$. 
We write $\tau$, $\tau^c$ for the complex conjugations on $\mathfrak{g}_\mathbb{C}$ with respect to the real forms $\mathfrak{g}$, $\mathfrak{g}^c$, respectively.
Then $\tau^c$ is the anti $\mathbb{C}$-linear extension of $\sigma$ on $\mathfrak{g}$ to $\mathfrak{g}_\mathbb{C}$, 
and hence $\tau$ and $\tau^c$ are commutative.
The compact real form $\mathfrak{u}$ of $\mathfrak{g}_\mathbb{C}$ 
is stable under both $\tau$ and $\tau^c$.
We denote by $\overline{\theta}$ the complex conjugation on $\mathfrak{g}_\mathbb{C}$ corresponding to $\mathfrak{u}$,
i.e.~$\overline{\theta}$ is anti $\mathbb{C}$-linear extension of $\theta$.
Then the restriction $\overline{\theta}|_{\mathfrak{g}^c}$ is a Cartan involution on $\mathfrak{g}^c$.
We write
\begin{equation}
\mathfrak{g}^c = \mathfrak{k}^c + \mathfrak{p}^c \label{eq:CD_of_c-dual}
\end{equation} 
for 
the Cartan decomposition of $\mathfrak{g}^c$
with respect to $\overline{\theta}|_{\mathfrak{g}^c}$.

Let us fix a maximal abelian subspace $\mathfrak{a}_\mathfrak{h}$ of $\mathfrak{p}(\mathfrak{h})$,
and extend it to a maximal abelian subspace $\mathfrak{a}$ of $\mathfrak{p}$ 
[resp. a maximal abelian subspace $\mathfrak{a}^c$ of $\mathfrak{p}^c$].
Obviously, $\mathfrak{a}_\mathfrak{h} = \mathfrak{a} \cap \mathfrak{a}^c$.
We show the next lemma below:
\begin{lem}\label{lem:commutative_extend}
$[\mathfrak{a},\mathfrak{a}^c] = \{0\}$.
\end{lem}

The next proposition gives a Cartan subalgebra of $\mathfrak{g}_\mathbb{C}$ which contains split Cartan subalgebras of $\mathfrak{g}$, $\mathfrak{g}^c$ and $\mathfrak{h}$ with respect to the Cartan decompositions.

\begin{prop}\label{prop:OS-Cartan}
There exists a Cartan subalgebra $\mathfrak{j}_\mathbb{C}$ of $\mathfrak{g}_\mathbb{C}$ with the following properties$:$
\begin{itemize}
\item $\mathfrak{j}_\mathfrak{g} := \mathfrak{j}_\mathbb{C} \cap \mathfrak{g}$ is a split Cartan subalgebra of $\mathfrak{g} = \mathfrak{k} + \mathfrak{p}$ $($see Definition $\ref{defn:split_Cartan}$ for the notation$)$
with $\mathfrak{j}_\mathfrak{g} \cap \mathfrak{p} = \mathfrak{a}$.
\item $\mathfrak{j}_{\mathfrak{g}^c} := \mathfrak{j}_\mathbb{C} \cap \mathfrak{g}^c$ is a split Cartan subalgebra of $\mathfrak{g}^c = \mathfrak{k}^c + \mathfrak{p}^c$
with $\mathfrak{j}_{\mathfrak{g}^c} \cap \mathfrak{p}^c = \mathfrak{a}^c$.
\item $\mathfrak{j}_{\mathfrak{h}} := \mathfrak{j}_\mathbb{C} \cap \mathfrak{h}$ is a split Cartan subalgebra of $\mathfrak{h} = \mathfrak{k}(\mathfrak{h}) + \mathfrak{p}(\mathfrak{h})$
with $\mathfrak{j}_\mathfrak{h} \cap \mathfrak{p}(\mathfrak{h}) = \mathfrak{a}_\mathfrak{h}$.
\end{itemize}
\end{prop}

\begin{proof}[Proof of Lemma $\ref{lem:commutative_extend}$ and Proposition $\ref{prop:OS-Cartan}$]
We put
\[
\mathfrak{h}^a := \{\, X \in \mathfrak{g} \mid \theta \sigma X = X  \,\}, \quad \mathfrak{q}^a := \{\, X \in \mathfrak{g} \mid \theta \sigma X = -X  \,\}.
\]
Then $(\mathfrak{g},\mathfrak{h}^a)$ is the associated symmetric pair of $(\mathfrak{g},\mathfrak{h})$ (see \cite[Section 1]{Oshima-Sekiguchi84} for the notation).
Note that 
$\mathfrak{q}^a = \mathfrak{p}^c \cap \mathfrak{g} + \sqrt{-1}(\mathfrak{p}^c \cap \sqrt{-1}\mathfrak{g})$ 
and 
$\mathfrak{p} \cap \mathfrak{q}^a = \mathfrak{p}(\mathfrak{h})$.
Let us apply \cite[Lemma 2.4 (i)]{Oshima-Sekiguchi84} 
to the symmetric pair $(\mathfrak{g},\mathfrak{h}^a)$. 
Then we have $[\mathfrak{a},\mathfrak{a}^c] = \{0\}$,
since the complexification of $\mathfrak{a}_c$ 
is a maximal abelian subspace of 
the complexification of $\mathfrak{q}^a$ 
containing $\mathfrak{a}_\mathfrak{h}$.
This completes the proof of Lemma \ref{lem:commutative_extend}.
Furthermore, 
let us extend $\mathfrak{a} + \mathfrak{a}^c$ to 
a Cartan subalgebra $\mathfrak{j}_\mathbb{C}$ of $\mathfrak{g}_\mathbb{C}$.
Then $\mathfrak{j}_\mathbb{C}$ satisfies the properties in Proposition \ref{prop:OS-Cartan}.
\end{proof}

We fix such a Cartan subalgebra $\mathfrak{j}_\mathbb{C}$ of $\mathfrak{g}_\mathbb{C}$, and put
\begin{align*}
\mathfrak{j}   := \mathfrak{j}_\mathbb{C} \cap \sqrt{-1}\mathfrak{u}.
\end{align*}
Throughout this subsection, we denote the root system of $(\mathfrak{g}_\mathbb{C},\mathfrak{j}_\mathbb{C})$ 
briefly by $\Delta$, which is realized in $\mathfrak{j}^*$.
Let us denote by $\Sigma$, $\Sigma^c$ the restricted root systems of $(\mathfrak{g},\mathfrak{a})$, $(\mathfrak{g}^c,\mathfrak{a}^c)$, respectively.
Namely, we put
\begin{align*}
\Sigma &:= \{\, \alpha|_\mathfrak{a} \mid \alpha \in \Delta \,\} \setminus \{0\} \subset \mathfrak{a}^*, \\
\Sigma^c &:= \{\, \alpha|_{\mathfrak{a}^c} \mid \alpha \in \Delta \,\} \setminus \{0\} \subset (\mathfrak{a}^c)^*.
\end{align*}
Then we can choose a positive system $\Delta^+$ of $\Delta$ with the properties below$:$
\begin{itemize}
\item $\Sigma^+ := \{\, \alpha|_\mathfrak{a} \mid \alpha \in \Delta^+ \,\} \setminus \{0\}$ is a positive system of $\Sigma$.
\item $(\Sigma^c)^+ := \{\, \alpha|_{\mathfrak{a}^c} \mid \alpha \in \Delta^+ \,\} \setminus \{0\}$ is a positive system of $\Sigma^c$.
\end{itemize}
In fact, if we take an ordering on $\mathfrak{a}_\mathfrak{h}$ and extend it stepwise to $\mathfrak{a}$, to $\mathfrak{a} + \mathfrak{a}^c$ and to $\mathfrak{j}$, then the corresponding positive system $\Delta^+$ satisfies the properties above (see \cite[Section 3]{Oshima-Sekiguchi84} for more detail).
Let us denote by 
\begin{align*}
\mathfrak{j}_+ &:= \{\, A \in \mathfrak{j} \mid \alpha(A) \geq 0 \ \text{for any } \alpha \in \Delta^+ \,\}, \\
\mathfrak{a}_+ &:=  \{\, A \in \mathfrak{a} \mid \xi(A) \geq 0 \ \text{for any } \xi \in \Sigma^+ \,\}, \\
\mathfrak{a}^c_+ &:=  \{\, A \in \mathfrak{a}^c \mid \xi^c(A) \geq 0 \ \text{for any } \xi^c \in (\Sigma^c)^+ \,\},
\end{align*}
the closed Weyl chambers in $\mathfrak{j}$, $\mathfrak{a}$ and $\mathfrak{a}^c$ with respect to $\Delta^+$, $\Sigma^+$ and $(\Sigma^c)^+$, respectively. 

Combining Fact \ref{fact:cpx_hyp} with Lemma \ref{lem:comp_hyp_meets_g_and_A_in_a}, 
we obtain the lemma below$:$

\begin{lem}\label{lem:hyp_element_in_Cartan_for_symm}
Let $\mathcal{O}^{G_\mathbb{C}}_\hyp$ be a complex hyperbolic orbit in $\mathfrak{g}_\mathbb{C}$.
Then the following 
holds$:$
\begin{enumerate}
\item There exists a unique element $A_{\mathcal{O}}$ in $\mathcal{O}^{G_\mathbb{C}}_\hyp \cap \mathfrak{j}_+$.
\item $\mathcal{O}^{G_\mathbb{C}}_\hyp$ meets $\mathfrak{g}$ if and only if $A_{\mathcal{O}}$ is in $\mathfrak{a}_+$.
\item $\mathcal{O}^{G_\mathbb{C}}_\hyp$ meets $\mathfrak{g}^c$ if and only if $A_{\mathcal{O}}$ is in $\mathfrak{a}^c_+$.
\end{enumerate}
\end{lem}

We now prove Proposition \ref{prop:keyprop} by using Lemma \ref{lem:hyp_element_in_Cartan_for_symm}.

\begin{proof}[Proof of Proposition $\ref{prop:keyprop}$]
Let $\mathcal{O}^{G_\mathbb{C}}_\hyp$ be a complex hyperbolic orbit in $\mathfrak{g}_\mathbb{C}$ meeting both $\mathfrak{g}$ and $\mathfrak{g}^c$.
We shall prove that $\mathcal{O}^{G_\mathbb{C}}_\hyp$ also meets $\mathfrak{h} = \mathfrak{g} \cap \mathfrak{g}^c$.
By Lemma \ref{lem:hyp_element_in_Cartan_for_symm}, there exists a unique element $A_{\mathcal{O}} \in \mathfrak{a}_+ \cap \mathfrak{a}^c_+$ with $A_\mathcal{O} \in \mathcal{O}^{G_\mathbb{C}}_\hyp$,
and hence our claim follows.
\end{proof}

Lemma \ref{lem:a_h} is proved by using Lemma \ref{lem:hyp_element_in_Cartan_for_symm} as follows.

\begin{proof}[Proof of Lemma $\ref{lem:a_h}$]
Let us take $A \in \mathfrak{a}_+$ such that $\mathcal{O}^G_A$ meets $\mathfrak{h}$,
where $\mathcal{O}^G_A$ is the adjoint orbit in $\mathfrak{g}$ through $A$.
To prove our claim,
we only need to show that 
$A$ is in $\mathfrak{a}_\mathfrak{h}$.
We denote by $\mathcal{O}^{G_\mathbb{C}}_A$ the complexification of $\mathcal{O}^G_A$.
Then $\mathcal{O}^{G_\mathbb{C}}_A$ is a complex hyperbolic orbit in $\mathfrak{g}_\mathbb{C}$ meeting $\mathfrak{h} = \mathfrak{g} \cap \mathfrak{g}^c$.
Let us extend $\mathfrak{a}_\mathfrak{h}$ to a maximal abelian subspace $\mathfrak{a}^c$ of $\mathfrak{p}^c$ 
(see \eqref{eq:CD_of_c-dual} for the notation of $\mathfrak{p}^c$)
and 
take a Cartan subalgebra $\mathfrak{j}_\mathbb{C}$ of $\mathfrak{g}_\mathbb{C}$ in Proposition \ref{prop:OS-Cartan}.
We also extend the ordering on $\mathfrak{a}$ stepwise to $\mathfrak{a} + \mathfrak{a}^c$ and to $\mathfrak{j}$.
Then by Lemma \ref{lem:hyp_element_in_Cartan_for_symm}, 
the orbit $\mathcal{O}^{G_\mathbb{C}}_A$ intersects $\mathfrak{j}_+$ 
with a unique element $A_\mathcal{O}$, 
and $A_\mathcal{O}$ is in $\mathfrak{a}_+ \cap \mathfrak{a}^c_+ \subset \mathfrak{a}_\mathfrak{h}$.
Since $A$ is also in $\mathcal{O}^{G_\mathbb{C}}_A \cap \mathfrak{j}_+$, 
we have $A = A_\mathcal{O}$.
Hence $A$ is in $\mathfrak{a}_\mathfrak{h}$.
\end{proof}

\section{Algorithm for classification}\label{section:classifications}

Let $(\mathfrak{g},\mathfrak{h})$ be a semisimple symmetric pair (see Setting \ref{setting:main}).
In this section, we describe an algorithm to check whether or not $(\mathfrak{g},\mathfrak{h})$ satisfies the condition \eqref{item:main_hyperbolic} in Theorem \ref{thm:main}, which coincides with the condition \eqref{item:intro_anti_hyp} in Theorem \ref{thm:intro}. 
More 
precisely, we give an algorithm to classify complex antipodal hyperbolic orbits $\mathcal{O}^{G_\mathbb{C}}_{\hyp}$ in $\mathfrak{g}_\mathbb{C}$ such that $\mathcal{O}^{G_\mathbb{C}}_{\hyp} \cap \mathfrak{g} \neq \emptyset$ and $\mathcal{O}^{G_\mathbb{C}}_{\hyp} \cap \mathfrak{g}^c = \emptyset$.

Recall that for any complex semisimple Lie algebra $\mathfrak{g}_\mathbb{C}$, 
we can determine the set of complex antipodal hyperbolic orbits in $\mathfrak{g}_\mathbb{C}$, which is denoted by $\mathcal{H}^a/{G_\mathbb{C}}$, as $\iota$-invariant weighted Dynkin diagrams by Theorem \ref{thm:Bhyp_to_iota-inv}. 
Further, for any real form $\mathfrak{g}$ of $\mathfrak{g}_\mathbb{C}$, 
we can classify complex antipodal hyperbolic orbits in $\mathfrak{g}_\mathbb{C}$ meeting $\mathfrak{g}$ by using the Satake diagram of $\mathfrak{g}$ (see Section \ref{subsection:antipodal_real}). 

For a semisimple symmetric pair $(\mathfrak{g},\mathfrak{h})$, 
we can specify another real form $\mathfrak{g}^c$ of $\mathfrak{g}_\mathbb{C}$ (see \eqref{eq:c-dual} in Section \ref{section:main_result} for the notation) by the list of \cite[Section 1]{Oshima-Sekiguchi84},
since the symmetric pair $(\mathfrak{g}^c,\mathfrak{h})$ is same as $(\mathfrak{g},\mathfrak{h})^{\text{ada}}$. 
The Satake diagram of the real form $\mathfrak{g}$ [resp. $\mathfrak{g}^c$] of $\mathfrak{g}_\mathbb{C}$ can be found in \cite{Araki62} or \cite[Chapter X, Section 6]{Helgason01book}.
Therefore, we can classify the set of complex antipodal hyperbolic orbits in $\mathfrak{g}_\mathbb{C}$ meeting $\mathfrak{g}$ [resp. $\mathfrak{g}^c$], which is denoted by $\mathcal{H}^a_\mathfrak{g}/{G_\mathbb{C}}$ [resp. $\mathcal{H}^a_{\mathfrak{g}^c}/{G_\mathbb{C}}$].
This provides an algorithm to check whether the condition \eqref{item:main_hyperbolic} in Theorem \ref{thm:main} holds or not on $(\mathfrak{g},\mathfrak{h})$.

Here, we give examples for the cases where $(\mathfrak{g},\mathfrak{h}) = (\mathfrak{su}(4,2),\mathfrak{sp}(2,1))$ or $(\mathfrak{su}^*(6),\mathfrak{sp}(2,1))$.
\begin{example}\label{ex:su42_sp21}
Let $(\mathfrak{g},\mathfrak{h}) = (\mathfrak{su}(4,2),\mathfrak{sp}(2,1))$.
Then $\mathfrak{g}_\mathbb{C} = \mathfrak{sl}(6,\mathbb{C})$ and $\mathfrak{g}^c = \mathfrak{su}^*(6)$.
We shall determine both $\mathcal{H}^a_\mathfrak{g}/{G_\mathbb{C}}$ and $\mathcal{H}^a_{\mathfrak{g}^c}/{G_\mathbb{C}}$, 
and prove that $(\mathfrak{g},\mathfrak{h})$ satisfies the condition \eqref{item:main_hyperbolic} in Theorem $\ref{thm:main}$.

The involutive endomorphism $\iota$ on weighted Dynkin diagrams of $\mathfrak{sl}(6,\mathbb{C})$ $($see Section $\ref{subsection:complex_Bhyp}$ for the notation$)$ is given by
\[
\begin{xy}
	*++!D{a} *\cir<2pt>{}        ="A",
	(10,0) *++!D{b} *\cir<2pt>{} ="B",
	(20,0) *++!D{c} *\cir<2pt>{} ="C",
	(30,0) *++!D{d} *\cir<2pt>{} ="D",
	(40,0) *++!D{e} *\cir<2pt>{} ="E",
	\ar@{-} "A";"B"
	\ar@{-} "B";"C"
	\ar@{-} "C";"D"
	\ar@{-} "D";"E"
\end{xy}
\mapsto 
\begin{xy}
	*++!D{e} *\cir<2pt>{}        ="A",
	(10,0) *++!D{d} *\cir<2pt>{} ="B",
	(20,0) *++!D{c} *\cir<2pt>{} ="C",
	(30,0) *++!D{b} *\cir<2pt>{} ="D",
	(40,0) *++!D{a} *\cir<2pt>{} ="E",
	\ar@{-} "A";"B"
	\ar@{-} "B";"C"
	\ar@{-} "C";"D"
	\ar@{-} "D";"E"
\end{xy}.
\]
Thus, by Theorem $\ref{thm:Bhyp_to_iota-inv}$, we have the bijection below$:$
\begin{align*}
\mathcal{H}^a/{G_\mathbb{C}} 
\bijarrow
\left\{
\begin{xy}
	*++!D{a} *\cir<2pt>{}        ="A",
	(10,0) *++!D{b} *\cir<2pt>{} ="B",
	(20,0) *++!D{c} *\cir<2pt>{} ="C",
	(30,0) *++!D{b} *\cir<2pt>{} ="D",
	(40,0) *++!D{a} *\cir<2pt>{} ="E",
	\ar@{-} "A";"B"
	\ar@{-} "B";"C"
	\ar@{-} "C";"D"
	\ar@{-} "D";"E"
\end{xy}
\mid 
a,b,c \in \mathbb{R}_{\geq 0}
\right\}.
\end{align*}
Here are the Satake diagrams of $\mathfrak{g} = \mathfrak{su}(4,2)$ and $\mathfrak{g}^c = \mathfrak{su}^*(6)$$:$\[
S_{\mathfrak{su}(4,2)} :
\begin{xy}
	*\cir<2pt>{}        ="A",
	(10,0) *\cir<2pt>{} ="B",
	(20,0) *{\bullet} ="C",
	(30,0) *\cir<2pt>{} ="D",
	(40,0) *\cir<2pt>{} ="E",
	\ar@{-} "A";"B"
	\ar@{-} "B";"C"
	\ar@{-} "C";"D"
	\ar@{-} "D";"E"
	
	\ar@(ru,lu) @<1mm> @{<->} "A" ; "E"
	\ar@/^4mm/ @<1mm> @{<->} "B" ; "D"	
\end{xy}\ ,\ 
S_{\mathfrak{su}^*(6)} :
\begin{xy}
	*{\bullet}        ="A",
	(10,0) *\cir<2pt>{} ="B",
	(20,0) *{\bullet} ="C",
	(30,0) *\cir<2pt>{} ="D",
	(40,0) *{\bullet} ="E",
	\ar@{-} "A";"B"
	\ar@{-} "B";"C"
	\ar@{-} "C";"D"
	\ar@{-} "D";"E"
\end{xy}.
\]
Thus, by Theorem $\ref{thm:hyp_match}$, 
we obtain the following one-to-one correspondences$:$
\begin{align*}
\mathcal{H}^a_{\mathfrak{g}}/{G_\mathbb{C}} &\bijarrow \left\{
\begin{xy}
	*++!D{a} *\cir<2pt>{}        ="A",
	(10,0) *++!D{b} *\cir<2pt>{} ="B",
	(20,0) *++!D{0} *\cir<2pt>{} ="C",
	(30,0) *++!D{b} *\cir<2pt>{} ="D",
	(40,0) *++!D{a} *\cir<2pt>{} ="E",
	\ar@{-} "A";"B"
	\ar@{-} "B";"C"
	\ar@{-} "C";"D"
	\ar@{-} "D";"E"
\end{xy}
\mid 
a,b \in \mathbb{R}_{\geq 0}
\right\}, \\
\mathcal{H}^a_{\mathfrak{g}^c}/{G_\mathbb{C}} &\bijarrow \left\{
\begin{xy}
	*++!D{0} *\cir<2pt>{}        ="A",
	(10,0) *++!D{b} *\cir<2pt>{} ="B",
	(20,0) *++!D{0} *\cir<2pt>{} ="C",
	(30,0) *++!D{b} *\cir<2pt>{} ="D",
	(40,0) *++!D{0} *\cir<2pt>{} ="E",
	\ar@{-} "A";"B"
	\ar@{-} "B";"C"
	\ar@{-} "C";"D"
	\ar@{-} "D";"E"
\end{xy}
\mid 
b \in \mathbb{R}_{\geq 0}
\right\}.
\end{align*}
Therefore, the condition $\eqref{item:main_hyperbolic}$ in Theorem $\ref{thm:main}$ holds on the symmetric pair $(\mathfrak{su}(4,2),\mathfrak{sp}(2,1))$. 
\end{example}

\begin{example}
Let $(\mathfrak{g},\mathfrak{h}) = (\mathfrak{su}^*(6), \mathfrak{sp}(2,1))$.
Then $\mathfrak{g}_\mathbb{C} = \mathfrak{sl}(6,\mathbb{C})$ and $\mathfrak{g}^c = \mathfrak{su}(4,2)$.
Thus, by the argument in Example $\ref{ex:su42_sp21}$, 
we have
\begin{align*}
\mathcal{H}^a_{\mathfrak{g}}/{G_\mathbb{C}} &\bijarrow \left\{
\begin{xy}
	*++!D{0} *\cir<2pt>{}        ="A",
	(10,0) *++!D{b} *\cir<2pt>{} ="B",
	(20,0) *++!D{0} *\cir<2pt>{} ="C",
	(30,0) *++!D{b} *\cir<2pt>{} ="D",
	(40,0) *++!D{0} *\cir<2pt>{} ="E",
	\ar@{-} "A";"B"
	\ar@{-} "B";"C"
	\ar@{-} "C";"D"
	\ar@{-} "D";"E"
\end{xy}
\mid 
b \in \mathbb{R}_{\geq 0}
\right\}, \\
\mathcal{H}^a_{\mathfrak{g}^c}/{G_\mathbb{C}} &\bijarrow \left\{
\begin{xy}
	*++!D{a} *\cir<2pt>{}        ="A",
	(10,0) *++!D{b} *\cir<2pt>{} ="B",
	(20,0) *++!D{0} *\cir<2pt>{} ="C",
	(30,0) *++!D{b} *\cir<2pt>{} ="D",
	(40,0) *++!D{a} *\cir<2pt>{} ="E",
	\ar@{-} "A";"B"
	\ar@{-} "B";"C"
	\ar@{-} "C";"D"
	\ar@{-} "D";"E"
\end{xy}
\mid 
a,b \in \mathbb{R}_{\geq 0}
\right\}.
\end{align*}
Therefore, the condition \eqref{item:main_hyperbolic} in Theorem $\ref{thm:main}$ does not hold on the symmetric pair $(\mathfrak{su}^*(6), \mathfrak{sp}(2,1))$.
However, if we take a complex hyperbolic orbit $\mathcal{O}'$ in $\mathfrak{sl}(6,\mathbb{C})$ corresponding to 
\[
\begin{xy}
	*++!D{0} *\cir<2pt>{}        ="A",
	(10,0) *++!D{b} *\cir<2pt>{} ="B",
	(20,0) *++!D{0} *\cir<2pt>{} ="C",
	(30,0) *++!D{b'} *\cir<2pt>{} ="D",
	(40,0) *++!D{0} *\cir<2pt>{} ="E",
	\ar@{-} "A";"B"
	\ar@{-} "B";"C"
	\ar@{-} "C";"D"
	\ar@{-} "D";"E"
\end{xy} 
\quad (\text{for some }b,b' \in \mathbb{R}_{\geq 0},\ b \neq b'),
\]
then $\mathcal{O}'$ meets $\mathfrak{g}$ but does not meet $\mathfrak{g}^c$.
Note that $\mathcal{O}'$ is not antipodal.
Thus the condition $\eqref{item:C-M_cpx_hyp_orbits}$ in Fact $\ref{fact:C-M}$ holds on the symmetric pair $(\mathfrak{su}^*(6), \mathfrak{sp}(2,1))$.
In particular, $\rank_\mathbb{R} \mathfrak{g} > \rank_\mathbb{R} \mathfrak{h}$.
\end{example}

Combining our algorithm with Berger's classification \cite{Berger57}, 
we obtain Table \ref{table:non_sl2} in Section \ref{section:main_result}.
Concerning this, if $\mathfrak{g}_\mathbb{C}$ has no simple factor of type $A_n$, $D_{2k+1}$ or $E_6$ ($n \geq 2$, $k \geq 2$), 
then the symmetric pair $(\mathfrak{g},\mathfrak{h})$ satisfies the condition \eqref{item:main_hyperbolic} in Theorem \ref{thm:main} if and only if $\rank_\mathbb{R} \mathfrak{g} > \rank_\mathbb{R} \mathfrak{h}$ (see Corollary \ref{cor:Bhyp_in_without_ADE} and Fact \ref{fact:C-M}).
Thus we need only consider the cases where $\mathfrak{g}_\mathbb{C}$ is of type $A_n$, $D_{2k+1}$ or $E_6$.
 
We also remark that for a given semisimple symmetric pair $(\mathfrak{g},\mathfrak{h})$, by using the Dynkin--Kostant classification \cite{Dynkin52eng} and Theorem \ref{thm:nilp_match}, 
we can check whether the condition \eqref{item:main_nilpotent} in Theorem \ref{thm:main} holds or not on $(\mathfrak{g},\mathfrak{h})$, directly (see also Section \ref{section:refine_sl2proper_nilp}).

\section{Proper actions of $SL(2,\mathbb{R})$ and real nilpotent orbits}\label{section:refine_sl2proper_nilp}

In this section, we describe a refinement of the equivalence \eqref{item:main_sl2-proper} $\Leftrightarrow$ \eqref{item:main_nilpotent} in Theorem \ref{thm:main}, which provides an algorithm to classify proper $SL(2,\mathbb{R})$-actions on a given semisimple symmetric space $G/H$.

Let $G$ be a connected linear semisimple Lie group and write $\mathfrak{g}$ for its Lie algebra.
By the Jacobson-Morozov theorem and Lemma \ref{lem:alg_hom_to_group_hom}, 
we have a one-to-one correspondence between 
Lie group homomorphisms $\Phi:SL(2,\mathbb{R}) \rightarrow G$ up to inner automorphisms of $G$ and real nilpotent orbits in $\mathfrak{g}$.
We denote by $\mathcal{O}^{G}_\Phi$ the real nilpotent orbit corresponding to $\Phi: SL(2,\mathbb{R}) \rightarrow G$.
Then, by combining Proposition \ref{prop:Sl_2_proper_and_sl2_triple}, Proposition \ref{prop:keyprop} with Lemma \ref{lem:A_meets_X_meets},
we obtain the next theorem:

\begin{thm}\label{thm:sl2_and_real_nilpotent}
In Setting $\ref{setting:main}$, the following conditions on a Lie group homomorphism $\Phi : SL(2,\mathbb{R}) \rightarrow G$ are equivalent$:$
\begin{enumerate}
\item $SL(2,\mathbb{R})$ acts on $G/H$ properly via $\Phi$.
\item The complex nilpotent orbit $\Ad(G_\mathbb{C}) \cdot  \mathcal{O}^{G}_\Phi$ in $\mathfrak{g}_\mathbb{C}$ does not meet $\mathfrak{g}^c$, 
where $\mathfrak{g}^c$ is the $c$-dual of the symmetric pair $(\mathfrak{g},\mathfrak{h})$ $($see $\eqref{eq:c-dual}$ after Setting $\ref{setting:main}$$)$.
\end{enumerate}
In particular, we have the one-to-one correspondence
\begin{multline*}
\{\, \Phi:SL(2,\mathbb{R}) \rightarrow G \mid \text{$SL(2,\mathbb{R})$ acts on $G/H$ properly via $\Phi$} \,\}/{G}
\\ \bijarrow
\{\, \text{Real nilpotent orbits $\mathcal{O}^{G}$ in $\mathfrak{g}$} \mid \text{$(\Ad(G_\mathbb{C}) \cdot \mathcal{O}^{G}) \cap \mathfrak{g}^c =\emptyset$} \,\}.
\end{multline*}
\end{thm}

Here is an example concerning Theorem \ref{thm:sl2_and_real_nilpotent}:

\begin{example}
Let $(G,H) = (SU(4,2),Sp(2,1))$.
Then we have 
$(\mathfrak{g}_\mathbb{C},\mathfrak{g},\mathfrak{g}^c) = (\mathfrak{sl}(6,\mathbb{C}),\mathfrak{su}(4,2),\mathfrak{su}^*(6))$.
Let us classify the following set$:$
\begin{multline}
\{\, \text{Real nilpotent orbits $\mathcal{O}^G$ in $\mathfrak{su}(4,2)$} \\ 
\mid \text{ the complexifications of $\mathcal{O}^G$ do not meet $\mathfrak{su}^*(6)$}  \,\}
\label{eq:proper_orbit}
\end{multline}

Recall that complex nilpotent orbits in $\mathfrak{sl}(6,\mathbb{C})$ are parameterized by 
partitions of $6$ and these weighted Dynkin diagrams are listed in Example $\ref{ex:list_of_nilp_in_sl6}$. 
By Theorem $\ref{thm:nilp_match}$, 
we can classify the complex nilpotent orbits in $\mathfrak{sl}(6,\mathbb{C})$ that meet $\mathfrak{su}(4,2)$ but not $\mathfrak{su}^*(6)$, by using these Satake diagrams $($see Example $\ref{ex:su42_sp21}$ for Satake diagrams of $\mathfrak{su}(4,2)$ and $\mathfrak{su}^*(6)$$)$, as follows$:$
\begin{multline*}
\{\, \text{Complex nilpotent orbits } \mathcal{O}^{G_\mathbb{C}} \text{ in } \mathfrak{sl}(6,\mathbb{C}) \\ \mid \mathcal{O}^{G_\mathbb{C}} \cap \mathfrak{su}(4,2) \neq \emptyset \text{ and } \mathcal{O}^{G_\mathbb{C}} \cap \mathfrak{su}^*(6) = \emptyset \,\} \\
\bijarrow
\{\, [5,1],\ [4,1^2],\ [3,2,1],\ [3,1^3],\ [2,1^4] \,\}.
\end{multline*}

It is known that real nilpotent orbits in $\mathfrak{su}(4,2)$ are parameterized by signed Young diagrams of signature $(4,2)$,
and 
the shape of the signed Young diagram corresponding to a real nilpotent orbit $\mathcal{O}^G$ in $\mathfrak{su}(4,2)$ 
is 
the partition corresponding to the complexification of $\mathcal{O}^G$ 
$($see \cite[Theorem 9.3.3 and a remark after Theorem 9.3.5]{Collingwood-McGovern93} for more details$)$.
Therefore, we have a classification of \eqref{eq:proper_orbit} as follows$:$

\begin{table}[htbp]
\begin{center}
\begin{tabular}{ll} \toprule
Partition & Signed Young diagram of signeture $(4,2)$ \\ \midrule
$[5,1]$ & $\young(+-+-+,+)$ \\[10pt] \hline 
$[4,1^2]$ & $\young(+-+-,+,+)$, $\young(-+-+,+,+)$ \\[15pt] \hline 
$[3,2,1]$ & $\young(+-+,+-,+)$, $\young(+-+,-+,+)$ \\[15pt] \hline 
$[3,1^3]$ & $\young(+-+,+,+,-)$, $\young(-+-,+,+,+)$ \\[21pt] \hline 
$[2,1^4]$ & $\young(+-,+,+,+,-)$, $\young(-+,+,+,+,-)$ \\ \bottomrule
\end{tabular}
\legend{Classification of \eqref{eq:proper_orbit}}
\end{center}
\end{table}

In particular, by Theorem $\ref{thm:sl2_and_real_nilpotent}$, there are nine kinds of Lie group homomorphisms $\Phi : SL(2,\mathbb{R}) \rightarrow SU(4,2)$ $($up to inner automorphisms of $SU(4,2)$$)$ for which the $SL(2,\mathbb{R})$-actions on $SU(4,2)/Sp(2,1)$ via $\Phi$ are proper.
\end{example}

\appendix

\section{Classification}\label{section:appendix}

Here is a complete list of symmetric pairs $(\mathfrak{g},\mathfrak{h})$ with the following property$:$
\begin{multline}
\text{$\mathfrak{g}$ is simple, $(\mathfrak{g},\mathfrak{h})$ is a symmetric pair} \\ 
\text{satisfying one of (therefore, all of) the conditions in Theorem \ref{thm:main}.} \label{condition:classification_sl2}
\end{multline}

\setcounter{table}{2}

\begin{center}
\begin{longtable}{ll} \toprule
$\mathfrak{g}$ & $\mathfrak{h}$  \\ \midrule
$\mathfrak{sl}(2k,\mathbb{R})$ & $\mathfrak{sl}(k,\mathbb{C}) \oplus \mathfrak{so}(2)$ \\ \hline
$\mathfrak{sl}(n,\mathbb{R})$ & $\mathfrak{so}(n-i,i)$  \\ 
& $(2i < n)$ \\ \hline
$\mathfrak{su}^*(2k)$ & $\mathfrak{sp}(k-i,i)$ \\
& $(2i < k-1)$  \\ \hline
$\mathfrak{su}(2p,2q)$ & $\mathfrak{sp}(p,q)$ \\ \hline
$\mathfrak{su}(2m-1,2m-1)$ & $\mathfrak{so}^*(4m-2)$ \\ \hline 
$\mathfrak{su}(p,q)$ & $\mathfrak{su}(i,j) \oplus \mathfrak{su}(p-i,q-j) \oplus \mathfrak{so}(2)$ \\ 
&$(\min \{ p,q \} > \min \{ i,j \} + \min \{ p-i,q-j \} )$  \\ \hline
$\mathfrak{so}(p,q)$ & $\mathfrak{so}(i,j) \oplus \mathfrak{so}(p-i,q-j)$ \\ 
$(p+q \text{ is odd})$ & ($\min \{ p,q \} > \min \{ i,j \} + \min \{ p-i,q-j \}$) \\ \hline 
$\mathfrak{sp}(n,\mathbb{R})$ & $\mathfrak{su}(n-i,i) \oplus \mathfrak{so}(2)$ \\ \hline
$\mathfrak{sp}(2k,\mathbb{R})$ & $\mathfrak{sp}(k,\mathbb{C})$ \\ \hline 
$\mathfrak{sp}(p,q)$ & $\mathfrak{sp}(i,j) \oplus \mathfrak{sp}(p-i,q-j)$ \\ 
& $(\min \{ p,q \} > \min \{ i,j \} + \min \{ p-i,q-j \} )$  \\ \hline
$\mathfrak{so}(p,q)$ & $\mathfrak{so}(i,j) \oplus \mathfrak{so}(p-i,q-j)$ \\ 
$(p+q \text{ is even})$ & ($\min \{ p,q \} > \min \{ i,j \} + \min \{ p-i,q-j \}$, \\
 & unless $p=q=2m+1$ and $|i-j| =1$)  \\ \hline
$\mathfrak{so}(2p,2q)$ & $\mathfrak{su}(p,q) \oplus \mathfrak{so}(2)$ \\ \hline
$\mathfrak{so}^*(2k)$ & $\mathfrak{su}(k-i,i) \oplus \mathfrak{so}(2)$ \\ 
& $(2i < k-1)$  \\ \hline
$\mathfrak{so}(k,k)$ & $\mathfrak{so}(2k,\mathbb{C}) \oplus \mathfrak{so}(2)$ \\ \hline
$\mathfrak{so}^*(4m)$ & $\mathfrak{so}^*(4m-4i+2) \oplus \mathfrak{so}^*(4i-2)$ \\ \hline
$\mathfrak{e}_{6(6)}$ & $\mathfrak{sp}(2,2)$ \\ \hline
$\mathfrak{e}_{6(6)}$ & $\mathfrak{su}^*(6) \oplus \mathfrak{su}(2)$ \\ \hline
$\mathfrak{e}_{6(2)}$ & $\mathfrak{so}^*(10) \oplus \mathfrak{so}(2)$ \\ \hline
$\mathfrak{e}_{6(2)}$ & $\mathfrak{su}(4,2) \oplus \mathfrak{su}(2)$ \\ \hline
$\mathfrak{e}_{6(2)}$ & $\mathfrak{sp}(3,1)$ \\ \hline
$\mathfrak{e}_{6(-14)}$ & $\mathfrak{f}_{4(-20)}$ \\ \hline
$\mathfrak{e}_{7(7)}$ & $\mathfrak{e}_{6(2)} \oplus \mathfrak{so}(2)$ \\ \hline
$\mathfrak{e}_{7(7)}$ & $\mathfrak{su}(4,4)$ \\ \hline
$\mathfrak{e}_{7(7)}$ & $\mathfrak{so}^*(12) \oplus \mathfrak{su}(2)$ \\ \hline
$\mathfrak{e}_{7(7)}$ & $\mathfrak{su}^*(8)$ \\ \hline
$\mathfrak{e}_{7(-5)}$ & $\mathfrak{e}_{6(-14)} \oplus \mathfrak{so}(2)$ \\ \hline
$\mathfrak{e}_{7(-5)}$ & $\mathfrak{su}(6,2)$ \\ \hline
$\mathfrak{e}_{7(-25)}$ & $\mathfrak{e}_{6(-14)} \oplus \mathfrak{so}(2)$ \\ \hline
$\mathfrak{e}_{7(-25)}$ & $\mathfrak{su}(6,2)$ \\ \hline
$\mathfrak{e}_{8(8)}$ & $\mathfrak{e}_{7(-5)} \oplus \mathfrak{su}(2)$ \\ \hline
$\mathfrak{e}_{8(8)}$ & $\mathfrak{so}^*(16)$ \\ \hline
$\mathfrak{f}_{4(4)}$ & $\mathfrak{sp}(2,1) \oplus \mathfrak{su}(2)$ \\ \hline 
$\mathfrak{sl}(2k,\mathbb{C})$ & $\mathfrak{su}^*(2k)$ \\ \hline 
$\mathfrak{sl}(n,\mathbb{C})$ & $\mathfrak{su}(n-i,i)$ \\ 
& ($2i < n$) \\ \hline 
$\mathfrak{so}(2k+1,\mathbb{C})$ & $\mathfrak{so}(2k+1-i,i)$ \\
& ($i < k$) \\ \hline
$\mathfrak{sp}(n,\mathbb{C})$ & $\mathfrak{sp}(n-i,i)$ \\ \hline
$\mathfrak{so}(2k,\mathbb{C})$ & $\mathfrak{so}(2k-i,i)$ \\
& ($i < k$ unless $k=i+1=2m+1$) \\ \hline
$\mathfrak{so}(4m,\mathbb{C})$ & $\mathfrak{so}(4m-2i+1,\mathbb{C}) \oplus \mathfrak{so}(2i-1,\mathbb{C})$ \\ \hline
$\mathfrak{so}(2k,\mathbb{C})$ & $\mathfrak{so}^*(2k)$ \\ \hline
$\mathfrak{e}_{6,\mathbb{C}}$ & $\mathfrak{e}_{6(-14)}$ \\ \hline
$\mathfrak{e}_{6,\mathbb{C}}$ & $\mathfrak{e}_{6(-26)}$ \\ \hline
$\mathfrak{e}_{7,\mathbb{C}}$ & $\mathfrak{e}_{7(-5)}$ \\ \hline
$\mathfrak{e}_{7,\mathbb{C}}$ & $\mathfrak{e}_{7(-25)}$ \\ \hline
$\mathfrak{e}_{8,\mathbb{C}}$ & $\mathfrak{e}_{8(-24)}$ \\ \hline
$\mathfrak{f}_{4,\mathbb{C}}$ & $\mathfrak{f}_{4(-20)}$ \\ 
\bottomrule
\caption{Classification of $(\mathfrak{g},\mathfrak{h})$ satisfying \eqref{condition:classification_sl2}}
\label{table:sl2-proper}
\end{longtable}
\end{center}

Here, $k \geq 1$, $m \geq 1$, $n \geq 2$, $p, q \geq 1$ and $i, j \geq 0$.
Note that $\mathfrak{so}(p,q)$ is simple if and only if $p+q \geq 3$ with $(p,q) \neq (2,2)$, and $\mathfrak{so}(2k,\mathbb{C})$ is simple if and only if $k \geq 3$.

\section{The Calabi--Markus phenomenon and hyperbolic orbits}\label{subsection:proof_of_fact}

Here is a proof of the equivalence among \eqref{item:C-M_split_rank}, \eqref{item:C-M_real_hyp_orbits} and \eqref{item:C-M_cpx_hyp_orbits} in Fact \ref{fact:C-M}:

\begin{proof}[Proof of $\eqref{item:C-M_split_rank}$ $\Leftrightarrow$ $\eqref{item:C-M_real_hyp_orbits}$ $\Leftrightarrow$ $\eqref{item:C-M_cpx_hyp_orbits}$ in Fact $\ref{fact:C-M}$]
We take $\mathfrak{a}$ and $\mathfrak{a}_\mathfrak{h}$ in Section $\ref{subsection:Kobayashi's_criterion}$. 
The condition \eqref{item:C-M_split_rank} means that $\mathfrak{a} \neq W(\mathfrak{g},\mathfrak{a}) \cdot \mathfrak{a}_\mathfrak{h}$.
By Fact \ref{fact:hyp_inter_a} and Lemma \ref{lem:hyperbolic_in_reductive}, 
we have a bijection between the following two sets$:$
\begin{itemize}
\item The set of $W(\mathfrak{g},\mathfrak{a})$-orbits in $\mathfrak{a}$ that do not meet $\mathfrak{a}_\mathfrak{h}$.
\item The set of real hyperbolic orbits in $\mathfrak{g}$ that do not meet $\mathfrak{h}$.
\end{itemize}
Then the equivalence \eqref{item:C-M_split_rank} $\Leftrightarrow$ \eqref{item:C-M_real_hyp_orbits} holds.
Further, \eqref{item:C-M_real_hyp_orbits} $\Leftrightarrow$ \eqref{item:C-M_cpx_hyp_orbits} follows from Proposition \ref{prop:Hyp_in_gC_in_g} \eqref{item:Hyp_c_r:hyp} and Propositon \ref{prop:keyprop}.
\end{proof}

\def\cprime{$'$}
\providecommand{\bysame}{\leavevmode\hbox to3em{\hrulefill}\thinspace}
\providecommand{\MR}{\relax\ifhmode\unskip\space\fi MR }
\providecommand{\MRhref}[2]{%
  \href{http://www.ams.org/mathscinet-getitem?mr=#1}{#2}
}
\providecommand{\href}[2]{#2}


\begin{thebibliography}{10}

\bibitem{Abels-Margulis-Soifer11}
Herbert Abels, Grigori~A. Margulis, and Grigori~A. Soifer,
  \emph{The
  linear part of an affine group acting properly discontinuously and leaving a
  quadratic form invariant}, Geom. Dedicata \textbf{153} (2011), \href{http://www.springerlink.com/content/9630229n72775n86/?MUD=MP}{1--46},
  MR2819661, Zbl 1228.22012.

\bibitem{Araki62}
Sh{\^o}r{\^o} Araki, \emph{On root systems and an infinitesimal classification
  of irreducible symmetric spaces}, J. Math. Osaka City Univ. \textbf{13}
  (1962), 1--34, MR0153782, Zbl 0123.03002.

\bibitem{Auslander64}
Louis Auslander,
  \emph{The
  structure of complete locally affine manifolds}, Topology \textbf{3} (1964),
  \href{http://www.sciencedirect.com/science/article/pii/0040938364900126}{131--139}, MR0161255, Zbl 0136.43102.

\bibitem{Bala-Carter76}
Pawan Bala and Roger~W. Carter,
  \emph{Classes
  of unipotent elements in simple algebraic groups. {I},
  {II}},
  Math. Proc. Cambridge Philos. Soc. \textbf{79} (1976), \href{http://journals.cambridge.org/action/displayAbstract?fromPage=online&aid=2076264}{401--425}, MR0417306,
  Zbl 0364.22006 \textit{bid} {\bf 80} (1976), \href{http://journals.cambridge.org/action/displayAbstract?fromPage=online&aid=2076512}{1--17}, MR0417307, Zbl
  0364.22007.

\bibitem{Benoist96}
Yves Benoist,
  \emph{Actions
  propres sur les espaces homog\`enes r\'eductifs}, Ann. of Math. (2)
  \textbf{144} (1996), \href{http://www.jstor.org/stable/2118594?origin=crossref&}{315--347}, MR1418901, Zbl 0868.22013.

\bibitem{Berger57}
Marcel Berger,
  \emph{Les espaces
  sym\'etriques noncompacts}, Ann. Sci. \'Ecole Norm. Sup. (3) \textbf{74}
  (1957), \href{http://www.numdam.org/item?id=ASENS_1957_3_74_2_85_0}{85--177}, MR0104763, Zbl 0093.35602.

\bibitem{Calabi-Murkus62}
Eugenio Calabi and Lawrence Markus,
  \emph{Relativistic
  space forms}, Ann. of Math. (2) \textbf{75} (1962), \href{http://www.jstor.org/stable/1970419?origin=crossref}{63--76}, MR0133789, Zbl
  0101.21804.

\bibitem{Collingwood-McGovern93}
David~H. Collingwood and William~M. McGovern, \emph{Nilpotent orbits in
  semisimple {L}ie algebras}, Van Nostrand Reinhold Mathematics Series, Van
  Nostrand Reinhold Co., New York, 1993, MR1251060, Zbl 0972.17008.

\bibitem{Djokovic82}
Dragomir~{\v{Z}}. Djokovi{\'c},
  \emph{Classification
  of {${\bf Z}$}-graded real semisimple {L}ie algebras}, J. Algebra
  \textbf{76} (1982), \href{http://www.sciencedirect.com/science/article/pii/0021869382902204}{367--382}, MR0661861, Zbl 0486.17006.

\bibitem{Dynkin52eng}
E.~B. Dynkin, \emph{Semisimple subalgebras of semisimple {L}ie algebras}, Amer.
  Math. Soc. Transl. \textbf{6} (1957), 111--244, Zbl 0077.03404 

\bibitem{Goldman-Kamishima84}
William~M. Goldman and Yoshinobu Kamishima,
  \emph{The
  fundamental group of a compact flat {L}orentz space form is virtually
  polycyclic}, J. Differential Geom. \textbf{19} (1984), \href{http://projecteuclid.org/DPubS?service=UI&version=1.0&verb=Display&handle=euclid.jdg/1214438430}{233--240}, MR0739789,
  Zbl 0546.53039.

\bibitem{Helgason01book}
Sigurdur Helgason, \emph{Differential geometry, {L}ie groups, and symmetric
  spaces}, Graduate Studies in Mathematics, vol.~34, American Mathematical
  Society, 2001, Corrected reprint of the 1978 original, MR1834454, Zbl
  0993.53002.

\bibitem{Kassel09}
Fanny Kassel,
  \emph{Deformation of
  proper actions on reductive homogeneous spaces}, Math. Ann. \textbf{353}
  (2012), \href{http://dx.doi.org/10.1007/s00208-011-0672-1}{599--632}, MR2915550, Zbl 1248.22005.

\bibitem{Kassel-Kobayashi11}
Fanny Kassel and Toshiyuki Kobayashi,
  \emph{Stable
  spectrum for pseudo-{R}iemannian locally symmetric spaces}, C. R. Math.
  Acad. Sci. Paris \textbf{349} (2011), \href{http://www.sciencedirect.com/science/article/pii/S1631073X10003523}{29--33}, MR2755691, Zbl 1208.22013.

\bibitem{Kobayashi89}
Toshiyuki Kobayashi,
  \emph{Proper
  action on a homogeneous space of reductive type}, Math. Ann. \textbf{285}
  (1989), \href{http://www.springerlink.com/content/k7p01xx50810h708/?MUD=MP}{249--263}, MR1016093, Zbl 0662.22008.

\bibitem{Kobayashi92b}
\bysame, \emph{Discontinuous groups acting on homogeneous spaces of reductive
  type}, Representation theory of {L}ie groups and {L}ie algebras
  ({F}uji-{K}awaguchiko, 1990), World Sci. Publ., River Edge, NJ, 1992,
  pp.~59--75, MR1190750, Zbl 1193.22010.

\bibitem{Kobayashi92a}
\bysame,
  \emph{A
  necessary condition for the existence of compact {C}lifford-{K}lein forms of
  homogeneous spaces of reductive type}, Duke Math. J. \textbf{67} (1992),
  \href{http://projecteuclid.org/DPubS?service=UI&version=1.0&verb=Display&handle=euclid.dmj/1077294544}{653--664}, MR1181319, Zbl 0799.53056.

\bibitem{Kobayashi93}
\bysame,
  \emph{On
  discontinuous groups acting on homogeneous spaces with noncompact isotropy
  subgroups}, J. Geom. Phys. \textbf{12} (1993), \href{http://www.sciencedirect.com/science/article/pii/0393044093900113}{133--144}, MR1231232, Zbl
  0815.57029.

\bibitem{Kobayashi97disc}
\bysame, \emph{Discontinuous groups and {C}lifford-{K}lein forms of
  pseudo-{R}iemannian homogeneous manifolds}, Algebraic and analytic methods in
  representation theory ({S}\o nderborg, 1994), Perspect. Math., 17, 1997,
  pp.~99--165, Academic Press, San Diego, CA, MR1415843, Zbl 0899.43005.

\bibitem{Kobayashi01}
\bysame, \emph{Discontinuous groups for non-{R}iemannian homogeneous spaces},
  Mathematics unlimited---2001 and beyond, Springer, Berlin, 2001,
  pp.~723--747, MR1852186, Zbl 1023.53031.

\bibitem{Kobayashi02}
\bysame,
  \emph{Introduction
  to actions of discrete groups on pseudo-{R}iemannian homogeneous manifolds},
  Acta Appl. Math. \textbf{73} (2002), \href{http://www.springerlink.com/content/w7624571663112n0/}{115--131}, 
The 2000 Twente Conference on Lie Groups (Enschede), MR1926496, Zbl 1026.57029.

\bibitem{Kobayashi09}
\bysame, \emph{On discontinuous group actions on non-{R}iemannian homogeneous
  spaces [translation of {S}\=ugaku {\bf 57} (2005), 267--281, \text{MR2163672}]},
  Sugaku Expositions \textbf{22} (2009), 1--19, translated by Miles Reid, MR2503485.

\bibitem{Kobayashi-Ono90Hirzebruch}
Toshiyuki Kobayashi and Kaoru Ono, \emph{Note on {H}irzebruch's proportionality
  principle}, J. Fac. Sci. Univ. Tokyo Sect. IA Math. \textbf{37} (1990),
  71--87, MR1049019, Zbl 0726.57019.

\bibitem{Kobayashi-Yoshino05}
Toshiyuki Kobayashi and Taro Yoshino, \emph{Compact {C}lifford-{K}lein forms of
  symmetric spaces---revisited}, Pure Appl. Math. Q. \textbf{1} (2005),
  591--663, MR2201328, Zbl 1145.22011.

\bibitem{Kostant59}
Bertram Kostant,
  \emph{The
  principal three-dimensional subgroup and the {B}etti numbers of a complex
  simple {L}ie group}, Amer. J. Math. \textbf{81} (1959), \href{http://www.jstor.org/stable/2372999?origin=crossref}{973--1032},
  MR0114875, Zbl 0099.25603.

\bibitem{Kra85}
Irwin Kra, \emph{On lifting {K}leinian groups to {${\rm SL}(2,{\bf C})$}},
  Differential geometry and complex analysis, Springer, Berlin, 1985,
  pp.~181--193, MR0780044, Zbl 0571.30037.

\bibitem{Labourie96}
Fran{\c{c}}ois Labourie, \emph{Quelques r\'esultats r\'ecents sur les espaces
  localement homog\`enes compacts}, Manifolds and geometry ({P}isa, 1993),
  Sympos. Math., XXXVI, Cambridge Univ. Press, 1996, pp.~267--283, MR1410076,
  Zbl 0861.53053.

\bibitem{Labourie-Zimmer95}
Fran{\c{c}}ois Labourie and Robert~J. Zimmer,
  \emph{On
  the non-existence of cocompact lattices for {${\rm SL}(n)/{\rm SL}(m)$}},
  Math. Res. Lett. \textbf{2} (1995), \href{http://www.mrlonline.org/mrl/1995-002-001/1995-002-001-007.html}{75--77}, MR1312978, Zbl 0852.22009.

\bibitem{Malcev50}
A.~I. Malcev, \emph{On semi-simple subgroups of {L}ie groups}, Amer. Math. Soc.
  Translation (1950), 43 pp, MR0037848, Zbl 0061.04701.

\bibitem{Margulis97}
Gregory Margulis,
  \emph{Existence
  of compact quotients of homogeneous spaces, measurably proper actions, and
  decay of matrix coefficients}, Bull. Soc. Math. France \textbf{125} (1997),
  \href{http://www.numdam.org/item?id=BSMF_1997__125_3_447_0}{447--456}, MR1605453, Zbl 0892.22009.

\bibitem{Margulis00}
\bysame, \emph{Problems and conjectures in rigidity theory}, Mathematics:
  frontiers and perspectives, Amer. Math. Soc., Providence, RI, 2000,
  pp.~161--174, MR1754775, Zbl 0952.22005.

\bibitem{Mostow55}
George~D. Mostow, \emph{Some new decomposition theorems for semi-simple
  groups}, Mem. Amer. Math. Soc. (1955), 31--54, MR0069829, Zbl 0064.25901.

\bibitem{Oh-Witte02}
Hee Oh and Dave Witte,
  \emph{Compact
  {C}lifford-{K}lein forms of homogeneous spaces of {${\rm SO}(2,n)$}}, Geom.
  Dedicata \textbf{89} (2002), \href{http://www.springerlink.com/content/hmhkkhhrn35cc0q2/}{25--57}, MR1890952, Zbl 1002.57074.

\bibitem{Okuda11_proc}
Takayuki Okuda,
  \emph{Proper
  actions of {${\rm SL}(2,\bold R)$} on semisimple symmetric spaces}, Proc.
  Japan Acad. Ser. A Math. Sci. \textbf{87} (2011), \href{http://projecteuclid.org/DPubS?service=UI&version=1.0&verb=Display&handle=euclid.pja/1299161393}{35--39}, MR2802605, Zbl
  1221.22020.

\bibitem{Oshima-Sekiguchi84}
Toshio {\=O}shima and Jir{\=o} Sekiguchi, \emph{The restricted root system of a
  semisimple symmetric pair}, Group representations and systems of differential
  equations ({T}okyo, 1982), Adv. Stud. Pure Math., vol.~4, 1984, pp.~433--497,
  MR0810638, Zbl 0577.17004.

\bibitem{Satake60}
Ichir{\^o} Satake,
  \emph{On
  representations and compactifications of symmetric {R}iemannian spaces},
  Ann. of Math. (2) \textbf{71} (1960), \href{http://www.jstor.org/stable/1969880?origin=crossref}{77--110}, MR0118775, Zbl 0094.34603.

\bibitem{Sekiguchi84}
Jir{\=o} Sekiguchi,
  \emph{The
  nilpotent subvariety of the vector space associated to a symmetric pair},
  Publ. Res. Inst. Math. Sci. \textbf{20} (1984), \href{http://www.ems-ph.org/journals/show_abstract.php?issn=0034-5318&vol=20&iss=1&rank=13}{155--212}, MR0736100, Zbl 0556.14022.

\bibitem{Sekiguchi87}
\bysame,
  \emph{Remarks
  on real nilpotent orbits of a symmetric pair}, J. Math. Soc. Japan
  \textbf{39} (1987), \href{http://projecteuclid.org/DPubS?service=UI&version=1.0&verb=Display&handle=euclid.jmsj/1230130942}{127--138}, MR0867991, Zbl 0627.22008.

\bibitem{Shalom00rigidity}
Yehuda Shalom, \emph{Rigidity, unitary
  representations of semisimple groups, and fundamental groups of manifolds
  with rank one transformation group}, Ann. of Math. (2) \textbf{152} (2000),
  \href{http://dx.doi.org/10.2307/2661380}{113--182}, MR1792293, Zbl 0970.22011.

\bibitem{Teduka08cpx_symm}
Katsuki Teduka,
  \emph{Proper
  actions of {${\rm SL}(2,{\bf C})$} on irreducible complex symmetric spaces},
  Proc. Japan Acad. Ser. A Math. Sci. \textbf{84} (2008), \href{http://projecteuclid.org/DPubS?service=UI&version=1.0&verb=Display&handle=euclid.pja/1216308251}{107--111}, MR2450061,
  Zbl 1156.22018.

\bibitem{Teduka08sln}
\bysame, \emph{Proper actions of {${\rm SL}(2,\Bbb R)$} on {${\rm SL}(n,\Bbb
  R)$}-homogeneous spaces}, J. Math. Sci. Univ. Tokyo \textbf{15} (2008),
  1--13, MR2450061, Zbl 1154.22025.

\bibitem{Tits66}
Jacques Tits, \emph{Classification of algebraic semisimple groups}, Algebraic
  {G}roups and {D}iscontinuous {S}ubgroups ({P}roc. {S}ympos. {P}ure {M}ath.,
  {B}oulder, {C}olo., 1965), Amer. Math. Soc., 1966, pp.~33--62, MR0224710, Zbl
  0238.20052.

\bibitem{Tomanov90}
George Tomanov,
  \emph{The
  virtual solvability of the fundamental group of a generalized {L}orentz space
  form}, J. Differential Geom. \textbf{32} (1990), \href{http://projecteuclid.org/DPubS?service=UI&version=1.0&verb=Display&handle=euclid.jdg/1214445319}{539--547}, MR1072918, Zbl
  0681.57027.

\bibitem{yosida38}
K{\^o}saku Yosida, \emph{A theorem concerning the semi-simple {L}ie groups},
  Tohoku Math. J \textbf{44} (1938), 81--84, Zbl 0018.29802.

\bibitem{Zimmer94}
Robert~J. Zimmer,
  \emph{Discrete
  groups and non-{R}iemannian homogeneous spaces}, J. Amer. Math. Soc.
  \textbf{7} (1994), \href{http://www.ams.org/journals/jams/1994-07-01/S0894-0347-1994-1207014-8/home.html}{159--168}, MR1207014, Zbl 0801.22009.

\end{thebibliography}
\end{document}